\documentclass{amsart}
\pdfoutput=1
\usepackage[all,pdf]{xy}
\usepackage{amsfonts,amssymb,amsmath,enumerate,verbatim,mathtools,tikz,bm,mathrsfs,tikz-cd,comment}
\usepackage{amsmath,amscd,amsfonts,amssymb,amsthm,graphicx,multirow,newlfont,exscale,xcolor,mathrsfs,stmaryrd, multicol,comment,kbordermatrix,extarrows}
\usepackage[utf8]{inputenc}
\usepackage[margin = 1 in]{geometry}
\usepackage{cancel,graphicx}
\usepackage[pdfencoding=auto, psdextra]{hyperref}

\hypersetup{
 colorlinks=true,
 linkcolor=blue,
 filecolor=magenta, 
 urlcolor=cyan,
}

\usepackage[capitalise]{cleveref}
\crefformat{equation}{(#2#1#3)}
\crefrangeformat{equation}{(#3#1#4--#5#2#6)}
\crefformat{enumi}{(#2#1#3)}
\crefrangeformat{enumi}{(#3#1#4--#5#2#6)}
\crefname{subsection}{Subsection}{Subsections}

\title[Differential operators of low order]{Resolutions of differential operators of low order for an isolated hypersurface singularity}

\author[R.~Diethorn]{Rachel N. Diethorn}
\address{Department of Mathematics,
 Oberlin College, Oberlin, OH 44074 U.S.A.}
\email{rdiethor@oberlin.edu}

\author[J.~Jeffries]{Jack Jeffries}
\address{Department of Mathematics,
University of Nebraska-Lincoln,
Lincoln, NE 68588 U.S.A.}
\email{jack.jeffries@unl.edu}

\author[C.~Miller]{Claudia Miller}
\address{Mathematics Department, Syracuse University, Syracuse, NY 13244 U.S.A.}
\email{clamille@syr.edu}

\author[N.~Packauskas]{Nicholas Packauskas}
\address{Department of Mathematics,
University of Kansas, Lawrence, KS 66045 U.S.A.}
\email{packauskas@ku.edu}

\author[J.~Pollitz]{Josh Pollitz}
\address{Mathematics Department, 
Syracuse University, 
Syracuse, NY 13244 U.S.A.}
\email{jhpollit@syr.edu}

\author[H.~Rahmati]{Hamidreza Rahmati}
\address{Mathematics Department, Syracuse University, Syracuse, NY 13244 U.S.A.}
\email{hrahmati@syr.edu}

\author[S.~Vassiliadou]{Sophia Vassiliadou}
\address{Department of Mathematics and Statistics,	Georgetown University, Washington, DC 20057 U.S.A}
\email{sv46@georgetown.edu
}

\newcommand{\dsum}[2]{\begin{matrix}{#1}\\ \oplus \\ {#2}\end{matrix}}

\DeclareMathOperator{\cone}{cone}
\DeclareMathOperator{\rank}{rank}

\newcommand{\R}{\mathbb{R}}

\pgfdeclarelayer{bg} 
\pgfsetlayers{bg,main} 
\newcommand{\x}{{\bm{x}}}

\newcommand{\del}{\partial}
\newcommand{\f}{{\bm{f}}}
\newcommand{\y}{{\bm{y}}}

\newcommand{\m}{\mathfrak{m}}
\newcommand{\p}{\mathfrak{p}}

\renewcommand{\H}{\mathcal{H}}
\newcommand{\E}{\mathcal{E}}
\newcommand{\A}{\mathcal{A}}
\newcommand{\Z}{\mathcal{Z}}
\newcommand{\D}{\mathcal{D}}

\DeclareMathOperator{\im}{im}

\DeclareMathOperator{\coker}{coker}

\DeclareMathOperator{\id}{id}

\DeclareMathOperator{\Hom}{Hom}

\DeclareMathOperator{\thick}{\mathsf{thick}}
\DeclareMathOperator{\Tor}{Tor}
\DeclareMathOperator{\Kos}{Kos}

\newcommand{\shift}{{\mathsf{\Sigma}}}
\newcommand{\level}{{\mathsf{level}}}
\newcommand{\Dsg}{\mathsf{D_{sg}}}
\newcommand{\Db}{\mathsf{D^{b}}}

\newcommand{\xla}{\xleftarrow}
\newcommand{\xra}{\xrightarrow}
\newcommand{\fold}{{\mathsf{Fold}}}

\newcommand{\Der}{\hbox{\rm Der}}

\newcommand{\cE}{\mathcal{E}}

\newcommand{\surjdown}{\raisebox{.5pt}{$\downarrow$}\kern -5pt \raisebox{-1pt}{$\downarrow$}}

\newcommand{\mx}{\begin{pmatrix}}
	\newcommand{\emx}{\end{pmatrix}}

\def\todo#1
{\par\vskip 7 pt \noindent {\textcolor{yellow}\footnotesize \color{yellow} \fbox{\parbox{4.8in}{\color{black}
				\textbf{To do: } {\color{blue}#1}}}} \vskip 3 pt \par}

\def\forth#1
{\par\vskip 7 pt \noindent {\textcolor{yellow}\footnotesize \color{yellow} \fbox{\parbox{4.8in}{\color{black}
				\textbf{For thought: } {\color{blue}#1}}}} \vskip 3 pt \par}
			
\def\fifth#1
			{\par\vskip 7 pt \noindent {\textcolor{yellow}\footnotesize \color{yellow} \fbox{\parbox{4.8in}{\color{black}
							\textbf{To discuss: } {\color{ora}#1}}}} \vskip 3 pt \par}

\definecolor{grn}{rgb}{0.0, 0.5, 0.0}
\definecolor{prp}{rgb}{0.4, 0.0, 0.3}
\definecolor{ora}{rgb}{0.8, 0.4, 0.0}

\newtheorem{theorem}{Theorem}[subsection]

\newtheorem{proposition}[theorem]{Proposition}

\newcounter{intro}
\newtheorem{introthm}[intro]{Theorem}
\newtheorem{introcor}[intro]{Corollary}

\newtheorem{lemma}[theorem]{Lemma}
\newtheorem{corollary}[theorem]{Corollary}
\newtheorem{question}[theorem]{Question}

\theoremstyle{definition}

\newtheorem{remark}[theorem]{Remark}
\newtheorem{notation}[theorem]{Notation}

\newtheorem{definition}[theorem]{Definition}

\newtheorem{chunk}[theorem]{}

\newtheorem*{ack}{Acknowledgements}

\begin{document}
%%%%%%%%%%%%%%%%%%%%%%%%%%%%%%%%%%%%
%%%%%%%%%%%%%%%%%%%%%%%%%%%%%%%%%%%%

\thanks{CM was supported by NSF Grant DMS-1802207 and NSF Grant DMS-2302198.\\ 
\indent JJ was supported by NSF CAREER Award DMS-2044833.\\
\indent JP was supported by NSF RTG Grant DMS-1840190, NSF MSPRF Grant DMS-2002173, and NSF Grant DMS-2302567.}
\maketitle
\pagestyle{myheadings}
\markboth{}{}

\begin{abstract}
In this paper we develop a new approach for studying differential operators of an isolated singularity graded hypersurface ring $R$ defining a surface in affine three-space over a field of characteristic zero.  With this method, we construct an explicit minimal generating set for the modules of differential operators of order two and three, as well as their minimal free resolutions; this expands results of Bernstein, Gel'fand, and Gel'fand and of Vigu\'e. 
Our construction relies, in part, on a description of these modules
that we derive in the singularity category of $R$. 
Namely, we build explicit matrix factorizations starting from that of the residue field.
\end{abstract}

%%%%%%%%%%%%%%%%%%%%%%%%%%%%%%%%%%%%%%%%%%%%%%%%%%%%%%%%%%%
\section{Introduction}
%%%%%%%%%%%%%%%%%%%%%%%%%%%%%%%%%%%%%%%%%%%%%%%%%%%%%%%%%%%

For an algebra $R$ over a field $k$, its module of derivations and its module of K\"ahler differentials are fundamental classical objects that inform on the singularities of $R$. The derivations are part of the ring of $k$-linear differential operators on $R$, denoted $D_{R|k}$; this ring is filtered by the submodules $D_{R|k}^i$, the differential operators of order at most $i$, for $i\geqslant 0$. The ring of differential operators has proved invaluable in both algebra and geometry \cite{Bjork:1979,Hotta/Takeuchi/Tanisaki:2008,Lyubeznik:1993,Malgrange:1968}.
 
Throughout this article, $k$ is of characteristic zero. 
In this setting, the ring of differential operators of a polynomial ring $R=k[x_1,\dots,x_n]$ is the Weyl algebra
$D_{R|k}=R\langle \partial_1,\dots,\partial_n \rangle$
where the differential operators of order $i$ are the $R$-linear combinations of partial derivatives of order, in the familiar sense, at most $i$; see \cref{d:diffops} for a precise definition. 
More generally, Grothendieck showed that $D_{R|k}$ is generated as an $R$-algebra by the derivations under composition whenever $R$ is smooth. 
In 1961, Nakai~\cite{Nakai} conjectured that the converse should hold; this is now known as \emph{Nakai's Conjecture} and implies the well-known \emph{Lipman-Zariski Conjecture}~\cite{Lip-65}. Nakai's conjecture is still wide open outside of a handful of cases~\cite{Bro-74,Ish-85,MouVil-73,Sch-94,Sin-86,Tra-99}.

When $R$ is singular, even in specific examples, it is extremely difficult to determine the differential operators of each order that are not compositions of lower order operators; the behavior is radically different from the smooth case. Nevertheless, studying $D_{R|k}$ when $R$ is singular is an old and interesting problem that has seen a revival of interest lately, especially with its connections with simplicity of $D$-modules~\cite{BarDua-20,BGG-72,BJN-19,DDGHN-18,DesGriJef-20,Kan-71,Mal-21,Tra-99,Tri-97,Vig-74}. Many approaches for studying differential operators are either algebraic or via sheaf cohomology.

The singular rings under consideration in the present article are those isolated singularity hypersurface rings considered by Vigu\'e \cite{Vig-74}. He established that $D_{R|k}$ is not generated by the operators of any bounded order and has no differential operators of negative degree when $R=k[x,y,z]/(f)$ is an isolated singularity hypersurface with $f$ homogeneous of degree at least 3; this generalizes the work of Bernstein, Gel'fand, and Gel'fand on the cubic cone $k[x,y,z]/(x^3+y^3+z^3)$ in \cite{Bernstein/Gelfand/Gelfand}. Moreover, Vigu\'e showed that in each order $i$, the module $D^i_{R|k}$ has {at least 3} generators that are not in the $R$-subalgebra of $D_{R|k}$ generated by lower order operators. These operators were identified abstractly by an analysis of sheaf cohomology.

In this paper we develop a method for gleaning new insights on these objects, including the explicit operators of a fixed order. 
Our approach, which is via the {\it homological algebra} of their resolutions, is in some ways similar in spirit to Herzog--Martsinkovsky's work on modules of derivations \cite{Herzog/Martsinkovsky:1993}.
Surprisingly, a crucial ingredient for discovering the 
differential operators of low orders, which can be described as syzygies (cf.\@ \cref{subsec:diff-conc}), was formulating a hypothesized structure for their resolutions, rather than calculating generators as a usual first step in building the resolutions. 
Our method is a way to work forward in the resolution to discover the generators. 
In fact, neither previous work 
in the literature, nor our extensive {\tt Macaulay2} \cite{M2} experiments, identified viable candidates for a possible set of generators. 
We describe this method in \cref{subsubsection:forward} but omit the extensive work involved and concentrate on proving that these are indeed the generators and resolutions by localization and depth counting methods. 

Indeed our methods reveal an unexpectedly beautiful structure to the resolutions: They come from matrix factorizations built in a very simple way from several copies of two different Koszul complexes on the ambient polynomial ring $Q$, namely $\Kos^Q(x,y,z)$ on the variables and  $\Kos^Q(f_x,f_y,f_z)$ on the partial derivatives of $f$, connected by natural maps constructed from the Hessian matrix of second derivatives of $f$ and its exterior powers. See \cref{subsubsection:matrix-factorizations} for details. 

It should be stressed that the computations involved (although they look complex) stem from the familiar Euler identity, and it is really the homological underpinnings that allow for success.

The generating operators can be expressed in terms of foundational derivations. Namely, the generators can be expressed in terms of the Euler and Hamiltonian derivations
\begin{equation*}
 \E \coloneqq x \del_x + y \del_y + z \del_z\,,\quad 
 \H_{yz} \coloneqq f_z \del_y - f_y \del_z \,, \quad
 \H_{zx} \coloneqq f_x \del_z - f_z \del_x\,, \quad \text{and} \quad
 \H_{xy} \coloneqq f_y \del_x - f_x \del_y\,.
 \end{equation*}
 These four operators form a minimal generating set for the module of derivations; its minimal resolution is recalled in \cref{c:d1andB}.
 Our main result is the following, contained in \cref{resolutionD2,genops2,resolutionD3,genops3}.

\begin{introthm}
\label{intro:main}
Assume $R=k[x,y,z]/(f)$ is an isolated hypersurface singularity where $f$ is homogeneous of degree $d\geqslant 3$ and $k$ is a field of characteristic zero.
A minimal set of generators for $D_{R|k}^2$ is given by \[\{1 , \, \E,\H_{yz},\, \H_{zx},\, \H_{xy},\, \E^2,\, \E\H_{yz},\, \E\H_{zx},\, \E\H_{xy},\, \A_x,\, \A_y,\, \A_z\}\,,\]
with 
\begin{align*}
\A_x &= \frac{1}{x} 
\left[
\H_{yz}^2+\frac{1}{(d-1)^2}\Delta_{xx}\E^2+\frac{d-2}{(d-1)^2}\Delta_{xx}\E
\right]\,,
\end{align*}
where $\Delta_{xx}$ is the 2$\times$2-minor obtained by deleting the 1st row and 1st column of the Hessian matrix of $f$, and similar formulas hold for $\A_y $, $\A_z $ as described in \cref{genops2}.

A minimal set of generators for $D^3_{R|k}$ is the union of the generating set above with
\[
\{\E^3,\,\E^2\H_{yz},\,\E^2\H_{zx},\,\E^2\H_{xy},\, \E\A_x,\, \E\A_y,\, \E\A_z,\,\Z_x,\,\Z_y,\,\Z_z\}\,,\]
where $\Z_x$, $\Z_y$, and $\Z_z$ are defined in \cref{genops3}.

Furthermore, the augmented minimal $R$-free resolutions of $D_{R|k}^2$ and $D^3_{R|k}$ have the forms
\[
\cdots \xra{\psi} R^{11} \xra{\varphi} R^{11} \xra{\psi} R^{11} \xra{\varphi} R^{11} \xra{\tiny \begin{bmatrix} \psi \\ 0\end{bmatrix}} R^{12} \to D_{R|k}^2 \to 0
\]
and
\[
\cdots \xra{\psi'} R^{21} \xra{\varphi'} 
R^{21} \xra{\psi'} R^{21} \xra{\varphi'} R^{21} \xra{\tiny \begin{bmatrix} \psi' \\ 0\end{bmatrix}} R^{22} \xra{} D_{R|k}^3 \to 0\,,
\]
where the differentials are described explicitly in \cref{resolutionD2,resolutionD3}, respectively. 
\\
\end{introthm}

The resolutions in \cref{intro:main} are minimal graded free resolutions, where $R$ is viewed as a standard graded $k$-algebra. We record the graded Betti numbers of $D_{R|k}^2$ and $D_{R|k}^3$ in \cref{cor:d2graded,cor:d3graded}, respectively.

In the process of proving \cref{intro:main} we also establish, in \cref{cor:exactsequenceD2,cor:exactsequenced3}, the following result.

\begin{introcor}
\label{intro_cor_exact}
For $i=2,3$, there are short exact sequences of $R$-modules 
\[
0\to D^{i-1}_{R|k}\to D^i_{R|k}\to \Hom_R(\mathrm{Sym}^i (\Omega_{R|k}), R) \to 0\,,
\]
where $D^{i-1}_{R|k}\to D^{i}_{R|k}$ is the inclusion; in particular, $D^i_{R|k}/D^{i-1}_{R|k}$ is maximal Cohen-Macaulay.
\end{introcor}

We note that there are always left exact sequences of the form above; surjectivity on the right is the novel content.
Some of the results above are naturally stated in terms of the 
 the singularity category $\Dsg(R)$ of $R$ introduced by Buchweitz~\cite{Buchweitz}; see \cref{c:equivalence}. The singularity category is a triangulated category that records the tails of resolutions over $R$. As $R$ is an isolated singularity, it well-known that every object in $\Dsg(R)$ can be built from $k$ in finitely many steps using shifts, mapping cones and retracts. The number of mapping cones $(+1)$ it takes to build $M$ from $k$ is called the level of $M$ with respect to $k$, denoted $\level_{\Dsg(R)}^kM$; this invariant was introduced in \cite{Avramov/Buchweitz/Iyengar/Miller:2010a}. Since $R$ is a hypersurface, the tail of any $R$-resolution is necessarily given by a matrix factorization, in the sense of \cite{Eisenbud:1980}, of $f$ over $k[x,y,z]$. 
In constructing the resolutions of $\Hom_R(\mathrm{Sym}^i(\Omega_{R|k}),R)$, for $i=2,3$, we describe a procedure to build its corresponding matrix factorization starting from the matrix factorization of $k$; see \cref{subsec:mfS2,subsec:mfS3}. As a consequence, we obtain the following; see \cref{cor:levels2,cor:levels3}.
\begin{introcor}
For $i=2,3$,
\[
\level_{\Dsg(R)}^k (\Hom_R(\mathrm{Sym}^i(\Omega_{R|k}),R))\leqslant i\quad \text{and}\quad\level_{\Dsg(R)}^k (D_{R|k}^i)\leqslant i(i+1)/2\,.
\]
\end{introcor}

%%%%%%%%%%%%%%%%%%%%%%%%%%%%%%%%%%%%%%%%%%%%%%%%%%%%%%%%%%%%%%
\subsection{Methods}\label{sec:method}
%%%%%%%%%%%%%%%%%%%%%%%%%%%%%%%%%%%%%%%%%%%%%%%%%%%%%%%%%%%%%%

Our general approach to \cref{intro:main} is as follows: Focusing on the order filtration on the ring of differential operators
\[
D_{R|k}^0 \subseteq D_{R|k}^1 \subseteq D_{R|k}^2 \subseteq \cdots,
\]
we proceed inductively. For $i>1$, we first find the minimal $R$-free resolution and a minimal set of generators of $\Hom_R(\mathrm{Sym}^i (\Omega_{R|k}), R)$, the $R$-linear dual of the $i$th symmetric power of the modules of K\"ahler differentials. Next, we build a degree one chain map from this resolution to our, inductively constructed, resolution of $D^{i-1}_{R|k}$. We then show its mapping cone is a minimal resolution of $D^i_{R|k}$ which yields \cref{intro_cor_exact}. 

We employ two key techniques in building the resolution of $\Hom_R(\mathrm{Sym}^i (\Omega_{R|k}), R)$.

\subsubsection{Matrix factorizations from diagrams of Koszul complexes}\label{subsubsection:matrix-factorizations}

Here we describe the structure of the matrix factorizations associated with the minimal resolutions of each $D^n/D^{n-1}$ for various orders $n$. They are obtained from the $\mathbb{Z}/2$-graded totalizations of the diagrams below by adding in nullhomotopies for multiplication by $f$. Equivalently to find homotopies, we instead describe a structure of dg module over the Koszul complex $\Kos(f)$ on these totalized complexes, as described in \cref{c:mf} and \cref{c:fold}. The diagram for $D^1/D^0$ is

\begin{figure}[htb]

\begin{equation*}
\begin{aligned} 
 \xymatrix@R+.5pc@C+.7pc{
%%%%%%%%%%%%%%%%%%%%%%%%%%%%%%%%%%%%%%%%%%%%%%%%%%% row 1
 \Kos(x,y,z) && Q \ar[r]^{\partial_3} &Q^3 \ar[r]^{\partial_2}&Q^3 \ar[r]^{\partial_1}& Q \,.
}
\end{aligned} 
 \end{equation*}
 \caption{Diagram for matrix factorization of $D^1/D^0$}
\end{figure}
The diagram for $D^2/D^1$ follows below. Notice that it is the diagram for $D^1/D^0$ with a new first row glued in. The new maps are described in \cref{sec:glossary}.

\begin{figure}[htb]
\begin{equation*}
\begin{aligned} 
 \xymatrix@R+.5pc@C+.7pc{
%%%%%%%%%%%%%%%%%%%%%%%%%%%%%%%%%%%%%%%%%%%%%%%%%%% row 1
 \Kos(f_x,f_y,f_z) \ar[d]^{\bigwedge \mathrm{Hessian}(f)} &&  Q \ar[r]^{D_3}\ar[d]^{\alpha_3}  &Q^3 \ar[r]^{D_2}  \ar[d]^{\alpha_2}   &Q^3 \ar[r]^{D_1}\ar[d]^{\alpha_1}& Q \ar@{=}[d]^{\alpha_0}\\
 %%%%%%%%%%%%%%%%%%%%%%%%%%%%%%%row 2
 \Kos(x,y,z) && Q \ar[r]^{\partial_3} &Q^3 \ar[r]^{\partial_2}&Q^3 \ar[r]^{\partial_1}& Q 
}
\end{aligned} 
 \end{equation*} 
 \caption{Diagram for matrix factorization of $D^2/D^1$}
\end{figure}

The diagram for $D^3/D^2$ is then obtained by gluing a new first row into the diagram for $D^2/D^1$. The diagonal maps are obtained from the Koszul complex on the partial derivatives of the Hessian determinant of $f$; these are described as $\sigma_i$ for $i =1,2,3$ in \cref{sec:glossary}.

\begin{figure}[htb]
\begin{equation*}
\begin{aligned} 
 \xymatrix@R+.5pc@C+.7pc{
%%%%%%%%%%%%%%%%%%%%%%%%%%%%%%%%%%%%%%%%%%%%%%%%%%% row 1
\Kos(x,y,z) \ar[d]^{(\bigwedge \mathrm{Hessian}(f))^*} \ar@[violet]@/_3pc/@{-->}[dd]&&  Q \ar[r]^{\partial_3}\ar@{=}[d]^{\alpha_0} &Q^3  \ar@[violet]@/^/[ldd]  \ar[r]^{\partial_2} \ar[d]^{-\alpha_1}&Q^3  \ar@[violet]@/^/[ldd]  \ar[r]^{\partial_1}\ar[d]^{\alpha_2}& Q \ar[d]^{-\alpha_3} \ar@[violet]@/^/[ldd] \\
 %%%%%%%%%%%%%%%%%%%%%%%%%%%%%%%row 2
% 0 \ar[r] & Q \ar[r]^{\partial_3^L} &Q^4 \ar[r]^{\partial_2^L}&Q^6 \ar[r]^{\partial_1^L}& Q^3 \ar[r]&0
 \Kos(f_x,f_y,f_z) \ar[d]^{\bigwedge \mathrm{Hessian}(f)} &&  Q \ar[r]^{D_3}\ar[d]^{\alpha_3} &Q^3 \ar[r]^{D_2} \ar[d]^{\alpha_2}&Q^3 \ar[r]^{D_1}\ar[d]^{\alpha_1}& Q \ar@{=}[d]^{\alpha_0}\\
 %%%%%%%%%%%%%%%%%%%%%%%%%%%%%%%row 2
 \Kos(x,y,z) && Q \ar[r]^{\partial_3} &Q^3 \ar[r]^{\partial_2}&Q^3 \ar[r]^{\partial_1}& Q 
}
\end{aligned} 
 \end{equation*}
  \caption{Diagram for matrix factorization of $D^3/D^2$}
\end{figure}

For order 4 we conjecture that the diagram on the right in Figure \ref{figure_conjecture}, together with additional maps from the top row to the bottom that are as yet undetermined (the orange maps indicated on the left), supports the matrix factorization for the resolution. 
The Betti numbers and the displayed maps certainly agree with our computations on \texttt{Macaulay2}. 
Similarly, we believe that analogous diagrams with $n$ rows support the factorizations for higher orders $n$. 

\begin{figure}[htb]
\begin{equation*}
\begin{aligned} 
 \xymatrix@R+.5pc@C+.7pc{
 %%%%%%%%%%%%%%%%%%%%%%%%%%%%%%%%%%%%%%%%%%%%%%%%%%% row 1
\Kos(f_x,f_y,f_z) \ar[d]^{\bigwedge \mathrm{Hessian}(f)} \ar@[teal]@/_3pc/@{-->}[dd] \ar@[orange]@/_4pc/@{-->}[ddd]_{{  \color{orange} \textbf{\Large ?}}}
&&  Q \ar[r]^{D_3}\ar[d]^{\alpha_3} \ar@[teal]@/^/[ddr] &Q^3 \ar[r]^{D_2} \ar[d]^{\alpha_2} \ar@[teal]@/^/[ddr] &Q^3 \ar[r]^{D_1}\ar[d]^{\alpha_1} \ar@[teal]@/^/[ddr]& Q \ar@{=}[d]^{\alpha_0}\\
%%%%%%%%%%%%%%%%%%%%%%%%%%%%%%%%%%%%%%%%%%%%%%%%%%% row 2
\Kos(x,y,z) \ar[d]^{(\bigwedge \mathrm{Hessian}(f))^*}\ar@[violet]@/_3pc/@{-->}[dd] && Q \ar[r]_{\partial_3}\ar@{=}[d]^{\alpha_0} &Q^3  \ar@[violet]@/^/[ldd]  \ar[r]_{\partial_2} \ar[d]^{-\alpha_1}&Q^3  \ar@[violet]@/^/[ldd]  \ar[r]_{\partial_1}\ar[d]^{\alpha_2}& Q \ar[d]^{-\alpha_3} \ar@[violet]@/^/[ldd] \\
 %%%%%%%%%%%%%%%%%%%%%%%%%%%%%%%row 3
\Kos(f_x,f_y,f_z) \ar[d]^{\bigwedge \mathrm{Hessian}(f)} && Q \ar[r]^{D_3}\ar[d]^{\alpha_3} &Q^3 \ar[r]^{D_2} \ar[d]^{\alpha_2}&Q^3 \ar[r]^{D_1}\ar[d]^{\alpha_1}& Q \ar@{=}[d]^{\alpha_0}\\
 %%%%%%%%%%%%%%%%%%%%%%%%%%%%%%%row 4
\Kos(x,y,z) &&   Q \ar[r]^{\partial_3} &Q^3 \ar[r]^{\partial_2}&Q^3 \ar[r]^{\partial_1}& Q 
}
\end{aligned} 
 \end{equation*}
   \caption{Conjectured diagram for matrix factorization of $D^4/D^3$}
   \label{figure_conjecture}
\end{figure}

\subsubsection{Working forward from matrix factorizations to relations and generators}\label{subsubsection:forward}

We briefly describe our method for discovering the generators of $D^i/D^{i-1}$ for $i=2,3$. 
However, the details of these calculations do not appear in this paper; instead we concentrate only on proving exactness of the augmented complexes that arise from the method.

We realize $\Hom_R(\mathrm{Sym}^i(\Omega_{R|k}),R)$ as the kernel of a matrix $J_{i, i-1}$ obtained from the Jacobian matrix; see \cref{r:Jnotation}.
Its minimal resolution fits into a diagram as follows.
\begin{equation}
\label{desired_res}
F=\qquad
\xymatrixrowsep{1pc}
\xymatrixcolsep{1pc}
\xymatrix{
\cdots 
\ar@{->}[rr]^{\partial_2}
&&F_1
\ar@{->}[rr]^-{\partial_1}
&&F_0
\ar@{->}[rr]^-{\partial_0}
\ar@{->>}[dr]
&&F_{-1}
\ar@{->}[rrr]^-{\partial_{-1}=J_{i, i-1}}
&&&F_{-2}
&
\\
&&&&&
\ker J_{i,i-1}
\ar@{->}[ur]
&&&&
}
\end{equation}
We conjecture that the form of $F_{\geq 0}$ is given by the complex coming from the 
matrix factorization described in \cref{subsubsection:matrix-factorizations} above. 
Then we work forward to discover a suitable map $\partial_0$ that completes the diagram. This gives the generators of $D^i/D^{i-1}$.

To discover the map $\partial_0$, first note that the dual of the complex $F$ in \cref{desired_res} is a complex.
Therefore, the image of $\partial_0^*$ is contained in the kernel of $\partial_1^*$. 
But, the dual $F_{\geq 0}^*$ of the 2-periodic portion of the complex $F_{\geq 0}$ is exact since it arises from the transposed matrix factorization. 
Therefore, the kernel of $\partial_1^*$ is the image of $\partial_0^*$. 
Choosing bases, this means that the rows of $\partial_0$ are combinations of the rows of $\partial_2$. 
We construct the rows of $\partial_0$ from suitable $R$-linear combinations of the rows of $\partial_2$ so that the columns of $\partial_0$ are in the kernel of $J_{i,i-1}$ 
and match the expected degrees of the operators.

\subsection{Convention}
Before sketching an outline for the paper, we wish to highlight an {important} bit of notation here.
In the body of the paper, given a differential operator $\mathcal{F}$ in $D_{R|k}^i$ we will write its images in $D^i_{R|k}/D_{R|k}^{i-1}$ and $\Hom_R(\mathrm{Sym}^i(\Omega_{R|k}),R)$ in roman font $F$.
In \cref{c:composition}, we also identify $D^i_{R|k}/D_{R|k}^{i-1}$ as a submodule of a free module and so $F$ will also be identified with its corresponding column vector. The basis of the column space is indexed by the divided power partial derivatives of order exactly $i$ (listed in lexicographic order); see \cref{c:basis} for further details.
For example, $EH_{xy}$ denotes the image of the differential operator $\E\circ \H_{xy}$ in $D_{R|k}^2/D_{R|k}^1$:
\[
EH_{xy}=xf_y\partial_x^2+(yf_y-xf_x)\partial_x\partial_y+zf_y\partial_x\partial_z-yf_x\partial_y^2-zf_x\partial_y\partial_z \mod D^1_{R|k}\,,
\] as well as the column vector 
\[
EH_{xy}=\begin{bmatrix}
2 x f_y\\ y f_y - x f_x \\ z f_y\\-2y f_x\\- z f_x\\0
\end{bmatrix} \,.
\]
It turns out that working in the quotients $D^i_{R|k}/D_{R|k}^{i-1}$, and this slight abuse of notation, simplifies several calculations. See \cref{rmk:lifting-gens}, \cref{c:composition}, \cref{subsec:gensD2}, and \cref{subsec:gensD3} for more details on these notational conventions.

\subsection{Outline}
Finally, we end the introduction with a brief outline of the paper. \cref{sec:background} contains background information on differential operators, homological constructions (such as matrix factorizations), and a minimal free resolution of $D^1_{R|k}$. \cref{sec:glossary} contains a glossary of matrices that are used throughout the remainder of the document. 

\cref{sec:D2,sec:D3} are the central parts of the paper. We briefly describe the former, and the latter follows the identical outline: In \cref{subsec:gensS2}, a set of generators for $\Hom_R(\mathrm{Sym}^2(\Omega_{R|k}),R)$ is proposed; a matrix factorization corresponding to this module is introduced in \cref{subsec:mfS2}; in \cref{subsec:exactnessS2}, the pieces from the previous subsections are stitched together to identify a minimal free resolution of $\Hom_R(\mathrm{Sym}^2(\Omega_{R|k}),R)$; it is in \cref{subsec:liftS2} that the degree one map from this resolution to the resolution of $D^1_{R|k}$ is constructed; as a consequence the minimal $R$-free resolution of $D^2_{R|k}$ is obtained in \cref{resolutionD2}; finally, in \cref{subsec:gensD2}, we identify a minimal set of generators for $D^2_{R|k}$ written as honest differential operators in $D_{R|k}$. 

The article is concluded by several appendices containing formulas and calculations that are utilized in the arguments discussed above.

%%%%%%%%%%%%%%%%%%%%%%%%%%%%%%%%%%%%%%%%%%%%%%%%%%%%%%%%%%%%%%
\section{Background}\label{sec:background}
%%%%%%%%%%%%%%%%%%%%%%%%%%%%%%%%%%%%%%%%%%%%%%%%%%%%%%%%%%%%%%

In this paper, we work in the following setting unless specified otherwise.

\begin{notation}\label{notation:main}
Throughout $Q$ is the polynomial ring $k[x,y,z]$ over a field $k$, let $f$ be a homogeneous polynomial in $Q$ of a fixed degree $d\geqslant 3,$ and set \[R\coloneqq Q/(f)\,.\]
We assume the characteristic of $k$ is zero.
As $f\in (x,y,z)^3$, it follows that $R$ is a singular hypersurface ring. Finally, we make the assumption that $R$ is an isolated singularity, meaning---in the usual sense---that $R_\p$ is a regular local ring for all primes $\p$ different from the maximal ideal $\m=(x,y,z)R$ of $R.$
This implies, since $R$ is therefore a normal graded ring, that $R$ is a domain. 
\end{notation}

%%%%%%%%%%%%%%%%%%%%%%%%%%%%%%%%%%%%
\subsection{Generalities on differential operators}\label{subsec:diff-gen}
%%%%%%%%%%%%%%%%%%%%%%%%%%%%%%%%%%%%

We recall some basics on Grothendieck's notion of differential operators; cf. \cite{Grothendieck:EGA}. We do not restrict ourselves to the setting of \cref{notation:main} in this subsection.

\begin{definition} 
\label{d:diffops}Let $\phi\colon k\to A$ be a homomorphism of commutative rings. The module of $k$-linear \emph{differential operators} on $A$ of order at most $i$, denoted $D^i_{A|k}$, is defined inductively as follows:
\begin{itemize}
\item $D^0_{A|k} = \mathrm{Hom}_A(A,A) \cong A$\,;
\item $D^i_{A|k} = \{ \delta\in \mathrm{Hom}_k(A,A) \mid \delta \circ \mu -\mu \circ \delta \in D^{i-1}_{A|k} \ \text{for all} \ \mu\in D^0_{A|k} \}\quad \text{for all }i>0$\,.
\end{itemize}
Each $D^i_{A|k}$ is an $A$-module where $A$ acts via postmultiplication.
The union of the modules $D^{i}_{A|k}$ over all natural numbers $i$ is the ring of $k$-linear differential operators on $A$, denoted $D_{A|k}$.
\end{definition}

\begin{chunk}
Given two differential operators $\alpha\in D^{i}_{A|k}$ and $\beta\in D^{j}_{A|k}$, the composition $\alpha\circ \beta$ is an element of $D^{i+j}_{A|k}$; the union $D_{A|k}$ obtains the structure of a noncommutative ring with composition as the multiplication operation. We write compositions with product notation in the sequel. As multiplication is compatible with the order filtration, $D_{A|k}$ equipped with the order filtration is a filtered ring. In general, the associated graded ring with respect to this filtration is a commutative ring, which we denote as $\mathrm{gr} (D_{A|k})$. We employ the tautological exact sequences
\[ 0 \to D^{i-1}_{A|k} \to D^i_{A|k} \to \mathrm{gr} (D_{A|k})_i \to 0 \]
below.
\end{chunk}

\begin{chunk}\label{c:splitting-order-ses} In general, the inclusion $D^{i-1}_{A|k} \subseteq D^i_{A|k}$ does not split as $A$-modules. However, the inclusion $D^0_{A|k} \subseteq D^i_{A|k}$ has an $A$-linear left inverse given by evaluation at $1\in A$; the kernel of this evaluation map is called the module of $i$-th order derivations. For $i=1$, this is the usual module of derivations $\mathrm{Der}_{A|k}$.
\end{chunk}

\begin{chunk} 
\label{c:basis}
If $k$ is a commutative ring, and $P=k[x_1,\dots,x_n]$ is a polynomial ring over $k$, then $D^i_{P|k}$ is a free $P$-module with basis
\[ \{ \partial^{(a_1)}_{x_1} \cdots \partial^{(a_n)}_{x_n} \mid a_1 + \cdots + a_n \leqslant i\}\,,\]
where $\partial^{(a_i)}_{x_i}$ is the $k$-linear map determined by
\[ \partial^{(a_i)}_{x_i} (x_1^{b_1} \cdots x_n^{b_n}) = \binom{b_i}{a_i} x_1^{b_1} \cdots x_i^{b_i-a_i} \cdots x_n^{b_n};\]
if $\mathbb{Q} \subseteq k$, then we can identify $\partial^{(a_i)}_{x_i}$ with $\frac{1}{a_i !}$ times the $a_i$-iterate of the derivation~$\frac{d}{d x_i}$. By convention, we set $\partial^{(a_i)}_{x_i}$ to be the zero map if $a_i<0$.

We also have that $\mathrm{gr} (D_{P|k})_i$ is a free $P$-module with basis
\[ \{ \overline{\partial^{(a_1)}_{x_1} \cdots \partial^{(a_n)}_{x_n}} \mid a_1 + \cdots + a_n = i\}\,,\]
where overline denotes congruence class modulo $D_{P|k}^{i-1}$. The ring $\mathrm{gr} (D_{P|k})$ is a divided power algebra over $B$ in $n$ divided power variables generated by the classes of the derivations $\frac{d}{d x_i}$; if $\mathbb{Q} \subseteq k$, it is isomorphic to a polynomial ring on these variables, or, equivalently, the symmetric algebra on $\Der_{P|k}$.
\end{chunk}

\begin{chunk}\label{chunk:DOps-quot-1} Let $A$ be a commutative ring, $P=k[x_1,\dots,x_n]$ be a polynomial ring over~$k$, and set $A=P/J$ with $J=(f_1,\dots,f_m)$ an ideal of $P$. There are isomorphisms
\[ D^i_{A|k} \cong \frac{ \{ \delta \in D^{i}_{P|k} \mid \delta(J)\subseteq J\}}{J D^{i}_{P|k} }\,.\]
To identify these elements in the numerator, we introduce the following notation: for a differential operator $\delta$ and sequence of elements $\mathbf{g}=(g_1,\dots,g_m)$, we set 
\[[\delta,\mathbf{g}]\coloneqq [ \cdots [[ \delta, g_1 ], g_2],\cdots ,g_m]\,.\]
Note that $[\delta,\mathbf{g}]$ is unaffected by permuting the elements of $\mathbf{g}$. If $\mathbf{g}$ is the empty sequence, we take $[\delta,\mathbf{g}]\coloneqq\delta$.
\end{chunk}

\begin{lemma}\label{lem:bracket-tuples-operators} In the notation of \ref{chunk:DOps-quot-1}, an element $\delta\in D^i_{P|k}$ maps $J$ into $J$ if and only if $[\delta, \mathbf{g}](f_j) \in J$
for all $j=1,\dots,m$ and all tuples $\mathbf{g}=(g_1,\dots,g_{\ell})$ such that $0 \leqslant \ell <i$ and $g_t \in \{x_1,\dots,x_n\}$ for all $t=1,\dots,\ell$.
\end{lemma}
\begin{proof} The forward implication is straightforward. For the reverse implication, it suffices to show that if $\delta\in D^i_{P|k}$ satisfies the condition in the statement, then \[\delta(x_1^{a_1}\cdots x_n^{a_n} f_j) \in J \quad \text{for all} \ a_1,\dots,a_n \geqslant 0\,.\]
We proceed by induction on $i$. For the case $i=0$, an operator of order zero stabilizes every ideal, so the claim is clear. For the inductive step, we proceed by induction on $\ell=a_1+\cdots+a_n$. For the case $\ell=0$, this follows by the hypothesis $[\delta, \varnothing](f_j) \in J$. For $\ell > 0$, we have $a_k>0$ for some $k$; without loss of generality, we can take $k=1$. For any sequence $\mathbf{g}$ of variables of length at most $i-1$, let $\mathbf{g'}$ be the sequence obtained from $\mathbf{g}$ by appending $x_1$. Then $[\delta,x_1] \in D^{i-1}_{B|A}$ and $[[\delta,x_1], \mathbf{g}](f_j) = [\delta, \mathbf{g'}](f_j) \in J$ by hypothesis, so by the induction hypothesis on $i$, $[\delta,x_1]$ maps $J$ into $J$. By the induction hypothesis on $\ell$, $\delta(x_1^{a_1-1}\cdots x_n^{a_n} f_j) \in J$. Thus,
\[ \begin{aligned} \delta(x_1^{a_1}\cdots x_n^{a_n} f_j)&= \delta x_1 (x_1^{a_1-1} \cdots x_n^{a_n} f_j) \\ &= x_1 \delta(x_1^{a_1-1}\cdots x_n^{a_n} f_j) + [\delta,x_1](x_1^{a_1-1}\cdots x_n^{a_n} f_j) \in J\,,\end{aligned}\]
as required.
\end{proof}

We may identify $D^i_{A|k}$ with a submodule of the free $A$-module $A\otimes_P D^i_{P|k}$ via the isomorphism above, where an element of $A\otimes_P D^i_{P|k}$ corresponds to an operator taken modulo~$J$. Under this identification, we have:

\begin{proposition}\label{prop:kernel-operator}
In the notation of \cref{chunk:DOps-quot-1}, an element of the free $A$-module $A\otimes_P D^i_{P|k}$ corresponds to an element of $D^i_{A|k}$ if and only if it is in the kernel of the homomorphism of free $A$-modules $\phi\colon A\otimes_P D^i_{P|k} \to F_{< i}$ given by
\[ \phi_i(\partial^{(a_1)}_{x_1} \cdots \partial^{(a_n)}_{x_n}) = \sum_{\substack{0\leqslant b_1 + \cdots + b_n < i \\ 1\leqslant j \leqslant m}} \partial^{(a_1-b_1)}_{x_1} \cdots \partial^{(a_n-b_n)}_{x_n}(f_j) e_{(b_1,\dots,b_n),j}\,,\]
where $F_{< i}$ is the free $A$-module with basis 
\[ \{ 
e_{(b_1,\dots,b_n) , j} \mid 0\leqslant b_1 + \cdots + b_n < i , \, 1\leqslant j \leqslant m\}\,. \]
\end{proposition}
\begin{proof} This follows from Lemma~\ref{lem:bracket-tuples-operators} plus the observation that, for the sequence $\mathbf{g}$ in which each $x_i$ appears $b_i$ times, we have $[\partial^{(a_1)}_{x_1} \cdots \partial^{(a_n)}_{x_n},\mathbf{g}] = \partial^{(a_1-b_1)}_{x_1} \cdots \partial^{(a_n-b_n)}_{x_n}$.
\end{proof}

We note that this description also follows from the presentations of modules of principal parts in \cite{Barajas/Duarte} and~\cite{Brenner/Jeffries/Nunez}.

\begin{chunk}\label{chunk:diagram-dops-and-gr}
The maps $\phi_{i}$ from Proposition~\ref{prop:kernel-operator} fit into a commutative diagram
\begin{equation}\label{eq-2.1.7a} \xymatrix{ 0 \ar[r] & A\otimes_P D^{i-1}_{P|k} \ar[r] \ar[d]^{\phi_{i-1}} & A\otimes_P D^{i}_{P|k} \ar[r]\ar[d]^{\phi_i} & A\otimes_P \mathrm{gr}(D^{i}_{P|k})_i \ar[r]\ar[d]^{\psi_i} & 0 \\
0 \ar[r] & F_{< i-1} \ar[r] & F_{< i} \ar[r] & F_{= i-1} \ar[r] & 0} \end{equation}
with exact rows, where $F_{= i-1}$ is the free $A$-module with basis 
\[ \{ 
e_{(b_1,\dots,b_n) , j} \mid b_1 + \cdots + b_n =i-1 , \, 1\leqslant j \leqslant m\}\,, \] and
\[\psi_i(\overline{\partial^{(a_1)}_{x_1} \cdots \partial^{(a_n)}_{x_n}}) = \sum_{\substack{ b_1 + \cdots + b_n = i-1 \\ 1\leqslant j \leqslant m}} \partial^{(a_1-b_1)}_{x_1} \cdots \partial^{(a_n-b_n)}_{x_n}(f_j) e_{(b_1,\dots,b_n),j}\,.\]

Note that $\partial^{(a_1-b_1)}_{x_1} \cdots \partial^{(a_n-b_n)}_{x_n}(f_j)$ is nonzero only when, for some $\ell\in\{1,\dots,n\}$, we have $a_\ell=b_\ell+1$ and $a_m = b_m$ for $m\neq \ell$; for this tuple $b_1,\dots,b_n$, we have \[ \partial^{(a_1-b_1)}_{x_1} \cdots \partial^{(a_n-b_n)}_{x_n}(f_j) = \sum_j \partial_{x_\ell}(f_j) e_{(b_1,\dots,b_n),j}\,.\]
From this description, one sees that $\psi_i$ can be identified with the $i$th symmetric power of the map $\psi_1$, which is concretely given by
\[ \psi_1 (\partial_{x_i}) = \sum_j \partial_{x_i}(f_j) e_{(0,\dots,0),j}\,;\]
i.e., $\psi_1$ is the transpose of the Jacobian morphism that presents the module of Kahler differentials.
\end{chunk}

\begin{chunk}\label{c:compositions&kerS} By considering kernels in \eqref{eq-2.1.7a}, we obtain a left exact sequence
\begin{equation}\label{eq:left-exact-seq-diff} 0 \to D^{i-1}_{A|k} \to D^i_{A|k} \to \mathrm{ker}(\psi_i)\end{equation}
induced by the natural inclusion and projection maps.
In general, this does not need to be a short exact sequence---i.e., the last map need not be surjective---however, when $i=1$, this is the case, and the kernel of $\psi_1$ is the usual module of derivations.

From \eqref{eq:left-exact-seq-diff}, and the definition, we obtain injective morphisms
\begin{equation}\label{eq:induced-maps-gr} \mathrm{gr}(D_{A|k})_i \to \mathrm{ker}(\psi_i)\,,\end{equation}
which are isomorphisms when \eqref{eq:left-exact-seq-diff} forms a short exact sequence.

For an $A$-linear derivation $\eta \in D^1_{A|k}$ and $\mu\in D^i_{A|k}$, we have $\eta\circ\mu \in D^{i+1}_{A|k}$. Thus, composition with $\eta$ gives a well-defined map:
\begin{equation}\label{eq:induced-maps-gr'} \mathrm{gr}(D_{A|k})_i \xrightarrow{\eta} \mathrm{gr}(D_{A|k})_{i+1}\,.
\end{equation}

If the morphism \eqref{eq:induced-maps-gr} is an isomorphism for some $i$, then the morphism \eqref{eq:induced-maps-gr'} of composition by $\eta$ yields a morphism
\begin{equation}\label{eq:induced-maps-ker} \mathrm{ker}(\psi_i) \xrightarrow{\eta} \mathrm{ker}(\psi_{i+1})\,.\end{equation}
\end{chunk}

%%%%%%%%%%%%%%%%%%%%%%%%%%%%%%%%%%%%
\subsection{Concrete considerations on differential operators}\label{subsec:diff-conc}
%%%%%%%%%%%%%%%%%%%%%%%%%%%%%%%%%%%%

We now realize the objects and maps discussed in \cref{subsec:diff-gen} concretely in the setting of \cref{notation:main}.

\begin{chunk}\label{chunk:concrete1} We have that $D^i_{Q|k}$ is a free $Q$-module with basis
\[ \{ \partial_x^{(a)} \partial_y^{(b)} \partial_z^{(c)} \mid 0\leqslant a+b+c \leqslant i\}\,.\]
We order our bases with the graded lexicographic order throughout. In particular, the ordered basis for $F_{\leqslant 3}$ we use throughout is
\[ \begin{aligned} &\partial_x^{(3)}, \partial_x^{(2)} \partial_y, \partial_x^{(2)} \partial_z, \partial_x \partial_y^{(2)}, \partial_x \partial_y \partial_z, \partial_x \partial_z^{(2)}, \partial_y^{(3)}, \partial_y^{(2)} \partial_z, \partial_y \partial_z^{(2)}, \partial_z^{(3)}, \\
&\partial_x^{(2)}, \partial_x \partial_y, \partial_x \partial_z, \partial_y^{(2)}, \partial_y \partial_z, \partial_z^{(2)}, \\
&\partial_x, \partial_y,\partial_z,\\ 
&1;
\end{aligned}\]
likewise, the first row above gives the ordered basis of $\mathrm{gr}(D_{Q|k})_3$, the last three rows give the ordered basis of $D_{Q|k}^2$, and so on. We order the bases of $F_{<i}$ similarly.
\end{chunk}

\begin{chunk}\label{c:presentationS2} 
We now write the matrices corresponding to the maps $\phi_i$ and $\psi_i$ from \cref{prop:kernel-operator} and \cref{chunk:diagram-dops-and-gr} with respect to the bases in \cref{chunk:concrete1}. First, the matrices for $\psi_1$, $\psi_2$, and $\psi_3$, respectively, are
\[ J_{1,0} \coloneqq \kbordermatrix{
& \partial_x & \partial_y & \partial_z \\ e_1 & f_x & f_y & f_z } \]

\[ J_{2,1} \coloneqq \kbordermatrix{
& \partial_x^{(2)}& \partial_x \partial_y& \partial_x \partial_z& \partial_y^{(2)}& \partial_y \partial_z& \partial_z^{(2)} \\ 
e_x & f_x & f_y & f_z & 0 & 0 & 0 \\
e_y & 0 & f_x & 0 & f_y & f_z & 0 \\
e_z & 0 & 0 & f_x & 0 & f_y & f_z } \]

\[ J_{3,2} \coloneqq { \kbordermatrix{
& \partial_x^{(3)}& \partial_x^{(2)} \partial_y& \partial_x^{(2)} \partial_z& \partial_x \partial_y^{(2)}& \partial_x \partial_y \partial_z&
\partial_x \partial_z^{(2)}& \partial_y^{(3)}& \partial_y^{(2)} \partial_z& \partial_y \partial_z^{(2)}& \partial_z^{(3)} \\ 
e_{xx} & f_x & f_y & f_z & 0 & 0 & 0 & 0 & 0 & 0 & 0 \\
e_{xy} & 0 & f_x & 0 & f_y & f_z & 0 & 0 & 0 & 0 & 0 \\
e_{xz} & 0 & 0 & f_x & 0 & f_y & f_z & 0 & 0 & 0 & 0 \\
e_{yy} & 0 & 0 & 0 & f_x & 0 & 0 & f_y & f_z & 0 & 0 \\
e_{yz} & 0 & 0 & 0 & 0 & f_x & 0 & 0 & f_y & f_z & 0 \\
e_{zz} & 0 & 0 & 0 & 0 & 0 & f_x & 0 & 0 & f_y & f_z
} } \,.\]

Above, we have rewritten the target bases as monomials; for example, $e_{xy}$ corresponds to $e_{(1,1,0),1}$ in the notation of \cref{chunk:diagram-dops-and-gr}. 

In particular, for $i=2,3$, the left exact sequences \eqref{eq:left-exact-seq-diff} take the form
\begin{equation}\label{eq:LES-ring} 0 \to D^{i-1}_{R|k} \to D^{i}_{R|k} \xra{\nu_i} \ker(J_{i,i-1})\,,\end{equation}
inducing injective maps as in \eqref{eq:induced-maps-gr},
\begin{equation}\label{eq:inj-maps-from-les}
\mathrm{gr}(D_{R|k})_i \to \ker(J_{i,i-1})\,.\end{equation}
Under these identifications, the maps 
\[ D^i_{R|k} \to \ker(J_{i,i-1}) \quad \text{and} \quad \mathrm{gr}(D_{R|k})_i \to \ker(J_{i,i-1})\]
 as in \eqref{eq:left-exact-seq-diff} are simply given by restricting coordinates.

We describe the matrices for $\phi_i$ as block matrices via grouping the basis elements in the source and the target by the sums of their indices. Block components include the following:

\[ J_{2,0} \coloneqq \kbordermatrix{ & \partial_x^{(2)}& \partial_x \partial_y& \partial_x \partial_z& \partial_y^{(2)}& \partial_y \partial_z& \partial_z^{(2)} \\
e_1 & \frac 12 f_{xx}& f_{xy} & f_{xz}&\frac 12 f_{yy}&f_{yz}&\frac 12f_{zz}}\]

\[ J_{3,0} \coloneqq \kbordermatrix{ & \partial_x^{(3)}& \partial_x^{(2)} \partial_y& \partial_x^{(2)} \partial_z& \partial_x \partial_y^{(2)}& \partial_x \partial_y \partial_z&
\partial_x \partial_z^{(2)}& \partial_y^{(3)}& \partial_y^{(2)} \partial_z& \partial_y \partial_z^{(2)}& \partial_z^{(3)} \\
e_1 & \frac 16 f_{xxx}& \frac 12 f_{xxy} & \frac 12 f_{xxz}&\frac 12 f_{xyy}&f_{xyz}&\frac 12f_{xzz}&\frac 16 f_{yyy}&\frac 12 f_{yyz}&\frac 12 f_{yzz}&\frac 16 f_{zzz} } \]

\[ J_{3,1} \coloneqq \kbordermatrix{ & \partial_x^{(3)}& \partial_x^{(2)} \partial_y& \partial_x^{(2)} \partial_z& \partial_x \partial_y^{(2)}& \partial_x \partial_y \partial_z&
\partial_x \partial_z^{(2)}& \partial_y^{(3)}& \partial_y^{(2)} \partial_z& \partial_y \partial_z^{(2)}& \partial_z^{(3)} \\
e_x & \frac 12 f_{xx}& f_{xy} & f_{xz}&\frac 12 f_{yy}&f_{yz}&\frac 12f_{zz}&0&0&0&0\\
e_y & 0&\frac 12 f_{xx}&0&f_{xy}&f_{xz}&0&\frac 12 f_{yy}&f_{yz}&\frac 12f_{zz}&0\\
e_z & 0&0&\frac 12 f_{xx}&0&f_{xy}&f_{xz}&0&\frac 12 f_{yy}&f_{yz}&\frac 12f_{zz}}.\]

\medskip

Consider the following block matrices:\[ P_1\coloneqq \begin{bmatrix} J_{1,0} \end{bmatrix}, \quad P_2\coloneqq \begin{bmatrix} J_{1,0} & J_{2,0} \\ 0 & J_{2,1} \end{bmatrix}, \ \text{and} \ P_3\coloneqq \begin{bmatrix} J_{1,0} & J_{2,0} & J_{3,0} \\ 0 & J_{2,1} & J_{3,1} \\ 0 & 0 & J_{3,2} \end{bmatrix}.\]
Then for $i=1,2,3$, the matrix for $\phi_i$ is given by $P_i$ with an additional column of zeroes corresponding to the $R$-basis element $1 \in R\otimes_Q D^i_{Q|k}$. The column of zeroes corresponds to a free cyclic summand of $\phi_i$ with basis $1$, which corresponds to the image of $D^0_{R|k}$ in $D^i_{R|k}$.
\end{chunk}

\begin{remark}\label{r:Jnotation}
For a matrix $M$, we take $\mathrm{S}^i(M) = \mathrm{Sym}^i(M)$ to $i$th symmetric power of a matrix: i.e., the matrix minimally presenting the $i$th symmetric power of the cokernel of $M$.

Let $J$ be the usual Jacobian matrix. 
Observe that $J_{1,0} = J$, $J_{2,1} = \mathrm{S}^2(J^T)^T$, and $J_{3,2}= \mathrm{S}^3(J^T)^T$, and more generally $J_{i,i-1} = \mathrm{S}^i(J^T)^T$. In particular, we have
\[ \ker(J_{i,i-1}) = \ker(\mathrm{S}^i(J^T)^T) \cong \Hom_R(\coker \mathrm{S}^i(J^T),R) \cong \Hom_R(\mathrm{S}^i(\coker J^T),R)\cong \Hom_R(\mathrm{S}^i \Omega_R, R)\,.\]
\end{remark}
 
\begin{chunk}\label{c:composition} In light of the discussion above, we identify elements of $D^2_{R|k}$ as constant operators in $D^0_{R|k}$ plus the collection of vectors $v\in R^9$ such that $P_2 v=0$, that is, 
\[
D^2_{R|k} \cong R \oplus \ker P_2\,.
\]
Likewise, we identify elements of $D^3_{R|k}$ as constant operators in $D^0_{R|k}$ plus the collection of vectors $v\in R^{19}$ such that $P_3v=0$, that is, 
\[
D^3_{R|k} \cong R \oplus \ker P_3\,.
\]

As in \eqref{eq:induced-maps-ker}, for any derivation $\eta$, we obtain a well-defined map induced by composition with $\eta$:
\[ \ker(J) \to \ker(J_{2,1})\,.\]

 If the restriction map $D^2_{R|k} \to \ker(J_{2,1})$ is surjective, then we also obtain a well-defined map induced by composition with $\eta$: 
 \[ \ker(J_{2,1}) \to \ker(J_{3,2})\,.\]
 \end{chunk}

%%%%%%%%%%%%%%%%%%%%%%%%%%%%%%%%%%%%
\subsection{Koszul complexes and matrix factorizations}\label{subsec:koszulmf}
%%%%%%%%%%%%%%%%%%%%%%%%%%%%%%%%%%%%
In this subsection we make use of differential graded (henceforth abbreviated to dg) algebras and their dg modules. A suitable reference for background on these topics is \cite{Avramov:1999}.
\begin{chunk}
\label{c:koszulcx}
Let $\f=f_1,\ldots,f_n$ be a list of elements in $Q$. We write $\Kos^Q(\f)$ for the Koszul complex on $Q$ over $\f$, regarded as a dg $Q$-algebra in the usual way. That is, $\Kos^Q(\f)$ is the exterior algebra on the free module $\bigoplus_{i=1}^nQe_i$, where $e_i$ has homological degree one, and differential which is completely determined by $\del(e_i)=f_i$ and the Leibniz rule. 

Of particular interest will be the case in which $n=3,$ where we order the bases of homological degrees 1 and 2 according to 
\[\Kos^Q_1(\f)=Qe_1\oplus Qe_2\oplus Qe_3 \ \text{ and } \ \Kos^Q_2(\f)=Qe_3 e_2 \oplus Qe_1e_3\oplus Qe_2e_1\,.\]
With this convention the differentials in $\Kos^Q(\f)$ can be expressed using the following matrices:
\[
0\to Q\xra{\begin{bmatrix}
f_1 \\ f_2 \\f_3 
\end{bmatrix}}
Q^3\xra{\begin{bmatrix}
0 & -f_3 & f_2\\ f_3 & 0 & -f_1 \\-f_2 & f_1 & 0
\end{bmatrix}}
Q^3\xra{\begin{bmatrix}
f_1& f_2 & f_3
\end{bmatrix}}Q \to 0\,.
\] 
\end{chunk}

\begin{notation}
\label{notation:koszul}
Consider the Koszul complexes $\Kos^Q(x,y,z)$ and $\Kos^Q(f_x,f_y,f_z).$ The differential of the first Koszul complex is denoted by $\del$ and the differential of the second Koszul complex is written as $D.$
Explicitly, 
\begin{align*}
 \del_2&=\begin{bmatrix}
0 & -z & y\\ 
z & 0 & -x \\
-y & x & 0
\end{bmatrix} & &
\del_1=\del_3^T=\begin{bmatrix}
x& y & z
\end{bmatrix} \\
 D_2&=\begin{bmatrix}
0 & -f_z & f_y\\ 
f_z & 0 & -f_x \\
-f_y & f_x & 0
\end{bmatrix} & &
D_1=D_3^T=\begin{bmatrix}
f_x& f_y & f_z
\end{bmatrix}\,.
\end{align*}
\end{notation}

\begin{remark}
\label{r:koscxres}
The assumption that $R$ has an isolated singularity implies $f_x,f_y,f_z$ forms a $Q$-regular sequence; this is explained at the beginning of \cref{subsec:d1}. Hence, $\Kos^Q(f_x,f_y,f_z)$ is a dg $Q$-algebra resolution of $Q/(f_x,f_y,f_z)$. Furthermore, $\Kos^Q(x,y,z)$ is a dg $Q$-algebra resolution of $k$.
\end{remark}

\begin{chunk}
\label{c:Ekoszul}
Throughout we will consider the dg $Q$-algebra $E=\Kos^Q(d\cdot f)$, with $e$ the degree one generator.
Since $f$ is a $Q$-regular element, the augmentation map \[E\xra{\simeq} Q/(d\cdot f)=R\] is a quasi-isomorphism of dg $Q$-algebras. 

Next, the Euler identity $d\cdot f=xf_x+yf_y+zf_z$ in $Q$ defines dg $E$-module structures on $\Kos^Q(f_x,f_y,f_z)$ and $\Kos^Q(x,y,z)$. 
Left multiplication by $e$ on $\Kos^Q(f_x,f_y,f_z)$ is described by
the following matrices
\[
 0\leftarrow Q\xleftarrow{\del_1}
Q^3\xleftarrow{-\del_2 }
Q^3\xleftarrow{\del_3}Q \leftarrow 0\,;
\]
recall $e$ has homological degree 1 and so left multiplication by $e$ increases homological degree by 1. In fact, $e\cdot$ is a square-zero nullhomotopy for multiplication by $d\cdot f$ on $\id^{\Kos^Q(f_x,f_y,f_z)},$ which is exactly the data of a dg $E$-module; see \cite[Remark~2.2.1]{Avramov:1999}. 

Similarly, left multiplication by $e$ on $\Kos^Q(x,y,z)$ is described by
the following matrices
\[
 0\leftarrow Q\xleftarrow{D_1}
Q^3\xleftarrow{-D_2 }
Q^3\xleftarrow{D_3}Q \leftarrow 0\,.
\]
When we refer to $\Kos^Q(f_x,f_y,f_z)$ and $\Kos^Q(x,y,z)$ as dg $E$-modules it is with the structures prescribed above, respectively. The fact that left multiplication by $e$ on these complexes define square-zero nullhomotopies for multiplication by $d\cdot f$ can be checked from \cref{appendix}.
\end{chunk}

In \cite{Eisenbud:1980}, Eisenbud showed that any $R$-module's minimal free resolution is eventually two-periodic; this data can be described completely in terms of his theory of matrix factorizations as introduced in \emph{loc.\@ cit.} We briefly recall these points and their connections with certain triangulated categories associated to $R$. For ease of exposition we describe the former over regular rings; see \cite{Eisenbud:1980,Leuschke/Wiegand:2012,Yoshino:1990} for more general treatments of this topic. 

For the rest of the subsection, let $A$ be a regular ring and $a\in A$ be an $A$-regular element.

\begin{chunk}
\label{c:mf}
 A \emph{matrix factorization of $a$ (over $A$)} is a pair of maps between finite rank free $A$-modules
\[ F_0\xra{\alpha} F_1 \xra{\beta} F_0\]
such that $\alpha\beta = a\cdot \id_{F_1}$ and $\beta\alpha = a\cdot\id _{F_0}$. It follows that $F_0,F_1$ have the same rank and so $\alpha$ and $\beta$ can be represented by square matrices of the same size.

The collection of matrix factorizations of $a$ (over $A$), with morphisms defined in the obvious way, form a Frobenius exact category. Hence its stable category, denoted $[\mathsf{mf}(A,a)]$, is a triangulated category where the suspension functor is given by 
\[
\shift(F_0\xra{\alpha} F_1 \xra{\beta} F_0)\coloneqq F_1\xra{-\alpha} F_0 \xra{-\beta} F_1\,. 
\]
\end{chunk}

\begin{chunk}
\label{c:fold}
Set $K \coloneqq \Kos^A(a)$. Given a bounded complex $G$ of finite rank free $A$-modules 
\[ 
0\to G_n \to G_{n-1} \to \dots \to G_m \to 0
\]
with a dg $K$-module structure, one can define a matrix factorization of $a$ over $A$ as follows:
Set 
\[
F_0 \coloneqq \bigoplus_{i \text{ even}} G_i, \quad F_1 \coloneqq \bigoplus_{i \text{ odd}} G_i,
\]
and, slightly abusing notation, set $\alpha$ and $\beta$ to be $\del^G + \sigma$ where $\sigma$ denotes left multiplication by the exterior generator of $K$. This matrix factorization of $a$ is denoted by $\fold(G)$.
\end{chunk}

\begin{chunk}\label{c:hypersurface}
In \cite[Section~6]{Eisenbud:1980}, the machinery above is applied to show that every minimal free resolution over a hypersurface ring is eventually two-periodic. A fundamental, yet key idea behind this is that given a matrix factorization $F_0\xra{\alpha} F_1 \xra{\beta} F_0$ of $a$ the complex
\[
\cdots \to F_0\otimes_A A/(a)\xra{\alpha\otimes 1} F_1\otimes_A A/(a) \xra{\beta\otimes 1} F_0\otimes_A A/(a)\xra{\alpha\otimes 1}\cdots 
\]
is exact; see \cite[Proposition~5.1]{Eisenbud:1980} or \cite{Shamash:1969}.
\end{chunk}

\begin{chunk}
\label{c:equivalence}
For a commutative noetherian ring $B$, we let $\Db(B)$ denote the bounded derived category of finitely generated $B$-modules. The objects are those $B$-complexes whose homology is concentrated in finitely many degrees and each homology module is finitely generated over $B$. This is a triangulated category in the standard way; see \cite{Krause:2022} and details on thick subcategories. 

We let $\Dsg(B)$ denote the Verdier quotient of $\Db(B)$ by the perfect $B$-complexes. This was independently introduced by Buchweitz in \cite{Buchweitz} and Orlov in \cite{Orlov}. This is a triangulated category that records the singularity of $B$, and more generally the asymptotics of free resolutions over $B$. 
Assume $B=A/(a)$ a singular hypersurface ring. By \emph{loc.\@ cit.\@}, it follows that there is an equivalence of triangulated categories
\[
\Dsg(B)\equiv [\mathsf{mf}(A,a)]\,.
\]
\end{chunk}

\begin{chunk}
\label{c:level}
Let $\mathsf{T}$ be a triangulated category. Here we recall the notion of level from \cite{Avramov/Buchweitz/Iyengar/Miller:2010a}; this is an invariant that records how many ``steps" it takes to build on object in $\mathsf{T}$ from another only utilizing the triangulated structure of $\mathsf{T}.$ 

Recall that a thick subcategory of $\mathsf{T}$ is a triangulated subcategory that is closed under retracts. For an object $X$ in $\mathsf{T}$, we let $\thick_\mathsf{T}X$ denote the smallest thick subcategory of $\mathsf{T}$ containing $X$.
There is an inductive construction of $\thick_{\mathsf{T}} X$ from \cite{Bondal/VanDenBergh:2003}, see also \cite{Avramov/Buchweitz/Iyengar/Miller:2010a}.
 Set $\thick_{\mathsf{T}}^1 X$ to be the smallest full subcategory of $\mathsf{T}$ containing $X$ and closed under (de-)suspensions, finite sums, and retracts. Now inductively, $\thick_{\mathsf{T}}^n X$ is the smallest full subcategory of $\mathsf{T}$ containing objects $C$ fitting into an exact triangle 
\[
A\to B\to C\to \,
\] with $A$ in $\thick_{\mathsf{T}}^{n-1}X$ and $B$ in $\thick_{\mathsf{T}}^1X$, which is closed under (de-)suspensions, finite sums, and retracts.
By \cite[Section~2.2]{Avramov/Buchweitz/Iyengar/Miller:2010a}, 
\[
\thick_{\mathsf{T}} X=\cup_{n\geqslant 1}\thick_{\mathsf{T}}^n X\,.
\]
Following \cite{Avramov/Buchweitz/Iyengar/Miller:2010a}, the \emph{level of an object $Y$ with respect to $X$} is 
\[
\level_{\mathsf{T}}^XY=\inf\{n\geqslant 0: Y\text{ is in} \thick_\mathsf{T}^n X\}\,.
\]
Now assume $B$ is an isolated singularity hypersurface ring with residue field $k$. It is well known that $\level_{\Dsg(R)}^kM<\infty$ for any $M$ in $\Dsg(R)$; see, for example, \cite{Dyckerhoff:2011,Keller/Murfet/VandenBergh:2011}. We give bounds on $\level_{\Dsg(R)}^kD^i_{R|k}$ for $i=2,3$ in \cref{cor:levels2,cor:levels3}, respectively. See also \cref{q:levels} for a proposed upper bound on $\level_{\Dsg(R)}^kD^i_{R|k}<\infty$ for arbitrary $i$. 
\end{chunk}

%%%%%%%%%%%%%%%%%%%%%%%%%%%%%%%%%%%%%%%%%%%%%%%%%%%%%%%%%%%%%%
\subsection{First order differential operators}\label{subsec:d1}
%%%%%%%%%%%%%%%%%%%%%%%%%%%%%%%%%%%%%%%%%%%%%%%%%%%%%%%%%%%%%%

Let $P=k[x_1,\ldots,x_n]$ and $A=P/(f)$ where $k$ is a perfect field and $f$ is a homogeneous element of degree $d\geqslant 2$. We also impose that $A$ is an isolated singularity. Thus, $(f_{x_1},\ldots,f_{x_n},f) = (f_{x_1},\ldots,f_{x_n})$ has height $n$, so the partial derivatives form a $P$-regular sequence. Set $B\coloneqq A/(f_{x_1},\ldots,f_{x_n})$, which is an artinian complete intersection quotient of $A$. 
In this subsection we recall the relationship between $B$ and the first order differential operators on $A$, as well as the minimal $A$-free resolution of $B$. We refer the reader to \cref{resD1} for the 3-variable case that examined in this paper. 

\begin{chunk}
\label{c:d1andB}
We recall from \ref{c:splitting-order-ses} that $D^1_{A|k} = A\oplus \Der_{A|k}$, where the copy of $A$ is the set of constant operators, and $\Der_{A|k}$ is the collection of derivations; in particular $\Der_{A|k}\cong [\mathrm{gr}(D_{A|k})]_1$.

As in \ref{chunk:diagram-dops-and-gr}, we have 
\[ \Der_{A|k} \cong \ker(J) = \ker \begin{bmatrix} f_{x_1} & \cdots & f_{x_n} \end{bmatrix}.\]
To find the generators and minimal free resolution of this $R$-module, we construct a minimal free resolution of $B=\coker J$
\[
\cdots \to F_3\xra{\del_3^F} F_2\xra{\del_2^F} R^n\xra{J} R\to 0
\]
and adding a contractible subcomplex $R\xra{=} R$ in degrees 1 and 2 yields $\Der_{A|k}$ as the image of the second differential. That is, in
\begin{equation}
\label{e:cokernel}
\xymatrixrowsep{1pc}
\xymatrixcolsep{1pc}
\xymatrix{
\cdots 
\ar@{->}[rr]^{}
&&F_3
\ar@{->}[rr]^-{\begin{bmatrix}
\del_3^F\\ 0
\end{bmatrix}}
&&F_2\oplus R
\ar@{->}[rr]^-{\begin{bmatrix}
\del_2^F & 0\\
0& 1
\end{bmatrix}}
\ar@{->>}[dr]_-{\del}
&&R^n\oplus R
\ar@{->}[rr]^-{J}
&&R
\ar@{->>}[dr]_{\pi}
&
\\
&&&&&
\Der_{A|k}
\ar@{->}[ur]_-{i}
&&&&B
}
\end{equation}
we obtain the generators of $\Der_{A|k}$ from the image of $\begin{bmatrix}
\del_2^F & 0\\
0& 1
\end{bmatrix}$, and the minimal free resolution of $\Der_{A|k}$ is procured by truncating the complex in \cref{e:cokernel} at homological degree two and shifting by two.
\end{chunk}

Next we explain how to obtain the minimal $R$-free resolution of $B$, and hence, of $D^1_{A|k}$ in light of \cref{c:d1andB}.
\begin{chunk}
\label{c:koszulresolutions}
Let $C=\Kos^B(d\cdot f)$ be the Koszul complex on the homogeneous element $d\cdot f$ over $P$, where recall $d$ is the degree of $f$. 
Since $f_{x_1},\ldots,f_{x_n}$ is a $P$-regular sequence, 
\[
K\coloneqq \Kos^P(f_{x_1},\ldots,f_{x_n})
\]
is a $P$-free resolution of $B$. Also, set 
\[
K'\coloneqq \Kos^P(x_1,\ldots,x_n)
\]
which is a $P$-free resolution of the residue field $k$. In fact, $K$ and $K'$ are dg $C$-algebras. 
Indeed, let the standard bases of $K_1$ and $K'_1$ be denoted $\xi_1,\ldots,\xi_n$ and $\xi_1',\ldots,\xi_n'$, respectively. 
The Euler identity in $P$,
\[
d\cdot f=\sum_{i=1}^n x_if_{x_i}
\]
defines dg $C$-module structures on $K$ and $K'$: Namely, if $\xi$ is exterior variable of $C$ then left multiplication by $\xi$ is given by left multiplication by
\begin{equation}
 \label{e:tate}
\sum_{i=1}^n x_i\xi_i \text{ on }K\quad \text{and}\quad \sum_{i=1}^n f_{x_i}\xi_i' \text{ on }K'\,,
\end{equation}
respectively. As a consequence, using \cite[Theorem~4]{Tate:1957}, the minimal $A$-free resolutions of $B$ and $k$ are given by 
\[
A \otimes_P K\langle y\mid \del y=\sum_{i=1}^n x_i\otimes\xi_i\rangle\quad \text{and}\quad A \otimes_P K'\langle y\mid \del y=\sum_{i=1}^n f_{x_i}\otimes\xi_i'\rangle\,,
\]
respectively; as underlying graded $A$-modules these have the form
\[
A\otimes_P K\otimes_P \bigoplus_{i\geqslant 0} Py^{(i)}\quad \text{and} \quad A\otimes_P K'\otimes_P \bigoplus_{i\geqslant 0}Py^{(i)}\,,
\]
respectively, where $y^{(i)}$ is a basis element living in homological degree $2i$ and the differential on each complex is uniquely extended by the rule specified and the formula $\del(y^{(i)})=\del y \cdot y^{(i-1)}$. 
See \cite{Avramov:2010} for a detailed account of this construction. Each of these is also the Shamash resolution from \cite{Shamash:1969} (see also \cite{Eisenbud:1980}). Note that upon fixing the lexicographically ordered bases on $K$ and $K'$, respectively, left multiplication by $\sum_{i=1}^n x_i\xi_i$ on $K$ is given by the appropriate matrix $\del^{K'}$ representing the differential of $K'$. Similarly, 
left multiplication by $\sum_{i=1}^n f_{x_i}\xi_i'$ on $K'$ is given by the appropriate matrix $\del^{K}$ representing the differential of $K$.
\end{chunk}

\begin{chunk}
Continue with the notation set in \cref{c:koszulresolutions}. By inspection of the resolutions in \cref{e:tate}, minding that $y$ is a divided power variable of degree two, the matrix factorizations describing the minimal $A$-resolutions of $B$ and $k$, respectively, have the following forms:
\[
P^{2^{n-1}}\xra{\alpha}P^{2^{n-1}}\xra{\beta}P^{2^{n-1}}\quad\text{and}\quad 
P^{2^{n-1}}\xra{\beta}P^{2^{n-1}}\xra{\alpha}P^{2^{n-1}}\,,
\]
respectively, where 
\[
\alpha=
\begin{bmatrix}
 (\del_1^{K'})^T & \del_2^K & 0 & 0 & \ldots \\
0 & (\del_3^{K'})^T & \del_4^K & 0& \ldots \\
0& 0 & (\del_5^{K'})^T & \del_6^K & \ddots \\
\vdots & \ddots & \ddots & \ddots & \ddots 
\end{bmatrix}
\quad \text{and}\quad 
\beta=\begin{bmatrix}
 \del_1^K & 0& 0 & 0 & \ldots \\
(\del_2^{K'})^T & \del_3^K & 0& 0& \ldots \\
0& (\del_4^{K'})^T & \del_5^K &0 & \ddots \\
\vdots & \ddots & \ddots & \ddots & \ddots 
\end{bmatrix}\,;
\]
here the self-duality of the Koszul complex is being used. Therefore, there are the following isomorphisms in $\Dsg(A):$
\[
D^1_{A|k}\simeq B\simeq \shift k\,.
\]
\end{chunk}
\begin{remark}
\label{r:hz}
In \cite{Herzog/Martsinkovsky:1993}, Herzog and Martsinkovsky show that $D_{A|k}^1\simeq \shift k$ in $\Dsg(A)$ whenever $A$ is an isolated singularity complete intersection ring of positive Krull dimension. The strategy employed in \emph{loc.\@ cit.\@} is the so-called ``gluing" method to produce complete resolutions; this differs from the dg-method described above.  We opt for the latter as it keeps the present discussion mostly self-contained and in line with the strategy employed in \cref{sec:D2,sec:D3}. 

Note that in the case that $A$ is an isolated singularity complete intersection ring, we trivially have \[\level_{\Dsg(A)}^k (D_{A|k}^1)=1\] from the isomorphism above. This equality
should be compared with \cref{cor:levels2,cor:levels3} for $D^2_{A|k}$ and $D^3_{A|k}$ when $A$ is a hypersurface of Krull dimension two; cf.\@ \cref{q:levels} as well. 
\end{remark}

\begin{chunk}\label{resD1}
Now returning to \cref{notation:main}, following \cref{c:koszulresolutions}, we obtain the following minimal $R$-free resolution of $B=R/(f_x,f_y,f_z)$: 
\[
\cdots \xra{\small\begin{bmatrix}
\del_3 & D_2 \\
0 & \del_1
\end{bmatrix}}R^4\xra{\small\begin{bmatrix}
D_1 & 0 \\
-\del_2 & D_3
\end{bmatrix}}R^4\xra{\small\begin{bmatrix}
\del_3 & D_2 \\
0 & \del_1
\end{bmatrix}}R^4\xra{\small\begin{bmatrix}
D_1 & 0 \\
-\del_2 & D_3
\end{bmatrix}}R^4\xra{\small\begin{bmatrix} 
\del_3 & D_2
\end{bmatrix}}R^3\xra{\small\begin{bmatrix}
f_x & f_y & f_z
\end{bmatrix}}R\to 0\,.
\]
where the matrices are defined in \cref{notation:koszul}. 
Therefore, again using \cref{c:d1andB}, the (augmented) minimal $R$-free resolution of $D^1_{R|k}$ has the form 
\begin{equation*}
\label{e:cokernel3}
\xymatrixrowsep{1pc}
\xymatrixcolsep{1pc}
\xymatrix{
\cdots 
\ar@{->}[rrr]^{\small\begin{bmatrix}
\del_3 & D_2 \\
0 & \del_1
\end{bmatrix}}
&&&R^4
\ar@{->}[rrr]^{\small\begin{bmatrix}
D_1 & 0 \\
-\del_2 & D_3
\end{bmatrix}}
&&&R^4
\ar@{->}[rrr]^{\small\begin{bmatrix}
\del_3 & D_2 \\
0 & \del_1
\end{bmatrix}}
&&&R^4
\ar@{->}[rrr]^{\small\begin{bmatrix}
D_1 & 0 \\
-\del_2 & D_3\\
0 & 0
\end{bmatrix}}
&&&R^4\oplus R
\ar@{->}[rr]^{ \ \ \small\begin{bmatrix}
\del_3 & D_2 & 0 \\
0 & 0 & 1
\end{bmatrix}\ \ \ }
\ar@{->>}[dr]_{}
&&R^3\oplus R
\ar@{->}[rrr]^{\ \ \ \small\begin{bmatrix}
f_x \!&\! f_y \!&\! f_z \!&\!0
\end{bmatrix}}
&&&R
\ar@{->>}[dr]_{\pi}
&
\\
&&&&&&&&&&&&&
D^1_{R|k}
\ar@{->}[ur]^{i}
&&&&&B
}
\end{equation*}
and a minimal set of generators $D^1_{R|k}$ is given by the columns of 
$\begin{bmatrix}
\del_3 & D_2 & 0 \\
0 & 0 & 1
\end{bmatrix}$, where the rows correspond to the basis $\del_x,\del_y,\del_z,1$ as described in \cref{chunk:concrete1}. 
The following notation for the matrices above will be used in the sequel. 
\[
M_0(1)=\begin{bmatrix} 
\del_3 & D_2
\end{bmatrix}\,,\quad 
M_1(1)=\begin{bmatrix}
D_1 & 0 \\
-\del_2 & D_3
\end{bmatrix}\,,\quad 
M_2(1)=\begin{bmatrix}
\del_3 & D_2 \\
0 & \del_1
\end{bmatrix}\,.
\]

Recall that $D^1_{R|k} = R\oplus \Der_{R|k}$. 
In particular, a minimal set of generators of $\Der_{R|k}$ is given by the columns of $\begin{bmatrix} 
\del_3 & D_2 \end{bmatrix}$. We name these derivations:
\begin{equation}
\label{E}
\E \coloneqq x \del_x + y \del_y + z \del_z
\end{equation}
is the Euler derivation, and
\begin{equation}
\label{H}
 \H_{yz} \coloneqq f_z \del_y - f_y \del_z , \quad
 \H_{zx} \coloneqq f_x \del_z - f_z \del_x, \quad \text{and} \quad
 \H_{xy} \coloneqq f_y \del_x - f_x \del_y 
 \end{equation}
are the Hamiltonian derivations. The relations on these derivations are given by the columns of $\begin{bmatrix}
D_1 & 0 \\
-\del_2 & D_3
\end{bmatrix}$, or more explicitly as 
\begin{align}
\begin{aligned}
\label{e:basicrelations}
f_x \E + y \H_{xy} - z \H_{zx} &= 0 \\
f_y \E - x \H_{xy} + z \H_{yz} &= 0 \\
f_z \E + x \H_{zx} - y \H_{yz} &= 0\,.
\end{aligned}
 \end{align}
\end{chunk}

As previously mentioned, a similar method will be employed in \cref{sec:D2,sec:D3} for calculating the minimal free resolutions of $D^2_{R|k}$ and $D^3_{R|k}$. We end this subsection by recording the graded Betti numbers of $D^1_{R|k}$ regarded as a graded module over the standard $k$-algebra $R$. 

\begin{chunk}
Regarding $R$ as a standard graded $k$-algebra, it follows that each $D^i_{R|k}$ is a graded $R$-submodule of the graded $k$-linear endomorphisms of $R$. For a $k$-linear graded endomomorphism $g$ of $R$, we let $g$ denote the degree of $g.$ That is, $|g|$ is the unique integer satisfying $g(R_j)\subseteq R_{j+|g|}$ for each $j\in \mathbb{Z}.$ Under these conventions, 
\[
|\E|=0\quad \text{ and }\quad |\H_{yz}|=|\H_{zx}|=|\H_{xy}|=d-2
\]
using that $|\del_x|=|\del_y|=|\del_z|=-1$. 
\end{chunk}

\begin{chunk}
\label{c:d1graded}
Let $N$ be a finitely generated graded $R$-module. Its $i,j^{\text{th}}$-graded Betti number is 
\[
\beta_{ij}^R(N)\coloneqq \rank_k\Tor^R_i(N,k)_j\,;
\]
that is, $\beta_{i,j}^R(N)$ is the rank of the
free module in homological degree $i$ and basis in internal degree $-j$.

It is straightforward to check that the differentials in the free resolution of $D^1=D^1_{R|k}$, in \cref{resD1}, are homogeneous with respect to the internal grading of $R$. Therefore, the graded Betti numbers of $D^1$ are
\[
\beta_{0,j}^R(D^1)=\begin{cases}
2 & j=0\\
3 & j=d-2\\
0 & \text{otherwise}
\end{cases}\,,
\]
and for $n\geqslant 1$, 
\[
\beta_{2n-1,j}^R(D^1)=\begin{cases}
1 & j=nd-1\\
3 & j=nd+d-3\\
0 & \text{otherwise}
\end{cases}\quad \text{and}\quad \beta_{2n,j}^R(D^1)=\begin{cases}
3 & j=nd\\
1 & j=nd+d-2\\
0 & \text{otherwise}
\end{cases}\,.
\]
\end{chunk}

%%%%%%%%%%%%%%%%%%%%%%%%%%%%%%%%%%%%%%%%%%%%%%%%%%%%%%%%%%%%%%
\section{Glossary of matrices}\label{sec:glossary}
%%%%%%%%%%%%%%%%%%%%%%%%%%%%%%%%%%%%%%%%%%%%%%%%%%%%%%%%%%%%%%
All of the following are matrices with entries in $R$; see \cref{notation:main}. When writing a block matrix, $0_{m\times n}$ will denote the $m\times n$-matrix whose entries are all zero.
\subsection{Matrices needed in \cref{sec:D2}}
\[
\del_1=\begin{bmatrix}
x& y & z
\end{bmatrix} 
\quad \quad 
\del_2=\begin{bmatrix}
0 & -z & y\\ 
z & 0 & -x \\
-y & x & 0
\end{bmatrix} \quad\quad
\del_3=\begin{bmatrix}
x\\ y \\ z
\end{bmatrix} 
\]
\[
D_1=\begin{bmatrix}f_x& f_y & f_z
\end{bmatrix} 
\quad \quad 
D_2=\begin{bmatrix}
0 & -f_z & f_y\\ 
f_z & 0 & -f_x \\
-f_y & f_x & 0
\end{bmatrix} \quad\quad
D_3=\begin{bmatrix}
f_x\\ f_y \\ f_z
\end{bmatrix} 
\]
\[
q=\begin{bmatrix}
x^2 & xy & xz \\
xy &{y^2} & {yz} \\
{xz} & {yz} & {z^2} \\
\end{bmatrix}\quad \quad
\Delta=\begin{bmatrix}
f_{xx}& f_{xy} &f_{xz}\\ 
f_{xy} & f_{yy}& f_{yz}\\
f_{xz}&f_{yz}& f_{zz}
\end{bmatrix}\quad \quad \delta=\det(\Delta)
\]
\[
\Delta_{xx}=f_{yy} f_{zz}-f^2_{yz} \quad\quad\Delta_{yy}=f_{xx} f_{zz}-f^2_{xz}\quad\quad\Delta_{zz}=f_{xx} f_{yy}-f^2_{xy}
\]
\[
\Delta_{xy}=f_{xz}f_{yz}-f_{xy}f_{zz}\quad\quad\Delta_{xz}=f_{xy} f_{yz}-f_{xz}f_{yy}\quad\quad\Delta_{yz}=f_{xy}f_{xz}-f_{xx} f_{yz}
\]
\[ \delta_x = \partial_x(\delta) \qquad\qquad\qquad \delta_y = \partial_y(\delta)\qquad\qquad\qquad \delta_z = \partial_z(\delta)\]
\[
\alpha_1=\frac{1}{d-1}\Delta \quad \quad \alpha_2=\frac{1}{(d-1)^2}\begin{bmatrix}\Delta_{xx}& \Delta_{xy} &\Delta_{xz}\\ 
\Delta_{xy} &\Delta_{yy}& \Delta_{yz}\\
\Delta_{xz}&\Delta_{yz}& \Delta_{zz}
\end{bmatrix} \quad \quad
\alpha_3=\frac{1}{(d-1)^3}\delta \quad 
\]
\[
\theta_0(2)= J_{2,0}= \begin{bmatrix}
\frac 12 f_{xx}& f_{xy} & f_{xz}&\frac 12 f_{yy}&f_{yz}&\frac 12f_{zz}
\end{bmatrix} 
\quad \quad 
\theta_1(2) = (d-1)\begin{bmatrix} 0_{3\times 4}&\alpha_2 \end{bmatrix}
\]
\begin{align*}
\theta_{2i}(2)&= (d-1)\begin{bmatrix} 0_{1\times 3}&0_{1\times3}&-\alpha_3\\ 
0_{3\times 3}&\alpha_1&0_{3\times 1} \end{bmatrix}~\text{for}~ i \geqslant 1\\
\theta_{2i+1}(2)&= \frac{-(d-1)}{2}\begin{bmatrix} 0_{3\times 1}&0_{3\times 3}&-2\alpha_2\\ 
0&\partial_1&0_{1\times 3} \end{bmatrix}~\text{for}~ i \geqslant 1
\end{align*}
\subsection{Additional matrices needed in \cref{sec:D3}}\label{sec:glossary2}
\[\sigma_1=\sigma_3^T=\frac{1}{(d-1)^3(d-2)}\begin{bmatrix}
\delta_{x} & \delta_{y} & \delta_{z}
\end{bmatrix}\quad\quad
\sigma_2=\frac{1}{(d-1)^3(d-2)}\begin{bmatrix}
0 & -\delta_z & \delta_y\\
\delta_z &0 &- \delta_x \\
-\delta_y & \delta_x & 0 
\end{bmatrix}
\]

\[B_1=\frac{1}{d-1}
\begin{bmatrix} 
x\Delta_{xx}& y\Delta_{xx}& z\Delta_{xx}\\
x\Delta_{xy}& y\Delta_{xy}& z\Delta_{xy}\\
x\Delta_{xz}& y\Delta_{xz}& z\Delta_{xz}\\
x\Delta_{yy}& y\Delta_{yy}& z\Delta_{yy}\\
x\Delta_{yz}& y\Delta_{yz}& z\Delta_{yz}\\
x\Delta_{zz}& y\Delta_{zz}& z\Delta_{zz}
\end{bmatrix}
\]

\[B_2=\scriptsize{
\begin{bmatrix} H_{yz}(\Delta_{xx})& H_{yz}(\Delta_{xy})-\frac{x\delta_z}{d-1}&H_{yz}(\Delta_{xz})+\frac{x\delta_y}{d-1}\\ 
H_{zx}(\Delta_{xx})&H_{yz}(\Delta_{yy})&\frac{1}{2}\left( H_{yz}(\Delta_{yz})+H_{zx}(\Delta_{xz})+\frac{y\delta_y-x\delta_x}{d-1}\right)\\
H_{xy}(\Delta_{xx})&\frac{1}{2}\left( H_{yz}(\Delta_{yz})+H_{xy}(\Delta_{xy})+\frac{x\delta_x-z\delta_z}{d-1}\right)&H_{yz}(\Delta_{zz})\\
H_{zx}(\Delta_{xy})+\frac{y\delta_z}{d-1}&H_{zx}(\Delta_{yy}&H_{zx}(\Delta_{yz})-\frac{y\delta_x}{d-1}\\
\frac{1}{2}\left( H_{zx}(\Delta_{xz})+H_{xy}(\Delta_{xy})+\frac{z\delta_z-y\delta_y}{d-1}\right)&H_{xy}(\Delta_{yy})&H_{zx}(\Delta_{zz})\\
H_{xy}(\Delta_{xz})-\frac{z\delta_y}{d-1}&H_{xy}(\Delta_{yz})+\frac{z\delta_x}{d-1}&H_{xy}(\Delta_{zz}) 
\end{bmatrix}}
\]
\[
\theta_0(3)= \begin{bmatrix} J_{3,0} \\ J_{3,1} \end{bmatrix}
= \begin{bmatrix}
\frac 16 f_{xxx}& \frac 12 f_{xxy} & \frac 12 f_{xxz}&\frac 12 f_{xyy}&f_{xyz}&\frac 12f_{xzz}&\frac 16 f_{yyy}&\frac 12 f_{yyz}&\frac 12 f_{yzz}&\frac 16 f_{zzz}\\
\frac 12 f_{xx}& f_{xy} & f_{xz}&\frac 12 f_{yy}&f_{yz}&\frac 12f_{zz}&0&0&0&0\\
0&\frac 12 f_{xx}&0&f_{xy}&f_{xz}&0&\frac 12 f_{yy}&f_{yz}&\frac 12f_{zz}&0\\
0&0&\frac 12 f_{xx}&0&f_{xy}&f_{xz}&0&\frac 12 f_{yy}&f_{yz}&\frac 12f_{zz}
\end{bmatrix} 
\]
\[
\theta_1(3)=\begin{bmatrix}
0_{3\times 4 } & -\frac{d^2-1}{2}\alpha_2 & -(d-1)(d-2)\sigma_2\\
0_{6\times 4} & B_1 & \frac{-3}{(d-1)(d-2)}B_2
\end{bmatrix}
\]

\[
\theta_{2i}(3)=\begin{bmatrix}
0_{1\times 3} & 0_{1\times 3} & -(d-1)(d-2)\sigma_1 & 0 \\
0_{3\times 3} & \frac{-(d+1)}{2}\Delta & 0_{3\times 3} & -(d-1)(d-2)\sigma_3\\
0_{1\times 3} & 0_{1\times 3} & -3(d-1)\sigma_1 & 0 \\
0_{3\times 3} & 0_{3\times 3} & 0_{3\times 3} & -3(d-1)\sigma_3\\
0_{3\times 3} & 0_{3\times 3} & \frac{9}{2}\Delta & 0_{3\times 1}
\end{bmatrix}\quad\text{for}\quad i\geqslant 1
\]
\[
\theta_{2i+1}(3)=\begin{bmatrix}
0_{3\times 1} & 0_{3\times 3} & \frac{-(d^2-1)}{2}\alpha_2 & -(d-1)(d-2)\sigma_2 \\
0 & \frac{d^2-1}{4}\del_1 & 0_{1\times 3} & 0_{1\times 3} \\
0_{3\times 1} & 0_{3\times 3} & 0_{3\times 3} & -3(d-1) \sigma_2 \\
0_{3\times 1} & 0_{3\times 3} & 0_{3\times 3} & \frac{9(d-1)}{2}\alpha_2 \\
0 & 0_{1\times 3} & \frac{-3(d-1)}{2}\del_1 & 0_{1\times 3}
\end{bmatrix}\quad\text{for}\quad i\geqslant 1
\]

%%%%%%%%%%%%%%%%%%%%%%%%%%%%%%%%%%%%%%%%%%%%%%%%%%%%%%%%%%%%%%
\section{Differential operators of order 2}\label{sec:D2}
%%%%%%%%%%%%%%%%%%%%%%%%%%%%%%%%%%%%%%%%%%%%%%%%%%%%%%%%%%%%%%
In this subsection we adopt \cref{notation:main}, and construct the minimal $R$-free resolution of $D^2_{R|k}$; see \cref{resolutionD2}. The bulk of the work is in \cref{subsec:gensS2,subsec:mfS2,subsec:exactnessS2}, and it is assembled in \cref{subsec:liftS2}. 
%%%%%%%%%%%%%%%%%%%%%%%%%%%%%%%%%%%%
\subsection{A set of generators of $\ker J_{2,1}$}\label{subsec:gensS2}
%%%%%%%%%%%%%%%%%%%%%%%%%%%%%%%%%%%%
In this subsection we introduce what we later show, see \cref{gensS2}, to be the generators of $\ker J_{2,1}$.
We have $\Der_{R|K}=R \langle \E, \H_{yz},\H_{zx},\H_{xy}\rangle$, where $\E=x\partial_x+y\partial_y+z\partial_z$ is the Euler operator and $\H_{yz} = f_z \partial_y - f_y \partial_z$, $\H_{zx}=f_x \partial_z - f_z \partial_x$, and $\H_{xy}= f_y \partial_x - f_x \partial_y$ are the Hamiltonians.
Composing $\E$ with each of these to obtain the elements $\E \circ \E$, $\E\circ \H_{yz}$, $\E\circ \H_{zx}$, and $\E\circ\H_{xy}$ in $D^2_{R|k}$. Their images in $\ker(J_{2,1})\subseteq R^6$ under the map $\nu_2$ from \eqref{eq:LES-ring} are of particular interest and so we introduce the following notation. 
\begin{align*}
E^2&\coloneqq \nu_2(\E\circ \E) \\
&= x^2\partial_x^2+2xy\partial_x\partial_y+2xz\partial_x\partial_z+y^2\partial_y^2+2yz\partial_y\partial_z+z^2\partial_z^2\\
EH_{yz}&\coloneqq \nu_2(\E \circ\H_{yz}) \\
&= xf_z\partial_x\partial_y-xf_z\partial_x\partial_z+yf_z\partial_y^2+(zf_z-yf_y)\partial_y\partial_z-zf_y\partial_z^2 \\
EH_{zx}&\coloneqq \nu_2(\E\circ \H_{zx}) \\
&= -xf_z\partial_x^2-yf_z\partial_x\partial_y+(xf_x-zf_z)\partial_x\partial_z+yf_x\partial_y\partial_z+zf_x\partial_z^2\\
EH_{xy}&\coloneqq \nu_2(\E\circ \H_{xy}) \\
&= xf_y\partial_x^2+(yf_y-xf_x)\partial_x\partial_y+zf_y\partial_x\partial_z-yf_x\partial_y^2-zf_x\partial_y\partial_z \,;
\end{align*}
note that lower order terms are dropped since the map $\nu_2$ factors through $D^2_{R|k}/D^1_{R|k}$.
These correspond to vectors in $R^6$ with the basis $\{\partial_x^{(2)} , \partial_x\partial_y , \partial_x\partial_z , \partial_y^{(2)} , \partial_y\partial_z , \partial_z^{(2)}\}$:
\[
E^2=2\begin{bmatrix}
x^2\\xy\\xz\\ y^2\\yz\\z^2
\end{bmatrix},
\hspace{0.25cm}
EH_{yz}=\begin{bmatrix}
0 \\ x f_z \\ -x f_y\\2y f_z\\z f_z - y f_y\\-2z f_y
\end{bmatrix},
\hspace{0.25cm} 
EH_{zx}=\begin{bmatrix}
-2 x f_z\\ -y f_z \\ x f_x-z f_z\\ 0\\y f_x\\ 2z f_x
\end{bmatrix},
\hspace{0.25cm}
EH_{xy}=\begin{bmatrix}
2 x f_y\\ y f_y - x f_x \\ z f_y\\-2y f_x\\- z f_x\\0
\end{bmatrix} \,.
\]

Similarly by composing each Hamiltonian with itself and considering their images 
under $\nu_2$, we have the following elements in $R^6$:
\[
H_{yz}^2=2\begin{bmatrix}
0\\ 0 \\ 0\\ f_z^2\\-f_yf_z\\f_y^2
\end{bmatrix},
\hspace{0.25cm}
H_{zx}^2=2\begin{bmatrix}
f_z^2\\ 0 \\ -f_xf_z\\0\\0\\f_x^2
\end{bmatrix} ,
\hspace{0.25cm}
H_{xy}^2=2\begin{bmatrix}
f_y^2\\ -f_xf_y \\ 0\\f_x^2\\0\\0
\end{bmatrix} \,.
\]

\begin{remark}\label{rmk:lifting-gens}
 In \cref{subsec:gensD2}, we identify elements of $\ker(J_{2,1})$ with elements in $R\otimes_Q F_{\leq 2}$, the free module with basis
 \[ \partial_x^{(2)}, \partial_x \partial_y, \partial_x \partial_z, \partial_y^{(2)}, \partial_y \partial_z, \partial_z^{(2)}, \partial_x, \partial_y,\partial_z,1
\]
as in \cref{chunk:concrete1}, by taking the naive lift: the vector with zeroes in the last four coordinates.
\end{remark}

We now introduce the remaining three generators of a minimal generating set of $\ker J_{2,1}$; cf.\@ Proposition~\ref{localgensS3}.

\begin{lemma}
\label{newgensS2}
The elements
\begin{align*}
\alpha_x &= \frac{1}{x} 
\left[
H_{yz}^2+\frac{1}{(d-1)^2}\Delta_{xx}E^2
\right]
\\
\alpha_y &= \frac{1}{y} 
\left[
H_{zx}^2+\frac{1}{(d-1)^2}\Delta_{yy}E^2
\right]
\\
\alpha_z &= \frac{1}{z} 
\left[
H_{xy}^2+\frac{1}{(d-1)^2}\Delta_{zz}E^2
\right]
\end{align*}
are well-defined in $R^6$.
\end{lemma}

\begin{proof}
We verify the statement for $\alpha_z$ and the other two arguments are similar. Consider
\begin{align}\label{alphazmatrix}
 H_{xy}^2+\frac{1}{(d-1)^2}\Delta_{zz}E^2
=
\kbordermatrix{
& \\
\partial_x^{(2)} & 2f_y^2+\frac{2}{(d-1)^2}x^2\Delta_{zz} \\
\partial_x\partial_y & -2f_xf_y+\frac{2}{(d-1)^2}xy\Delta_{zz} \\
\partial_x\partial_z & \frac{2}{(d-1)^2}xz\Delta_{zz} \\
\partial_y^{(2)} & 2f_x^2+\frac{2}{(d-1)^2}y^2\Delta_{zz} \\
\partial_y\partial_z & \frac{2}{(d-1)^2}yz\Delta_{zz} \\
\partial_z^{(2)} & \frac{2}{(d-1)^2}z^2\Delta_{zz}
}\,,
\end{align}
which we claim is divisible by $z$. This is clear for the $\partial_x\partial_z$, $\partial_y\partial_z$, and $\partial_z^{(2)}$ entries. 

For the $\partial_x^{(2)}$ entry we have 
\begin{align*}
2f_y^2+\frac{2}{(d-1)^2}x^2\Delta_{zz}&=\frac{2}{(d-1)^2}\left(-x^2\Delta_{zz}-z^2\Delta_{xx}+2xz\Delta_{xz}+x^2\Delta_{zz}\right)\\
&=\frac{2z}{(d-1)^2}\left(2x\Delta_{xz}-z\Delta_{xx}\right)\,,
\end{align*}
where the first equality follows from the identity \cref{PD-prodaa}.

For the $\partial_x\partial_y$ entry we have
\begin{align*}
-2f_xf_y+\frac{2}{(d-1)^2}xy\Delta_{zz}&=\frac{2}{(d-1)^2}\left(-z^2\Delta_{xy}-xy\Delta_{zz}+xz\Delta_{yz}+yz\Delta_{xz}+xy\Delta_{zz}\right)\\
&=\frac{2z}{(d-1)^2}\left(x\Delta_{yz}+y\Delta_{xz}-z\Delta_{xy}\right),
\end{align*}
where the first equality follows from the identity \cref{PD-prodab}.

Similarly, applying the identity \cref{PD-prodaa} to the $\partial_y^{(2)}$ entry (or swapping $x$ and $y$ in the $\partial_x^{(2)}$ entry), we find that the vector \cref{alphazmatrix} is given by 
\begin{align}
\frac{2z}{(d-1)^2}\kbordermatrix{
& \\
\partial_x^{(2)} & 2x\Delta_{xz}-z\Delta_{xx} \\
\partial_x\partial_y & x\Delta_{yz}+y\Delta_{xz}-z\Delta_{xy} \\
\partial_x\partial_z & x\Delta_{zz} \\
\partial_y^{(2)} & 2y\Delta_{yz}-z\Delta_{yy} \\
\partial_y\partial_z & y\Delta_{zz} \\
\partial_z^{(2)} & z\Delta_{zz}
}
\end{align}
which is divisible by $z$, as desired.
\end{proof}

We will show that $\{E^2, EH_{yz}, EH_{zx}, EH_{xy}, \alpha_x, \alpha_y, \alpha_z\}$ is a minimal generating set for $\ker J_{2,1}$ in \cref{subsec:exactnessS2}.

\begin{chunk}\label{M0(2)}
From the computations above, writing these generators in terms of the basis $\{\partial_x^{(2)} , \partial_x\partial_y , \partial_x\partial_z , \partial_y^{(2)} , \partial_y\partial_z , \partial_z^{(2)}\}$ gives the columns of the following matrix 
\[
M_0(2)\coloneqq
\begin{bmatrix}
E^2 & EH_{yz} & EH_{zx} & EH_{xy} & \frac{2}{(d-1)^2}A(2)
\end{bmatrix}\,,
\] 
where
\begin{align*}
A(2)\coloneqq\kbordermatrix{ 
& \alpha_x & \alpha_y& \alpha_z \\
& x\Delta_{xx} & 2x\Delta_{xy} - y\Delta_{xx} & 2x\Delta_{xz} - z \Delta_{xx} \\
& y\Delta_{xx} & x\Delta_{yy} & x \Delta_{yz} + y \Delta_{xz} - z \Delta_{xy} \\
& z\Delta_{xx} & z \Delta_{xy} + x \Delta_{yz} - y \Delta_{xz} & x\Delta_{zz}\\
& 2y\Delta_{xy} - x \Delta_{yy} & y\Delta_{yy} & 2y\Delta_{yz} - z \Delta_{yy}\\
& y \Delta_{xz} + z \Delta_{xy} - x \Delta_{yz} & z\Delta_{yy} & y\Delta_{zz}\\
& 2z\Delta_{xz} - x \Delta_{zz} & 2z\Delta_{yz} - y \Delta_{zz} & z\Delta_{zz}
}\,.
\end{align*}
\end{chunk}

%%%%%%%%%%%%%%%%%%%%%%%%%%%%%%%%%%%%
\subsection{Matrix factorization}\label{subsec:mfS2}
%%%%%%%%%%%%%%%%%%%%%%%%%%%%%%%%%%%%

In this subsection we construct a matrix factorization that is associated to $\ker J_{2,1}$; cf. \cref{c:hypersurface}. Adopt the notation and conventions from \cref{subsec:koszulmf}. 

\begin{chunk}\label{c:alphamap}
We first define a dg $E$-module map $\alpha\colon \Kos^Q(f_x,f_y,f_z)\to \Kos^Q(x,y,z)$: 
\begin{equation*}
\begin{aligned} 
 \xymatrix@R+.5pc@C+.7pc{
%%%%%%%%%%%%%%%%%%%%%%%%%%%%%%%%%%%%%%%%%%%%%%%%%%% row 1
 0 \ar[r] & Q \ar[r]^{D_3}\ar[d]^{\alpha_3} &Q^3 \ar[r]^{D_2} \ar[d]^{\alpha_2}&Q^3 \ar[r]^{D_1}\ar[d]^{\alpha_1}& Q \ar[r]\ar@{=}[d]^{\alpha_0}&0\\
 %%%%%%%%%%%%%%%%%%%%%%%%%%%%%%%row 2
 0 \ar[r] & Q \ar[r]^{\partial_3} &Q^3 \ar[r]^{\partial_2}&Q^3 \ar[r]^{\partial_1}& Q \ar[r]&0
}
\end{aligned} 
 \end{equation*} 
with $\alpha_i$ defined in \cref{sec:glossary}.
Using \cref{a:matrixidentities}, namely equations \cref{c:basic-Eul}, $\alpha$ is a map of complexes. Furthermore, combining the fact that each $\alpha_i$ is symmetric and the $E$-actions in \cref{c:Ekoszul}, it follows from \cref{c:basic-Eul} that $\alpha$ is a dg $E$-module map.

Next, since $\alpha$ is a dg $E$-module map its mapping cone $\cone(\alpha)$ is a dg $E$-module; see, for example, \cite[Section~1.1]{Avramov:2010}. Explicitly, $\cone(\alpha)$ is the complex of free $Q$-modules 
\[
0\to Q \xra{ \begin{bmatrix}
 -D_3 \\
 \alpha_3
\end{bmatrix}} \dsum{Q^3}{Q}\xra{\begin{bmatrix}
-D_2 & 0 \\
\alpha_2 & \del_3 
\end{bmatrix}} \dsum{Q^3}{Q^3}\xra{\begin{bmatrix}
-D_1 & 0 \\
\alpha_1 & \del_2
\end{bmatrix}} \dsum{Q}{Q^3} \xra{\begin{bmatrix}
\alpha_0 & \del_1
\end{bmatrix}
}Q\to 0
\]
with $e\cdot$ given by 
\[
0\leftarrow Q \xla{\begin{bmatrix}
-\del_1 & 0
\end{bmatrix}} \dsum{Q^3}{Q}\xla{\begin{bmatrix}
\del_2 & 0 \\
0 & D_1
\end{bmatrix}} \dsum{Q^3}{Q^3}\xla{\begin{bmatrix}
-\del_3 & 0 \\
0 & -D_2 
\end{bmatrix}} \dsum{Q}{Q^3} \xla{\begin{bmatrix}
0 \\ D_3
\end{bmatrix}
}Q\leftarrow 0\,.
\]
The $Q$-linear quotient complex $L$ of $\shift^{-1} \cone(\alpha)$, obtained by contracting away the copy of $Q\xra{=}Q$, given by 
\begin{align}\label{e:Lcomplex}
  L: \quad
0\to Q \xra{ \begin{bmatrix}
 -D_3 \\
 \alpha_3
\end{bmatrix}} \dsum{Q^3}{Q}\xra{\begin{bmatrix}
-D_2 & 0 \\
\alpha_2 & \del_3 
\end{bmatrix}} \dsum{Q^3}{Q^3}\xra{\begin{bmatrix}
\alpha_1 & \del_2
\end{bmatrix}} Q^3\to 0
\end{align}
can be equipped with a dg $E$-module structure by defining $e\cdot$ as 
\[
0\leftarrow Q \xla{\begin{bmatrix}
-\del_1 & 0
\end{bmatrix}} \dsum{Q^3}{Q}\xla{\begin{bmatrix}
\del_2 & 0 \\
0 & D_1
\end{bmatrix}} \dsum{Q^3}{Q^3}\xla{\begin{bmatrix}
 q \\
 -D_2 
\end{bmatrix}}Q^3\leftarrow 0\,.
\]
The dg $E$-module structure is obtained from the quasi-isomorphism $\iota\colon L\to \shift^{-1}\cone(\alpha)$ given by
\[
\iota_i=\begin{cases}\begin{bmatrix}
-\del_1 \\ 1
\end{bmatrix} &i=0 \\
\id & i=1,2,3\\
0 & \text{else}\,.
\end{cases}
\]
\end{chunk}

\begin{chunk}\label{c:explicitMF(2)}
Let $L$ be as constructed in \cref{c:alphamap}. 
Following \cref{c:fold}, from $L$ we obtain the matrix factorization $\fold(L)$ of $d\cdot f$ over $Q$: 
\[
Q^7 \xra{\begin{bmatrix}
q & -D_2 & 0 \\
-D_2 & \alpha_2 & \del_3\\
0 & -\del_1 & 0
\end{bmatrix}} Q^7 \xra{\begin{bmatrix}
\alpha_1 & \del_2 & 0 \\
\del_2 & 0 & -D_3\\
0 & D_1 & \alpha_3
\end{bmatrix}}Q^7\,.
\] 
Changing bases yields the equivalent matrix factorization:
\[
Q^7 \xra{\begin{bmatrix}
\del_3 & D_2 & 2\alpha_2 \\
0 & \frac{1}{2}q & D_2\\
0 & 0 & \del_1
\end{bmatrix}} Q^7 \xra{\begin{bmatrix}
D_1 & 0 & -2\alpha_3\\
-\del_2 &2\alpha_1 & 0\\
0 &-\del_2 & D_3
\end{bmatrix}}Q^7\,.
\] Finally, we set $M_2(2)$ and $M_1(2)$ to be these matrices tensored down to $R$, i.e., 
\[
M_2(2)\coloneqq \begin{bmatrix}
\del_3 & D_2 & 2\alpha_2 \\
0 & \frac{1}{2}q & D_2\\
0 & 0 & \del_1
\end{bmatrix}\text{ and }
M_1(2)\coloneqq\begin{bmatrix}
D_1 & 0 & -2\alpha_3\\
-\del_2 &2\alpha_1 & 0\\
0 &-\del_2 & D_3
\end{bmatrix}
\] where both matrices have entries in $R$. Recall from \cref{c:hypersurface} that the two-periodic complex 
\[
\cdots \to R^7\xra{M_2(2)}R^7\xra{M_1(2)} R^7\xra{M_2(2)}R^7\xra{M_1(2)}R^7\to\cdots 
\]
is exact. 
\end{chunk}

%%%%%%%%%%%%%%%%%%%%%%%%%%%%%%%%%%%%
\subsection{Resolution of $\ker J_{2,1}$}\label{subsec:exactnessS2}
%%%%%%%%%%%%%%%%%%%%%%%%%%%%%%%%%%%%

In this subsection we establish exactness of the following sequence of $R$-modules 
\[
\cdots
\xrightarrow{M_1(2)}
R^7
\xrightarrow{M_2(2)}
R^7
\xrightarrow{M_1(2)}
R^7
\xrightarrow{M_0(2)} 
R^6 
\xrightarrow{J_{2,1}}
R^3 
\]
where the definitions of the matrices $M_i(2)$ can be found in \ref{M0(2)} and \ref{c:explicitMF(2)}. 

\begin{lemma}
\label{localgensS2}
The $R_x$-module 
$(\ker_{R}(J_{2,1}))_x=\ker_{R_x}(J_{2,1})$ is freely generated by $E^2$, $EH_{yz}$, and $H_{yz}^2$.
\end{lemma}

\begin{proof} 
We show that the given set generates $\ker J_{2,1}$. For this, recall that since $R$ has an isolated singularity, we have that $f_y,f_z$ forms a (possibly improper) regular sequence in $R_x$: indeed, $R_x$ is a domain so $f_y$ and $f_z$ are nonzerodivisors, and by the Euler relation, $f_x\in (f_y,f_z)R_x$, so $(f_y,f_z)R_x = (f_x,f_y,f_z)R_x = R_x$ by the isolated singularity hypothesis, and then $f_y,f_z$ is an (improper) regular sequence by the Chinese Remainder Theorem. To compute $\ker J_{2,1}$ we let $[a_{xx},a_{xy},a_{xz},a_{yy},a_{yz},a_{zz}]^T$ be an element in $\ker_{R_x}(J_{2,1})$. 

Rewriting each $f_x=-(\frac{y}{x}f_y+\frac{z}{x}f_z)$ in the $J_{2,1}$ matrix, and multiplying it by $[a_{xx},a_{xy},a_{xz},a_{yy},a_{yz},a_{zz}]^T$, we get the following equalities:
\begin{align*}0&= -\left(\frac{y}{x} f_y + \frac{z}{x} f_z\right)a_{xx} + f_y a_{xy} + f_z a_{xz} \\
&=f_y\left(a_{xy}-\frac{y}{x} a_{xx}\right) + f_z\left(a_{xz}-\frac{z}{x} a_{xx}\right) 
\end{align*}
and hence as $f_y,f_z$ is $R_x$-regular, there exists $b_1$ in $R_x$ such that
\begin{equation}
\label{axyaxz}
a_{xy}=b_1 f_z + \frac{y}{x} a_{xx} \qquad a_{xz}=-b_1 f_y + \frac{z}{x} a_{xx}\,.
\end{equation}
Similarly, from \[-(y/x f_y + z/x f_z)a_{xy} + f_y a_{yy} + f_z a_{yz} =0
\quad\text{and}\quad -(y/x f_y + z/x f_z)a_{xz} + f_y a_{yz} + f_z a_{zz} =0
\]we conclude that there exists $b_2\in R_x$, from the first equation, and $b_3\in R_x$, from the second equation, such that 
\begin{align*}
a_{yy}&=b_2 f_z + \frac{y}{x} a_{xy} & a_{yz}&=-b_2 f_y + \frac{z}{x} a_{xy} \\ 
a_{yz}&=b_3 f_z + \frac{y}{x} a_{xz} &a_{zz}&=-b_3 f_y + \frac{z}{x} a_{xz}\,.
\end{align*}
Substituting \cref{axyaxz} into these equations yield
\begin{align*}
a_{yy}&=b_2 f_z + b_1 \frac{y}{x} f_z + \frac{y^2}{x^2} a_{xx} & a_{yz}&=-b_2 f_y + b_1 \frac{z}{x} f_z + \frac{yz}{x^2} a_{xx}\\
a_{yz}&=b_3 f_z - b_1 \frac{y}{x} f_y + \frac{yz}{x^2} a_{xx} & a_{zz}&=-b_3 f_y - b_1 \frac{z}{x} f_y + \frac{z^2}{x^2} a_{xx}\,.
\end{align*}

Equating the expressions for $a_{yz}$ one can see that
\begin{align*} 
	f_y\left(b_2-\frac{y}{x} b_1 \right) + f_z\left(b_3-\frac{z}{x} b_1\right) =0
	\end{align*}
	and so finally there exists $c\in R_x$ such that 
	\begin{align*}
	b_2 = c f_z + \frac{y}{x} b_1 \qquad b_3 = -c f_y + \frac{z}{x} b_1\,.
\end{align*}
Therefore, from the equalities above and substituting in $a=a_{xx}/2x^2$, $b=b_1/x$, we get that
\[\begin{bmatrix} a_{xx} \\ a_{xy} \\a_{xz} \\ a_{yy} \\ a_{yz} \\ a_{zz}\end{bmatrix} = \begin{bmatrix} 2a x^2 \\ b x f_z + 2a xy \\ - b x f_y + 2a xz \\ c f_z^2 + 2 b y f_z + 2a y^2 \\ - c f_y f_z - b (z f_z - y f_y) + 2a yz \\ c f_y^2 -2 b z f_y + 2a z^2 \end{bmatrix} = a E^2 + b E H_{yz} + c H^2_{yz} \,.\]

To see linear independence over $R_x$, suppose that $a E^2 + b E H_{yz} + c H^2_{yz}=0$ with $a,b,c\in R_x$. The $\partial^{(2)}_x$-coordinate of the left-hand side is $2a x^2$, so $a=0$. After substituting $a=0$, the $\partial_x \partial_y$-coordinate of the left-hand side is $b x f_z$, so $b=0$. Then, substituting $c=0$, the $\partial_y^{(2)}$-coordinate is $cf_z^2$, so $c=0$, and the linear independence follows.
\end{proof}

\begin{lemma}\label{l:complexM0M1(2)}
We have $M_0(2) M_1(2)=0_{6\times 7}$.
\end{lemma}
\begin{proof} We verify that each column of $M_1(2)$ yields a relation on the columns of $M_0(2)$, which correspond to the generators of $\ker(J_{2,1})$.

Composing the relations in \cref{e:basicrelations} with $\cE$ (cf.~\cref{c:composition}) yields relations
\[\begin{aligned} f_x E^2 + y EH_{xy} - z EH_{zx} &= 0 \\
f_y E^2 - x EH_{xy} + z EH_{yz} &= 0 \\
f_z E^2 + x EH_{zx} - y EH_{yz} &= 0\,. \end{aligned}\] 
which correspond to the first three columns of $M_1(2)$.

There is another relation of the form 
\begin{equation}
\label{sixthcolumn}
y \alpha_x - x\alpha_y =\frac{2}{d-1}(f_{xz} EH_{yz} + f_{yz} EH_{zx} + f_{zz} EH_{xy})\,.
\end{equation}
We check coordinate by coordinate.
In the $\partial_x^{(2)}$ coordinate, we have 
\begin{align*} \frac{2}{(d-1)^2}(xy \Delta_{xx} - 2x^2 \Delta_{xy} + xy \Delta_{xx}) &= \frac{4x}{(d-1)^2}(y \Delta_{xx} - x \Delta_{xy}) \\
&= \frac{2}{d-1} (-2xf_z f_{yz} + 2x f_y f_{zz})\,,
\end{align*}
where the second equality uses \eqref{eq:MinDiff}.
 In the $\partial_x\partial_y$ coordinate, we have the following, where \eqref{eq:MinDiff} is again applied for the second equality:
\begin{align*} \frac{2}{(d-1)^2}(y^2 \Delta_{xx} - x^2 \Delta_{yy}) 
&= \frac{2y}{(d-1)^2}(y\Delta_{xx} - x \Delta_{xy}) - \frac{2x}{(d-1)^2} (x\Delta_{yy} + y \Delta_{xy})\\
&= \frac{2}{d-1} (xf_z f_{xz} -y f_z f_{yz} + y f_y f_{zz} - xf_x f_{zz})\,.
\end{align*}

In the $\partial_x\partial_z$ coordinate, another application of \eqref{eq:MinDiff} yields the second equality below
\begin{align*} \frac{2}{(d-1)^2}(yz \Delta_{xx} - xz\Delta_{xy} -x^2 \Delta_{yz} + xy\Delta_{xz})
&= \frac{2z}{(d-1)^2}(y\Delta_{xx} - x \Delta_{xy}) - \frac{2x}{(d-1)^2} (x\Delta_{yz} - y \Delta_{xz})\\
&= \frac{2}{d-1}(-xf_y f_{xz} -z f_z f_{yz} + x f_x f_{yz} + zf_y f_{zz})\,.
\end{align*}
 The $\partial_y^{(2)}$ coordinate follows from the $\partial_x^{(2)}$ coordinate by symmetry, and likewise the $\partial_y\partial_z$ coordinate follows from the $\partial_x\partial_z$ coordinate.
 In the $\partial_z^{(2)}$ coordinate, a final application of \eqref{eq:MinDiff} establishes
\begin{align*} \frac{2}{(d-1)^2} (2 yz \Delta_{xz} - xy\Delta_{zz} - 2xz \Delta_{yz} + xy \Delta_{zz})
= \frac{2}{d-1}( -2zf_y f_{xz} + 2z f_x f_{yz})\,.
\end{align*}
Thus we have established the relation in \cref{sixthcolumn} that corresponds to the sixth column of $M_1(2)$; the relations coming from fourth and fifth columns of $M_1(2)$ follow from this one by symmetry.

The last relation, corresponding to the seventh column of $M_1(2)$ is 
\[-\frac{2\delta}{{(d-1)^3}} E^2 + f_x \alpha_x + f_y\alpha_y + f_z\alpha_z=0\,.\]
To see this, by symmetry, it suffices to check for the $\partial_x^{(2)}$ and $\partial_x\partial_y$ coordinates. In the $\partial_x^{(2)}$ coordinate, we have 
\begin{align*}
-\frac{2}{{(d-1)^3}}2x^2 &\delta + \frac{2}{(d-1)^2}(x f_x\Delta_{xx} + 2x f_y \Delta_{xy} - y f_y \Delta_{xx} + 2x f_z \Delta_{xz} -z f_z\Delta_{xx})\\ 
&=\frac{4x}{(d-1)^3}(-x \delta+ (d-1)(f_x \Delta_{xx} + f_y \Delta_{xy} + f_z\Delta_{xz}))\\
&=0\,,
\end{align*}
where the last equality is \eqref{eqCram}.

In the $\partial_x\partial_y$ coordinate, we have 
\begin{align*} -\frac{2}{{(d-1)^3}} 2xy &\delta + \frac{2}{(d-1)^2}(y f_x \Delta_{xx} + x f_y \Delta_{yy} + x f_z \Delta_{yz} + y f_z \Delta_{xz} - z f_z \Delta_{xy}) \\
&= \frac{-4xy}{(d-1)^3} \delta + \frac{2x}{(d-1)^2}(f_y \Delta_{yy} + f_z \Delta_{yz}) + \frac{2y}{(d-1)^2}(f_x \Delta_{xx} + f_z \Delta_{xz}) -\frac{2z}{(d-1)^2} f_z \Delta_{xy}\\
&=\frac{-4xy}{(d-1)^3} \delta + \frac{2x}{(d-1)^2}\left(\frac{y\delta}{d-1} -f_x \Delta_{xy}\right) + \frac{2y}{(d-1)^2}\left(\frac{x \delta}{d-1} - f_y \Delta_{xy}\right) -\frac{2z}{(d-1)^2} f_z \Delta_{xy}\\
&= \frac{2 \Delta_{xy}}{(d-1)^2}(-x f_x -y f_y - z f_z )=0\,,
\end{align*}
where the second equality uses \eqref{eqCram}.
\end{proof}

\begin{lemma}\label{l:localexactM0M1(2)}
We have $(\ker M_0(2))_x=(\im M_1(2))_x$.
\end{lemma}
\begin{proof}
Let $v\in (\ker M_0(2))_x$; write $v=(a,b_x,b_y,b_z,c_x,c_y,c_z)\in R_x^7$ so that 
\begin{equation}\label{eq:rel:S2} a E^2 + b_x H_{yz} + b_y H_{zx} + b_z H_{xy} + c_x \alpha_x + c_y \alpha_y + c_z \alpha_z = 0\,.\end{equation}
From \cref{localgensS2} above and symmetry we have that $E^2$, $EH_{xy}$, and $H_{xy}^2$ freely generate 
\[R_x \langle E^2, EH_{yz}, EH_{zx}, EH_{xy}, \alpha_x, \alpha_y, \alpha_z\rangle\,.\]
Using the relations on the generators, we have that
\begin{align*} \alpha_x &= \frac{x}{z^2} H_{xy}^2 - 2 \frac{f_y}{z^2} E H_{xy} + \frac{1}{x}\left( \frac{f_y^2}{z^2} + \frac{\Delta_{xx}}{(d-1)^2}\right) E^2 \\
\alpha_y &= \frac{y}{z^2} H_{xy}^2 + 2 \frac{f_x}{z^2} E H_{xy} + \frac{1}{y}\left( \frac{f_x^2}{z^2} + \frac{\Delta_{yy}}{(d-1)^2}\right) E^2\\
\alpha_z &= \frac{1}{z} H_{xy}^2 + \frac{\Delta_{zz}}{(d-1)^2 z} E^2\,.\end{align*}
By considering the $H_{yz}^2$ coefficient on the relation \eqref{eq:rel:S2}, we get that $x c_x + y c_y+z c_z=0$. Observe that $D_3 = \partial_2 [ 0, f_z/x, f_y/x]^T$, so $\ker \begin{bmatrix} x & y & z\end{bmatrix} R_x$ is generated by the image of $\partial_2^T$. Thus, using columns 4--6 of $M_1(2)$ there exists a relation $v'\in (\im M_1(2))_x$ such that the $c_x,c_y,c_z$ coordinates of $v'$ agree with $v$. Replacing $v$ with $v-v'$, we may assume that $c_x=c_y=c_z=0$.

Now, we have
\[ E H_{yz} = \frac{x}{z} E H_{xy} - \frac{f_y}{z^2} E^2\,,\quad \quad
E H_{zx} = \frac{y}{z} E H_{xy} + \frac{f_x}{z^2} E^2\,,\quad \quad
E H_{xy} = E H_{xy}\,, 
\]
and by considering the $E H_{yz}$ coefficient on the relation \eqref{eq:rel:S2}, we get that $x b_x + y b_y + z b_z=0$. Using columns 1--3 of $M_1(2)$, we can find a relation $v''\in (\im M_1(2))_x$ such that the $b_x,b_y,b_z$-coordinates of $v''$ agree with those of $v$ and the $c_x,c_y,c_z$ coordinates are zero. Replacing $v$ with $v-v''$ we can assume that $b_x=b_y=b_z=c_x=c_y=c_z=0$. But $a E^2=0$ implies $a=0$, and the assertion follows.
\end{proof}

We are now ready to prove the main result of this subsection. 

\begin{proposition}
\label{gensS2}
The sequence of $R$-modules 
\[
\cdots
\xrightarrow{M_1(2)}
R^7
\xrightarrow{M_2(2)}
R^7
\xrightarrow{M_1(2)}
R^7
\xrightarrow{M_0(2)} 
R^6 
\xrightarrow{J_{2,1}}
R^3 
\]
is exact. 
In particular, the $R$-module 
$\ker_{R}J_{2,1}$ is generated by the operators 
\[
\{E^2, EH_{yz}, EH_{zx}, EH_{xy}, \alpha_x, \alpha_y, \alpha_z\}\,.
\]
\end{proposition}

\begin{proof}
First note that since $M_1(2)$ and $M_2(2)$ are the images over $R$ of a matrix factorization of $f$ over $Q$ as proved in \ref{c:explicitMF(2)}, one obtains that the (infinite) portion of the sequence involving those matrices is exact; see \ref{c:explicitMF(2)}. 
Furthermore, the initial portion of the sequence is also a complex by \ref{l:complexM0M1(2)} and the facts that the columns of $M_0(2)$ correspond locally to compositions of lower order operators and that $x$, $y$, and $z$ are each nonzerodivisors over $R$. 

For the initial portion of the sequence we argue by localization. 
We begin by using the lemmas above to establish that the homologies of the sequence are supported at the homogeneous maximal ideal (that is, the sequence is exact on the punctured spectrum).
Let $\p \neq (x,y,z)$, so that $x, y,$ or $z$ is a unit after localizing at $\p$. 
By Lemma~\ref{l:localexactM0M1(2)} and symmetry, we have $(\ker M_0(2))_\p=(\im M_1(2))_\p$. 
Furthermore, to see $\ker(J_{2,1})_{\p}=\im M_0(2)$, note that, since $H_{yz}^2$ is generated by $\alpha_x$ and $E^2$ over $R_x$, $\ker(J_{2,1})_x$ is generated by $E^2$, $EH_{yz}$, and $\alpha_x$. By symmetry, we have that $\ker(J_{2,1})_y$ is generated by $E^2$, $EH_{zx}$, and $\alpha_y$; and $\ker(J_{2,1})_z$ is generated by $E^2$, $EH_{xy}$, and $\alpha_z$. Thus, if $\p \neq (x,y,z)$, we have that $\ker(J_{2,1})_{\p}$ is generated by $E^2,EH_{yz},EH_{zx},EH_{xy},\alpha_x,\alpha_y,\alpha_z$.

Next we use depth to argue that the initial portion of the sequence is, in fact, exact. 
For this, we use the lemma below in stages. 
First consider the inclusion 
\[
i\colon \im M_1(2) \hookrightarrow \ker M_0(2)
\]
From the exactness of the infinite periodic portion of the complex above, we see that $\im M_1(2)$ is an infinite syzygy, hence maximal Cohen-Macaulay and so satisfies $S_2$ as the ring $R$ is a 2-dimensional hypersurface. Furthermore, since $\ker M_0(2)$ is contained in $R^7$, it satisfies $S_1$. By exactness on the punctured spectrum as explained above, the inclusion $i$ is isomorphism when localized at each prime of height at most 1, and so by Lemma~\ref{l:iso} below it is an isomorphism. 

Next consider the inclusion 
\[
j\colon \im M_0(2) \hookrightarrow \ker J_{2,1}
\]
To see that $\im M_0(2)$ satisfies $S_2$, we note that by the previous step, one has an isomorphism $\coker M_1(2)\cong \im M_0(2)$, and again from the exactness of the infinite periodic portion of the complex above $\coker M_1(2)=\ker M_2(2)$ is an infinite syzygy and hence maximal Cohen-Macaulay. 
The module $\ker J_{2,1}$ satisfies $S_2$ as it is contained in a free $R$-module, and hence $j$ is an isomorphism by Lemma~\ref{l:iso}, below, and exactness on the punctured spectrum as explained above. 
\end{proof}

The following is a standard useful tool for proving that maps are isomorphisms. It can be verified by depth-counting arguments after localization at minimal primes in the support of the cokernel and the kernel. 

\begin{lemma}\label{l:iso}
Let $\varphi\colon M \to N$ be a homomorphism of finitely generated modules over a Noetherian ring. Suppose that $M$ and $N$ satisfy Serre's property $S_2$ and $S_1$, respectively. 
If the localization $\varphi_\p$ is an isomorphism for all primes of height at most 1, then $\varphi$ is an isomorphism.
\end{lemma}

%%%%%%%%%%%%%%%%%%%%%%%%%%%%%%%%%%%%
\subsection{Resolution of $D^2_{R|k}$}\label{subsec:liftS2}
%%%%%%%%%%%%%%%%%%%%%%%%%%%%%%%%%%%%
We now construct the minimal $R$-free resolution of $D^2_{R|k}.$
\begin{chunk}
\label{lifts2diagram}
Recall the $R$-free resolution of $R/(f_x,f_y,f_z)$ constructed in \cref{resD1}:
\[
G(1)=\cdots \xra{M_2(1)}R^4\xra{M_1(1)}R^4\xra{M_2(1)}R^4\xra{M_1(1)}R^4\xra{M_0(1)}R^3\xra{J}R\to 0\,.
\]
Also, consider the sequence 
\[
G(2)=
\cdots
\xrightarrow{M_2(2)}
R^7
\xrightarrow{M_1(2)}
R^7
\xrightarrow{M_2(2)}
R^7
\xrightarrow{M_1(2)}
R^7
\xrightarrow{M_0(2)} 
R^6 
\xrightarrow{-J_{2,1}}
R^3 \to 0\,,
\]
which is an $R$-free resolution of $\coker J_{2,1}$ by \cref{gensS2}. Define $\theta(2)\colon \shift^{-1}G(2)\to G(1)$ according to the diagram 
\begin{equation*}
\begin{aligned} 
 \xymatrix@R+.5pc@C+.7pc{
%%%%%%%%%%%%%%%%%%%%%%%%%%%%%%%%%%%%%%%%%%%%%%%%%%% row 1
\cdots\ar[r]^{-M_1(2)}&R^7\ar[d]^{\theta_3(2)} \ar[r]^{-M_2(2)}&R^7\ar[r]^{-M_1(2)}\ar[d]^{\theta_2(2)}& R^7\ar[r]^{-M_0(2)} \ar[d]^{\theta_1(2)}&R^6 \ar[d]^{\theta_0(2)}\ar[r]^{-J_{2,1}}&R^3 \ar[r] &0\\
%%%%%%%%%%%%%%%%%%%%%%%%%%%%%%%%%%%%%%%%%%%%%%%%%% row 2
\cdots \ar[r]^{M_2(1)}&R^4 \ar[r]^{M_1(1)} &R^4 \ar[r]^{M_0(1)}&R^3\ar[r]^{J}&R\ar[r]&0&
}
\end{aligned} 
 \end{equation*} 
 where each $\theta_i(2)$ is defined in \cref{sec:glossary}. 
\end{chunk}

\begin{lemma}\label{l:liftS2}
The map $\theta(2)\colon \shift^{-1} G(2)\to G(1)$, defined in \cref{lifts2diagram}, is a morphism of complexes. In particular, its cone, which we denote by $(C,\partial^C)$, is the following nonnegatively graded complex of free $R$-modules: 
\[
C\coloneqq \cone(\theta(2))= \cdots \to \begin{matrix}R^4\\\oplus\\R^7 \end{matrix} \xra{\small \begin{bmatrix} M_1(1)&\theta_2(2)\\0&M_1(2) \end{bmatrix}} \begin{matrix}R^4\\\oplus\\R^7 \end{matrix} \xra{\small \begin{bmatrix} M_0(1)&\theta_1(2)\\0&M_0(2) \end{bmatrix}} \begin{matrix}R^3\\\oplus\\R^6 \end{matrix} \xra{\small \begin{bmatrix} D_1&\theta_0(2)\\0&J_{2,1} \end{bmatrix}} \begin{matrix}R\\\oplus\\R^3 \end{matrix} \to 0\,.
\]
\end{lemma}
\begin{proof}
Adopt the notation from \cref{M0(2)}. For the rightmost square in \cref{lifts2diagram}, observe that 
\begin{align*}
 -\theta_0(2)M_0(2)
&={\small\frac{1}{d-1}\begin{bmatrix}
 0_{1\times 4} & f_x\Delta_{xx}+f_y\Delta_{xy}+f_z\Delta_{xz} & f_x\Delta_{xy}+f_y\Delta_{yy}+f_z\Delta_{yz} & 
 f_x\Delta_{xz}+f_y\Delta_{yz}+f_z\Delta_{zz}
\end{bmatrix}}\\
&=\begin{bmatrix}
 0_{1\times 4} &J\alpha_2
 \end{bmatrix}=J\theta_1(2)
 \,;
\end{align*}
we justify the first equality, and the remaining are evident.

For the first equality, first use \cref{2-1deriv} and \cref{2nd order Euler id} from \cref{appendix} to deduce that the first four columns of $M_0(2)$ are in the kernel of $\theta_0(2)$. Now note that
\begin{align*}
 \theta_0(2)(\alpha_z)&=(d-1)(f_x\Delta_{xz}+f_y\Delta_{yz}+f_z\Delta_{zz})-z\theta_0(2)\begin{bmatrix}
 \Delta_{xx}\\
 \Delta_{xy}\\
 \Delta_{xz}\\
 \Delta_{yy}\\
 \Delta_{yz}\\
 \Delta_{zz}
 \end{bmatrix}
 \\
 &=(d-1)(f_x\Delta_{xz}+f_y\Delta_{yz}+f_z\Delta_{zz})-\frac{3}{2}z\delta\\
 &=(d-1)(f_x\Delta_{xz}+f_y\Delta_{yz}+f_z\Delta_{zz})-\frac{3}{2}(d-1)(f_x\Delta_{xz}+f_y\Delta_{yz}+f_z\Delta_{zz})\\
 &=-\frac{d-1}{2}(f_x\Delta_{xz}+f_y\Delta_{yz}+f_z\Delta_{zz})\,;
\end{align*}
the first equality follows from \cref{2-1deriv}, the second equality uses \cref{det-exp}, the third equality follows \cref{eqCram}.
This shows that the last entries are equal and the equality in the two remaining entries can be shown similarly.

The second square commutes as
\begin{align*}
 -\theta_1(2)M_1(2) &=(d-1)\begin{bmatrix}
 0_{3\times 3} & \alpha_2\del_2 & -\alpha_2D_3 
 \end{bmatrix}\\
 &=(d-1)\begin{bmatrix}
 0_{3\times 3} & D_2\alpha_1 & -\del_3\alpha_3
 \end{bmatrix}\\
 &=M_0(1)\theta_2(2)\,;
\end{align*}
where the second equality uses \cref{c:basic-Eul}. Also, to see that the third square commutes observe that 
\begin{align*}
 -\theta_2(2)M_2(2)&= -(d-1)\begin{bmatrix}
 0_{3\times 3} & 0_{3\times 3} & -\alpha_3 \del_1\\
 0_{1\times 3} & \frac{1}{2}\alpha_1 q & \alpha_1 D_2
 \end{bmatrix}\\
 &=-\frac{d-1}{2}\begin{bmatrix}
 0_{3\times 3} & 0_{3\times 3} & -2D_1\alpha_2\\
 0_{1\times 3} & D_3\del_1 & \del_2\alpha_2
 \end{bmatrix}\\
 &=M_1(1)\theta_3(2)\,;
\end{align*}
where the second equality uses symmetry of $\alpha_i$, 
\cref{c:basic-Eul}, and \cref{c:d-D-alpha-q}.

Finally, as $(M_2(i),M_1(i))$ for $i=1,2$ correspond to matrix factorizations of $d\cdot f$ over $Q$ and 
\[
-\theta_2(2)M_2(2)=M_1(1)\theta_3(2)\,,
\]
it is standard that 
\[
-\theta_3(2)M_2(1)=M_2(1)\theta_2(2)\,.
\]
Now using the two-periodicity of $G(1)$ and $\shift^{-1} G(2)$ after homological degree two, all the remaining squares in the diagram commute. Thus $\theta(2)$ is a map of complexes so taking its cone yields the desired complex $C$ defined above.
\end{proof}

\begin{theorem}
\label{resolutionD2}
Assume $R=k[x,y,z]/(f)$ is an isolated singularity hypersurface where $k$ is a field of characteristic zero. The augmented minimal $R$-free resolution of $D_{R|k}^2\subseteq R^3\oplus R^6\oplus R$ has the form 
\[
\cdots \xra{\small \begin{bmatrix} M_2(1)&\theta_3(2)\\0&M_2(2) \end{bmatrix}}
\begin{matrix}R^4\\\oplus\\R^7 \end{matrix} \xra{\small \begin{bmatrix} M_1(1)&\theta_2(2)\\0&M_1(2) \end{bmatrix}}\begin{matrix}R^4\\\oplus\\R^7 \end{matrix} \xra{\small \begin{bmatrix} M_2(1)&\theta_3(2)\\0&M_2(2) \end{bmatrix}} \begin{matrix}R^4\\\oplus\\R^7 \end{matrix} \xra{\small \begin{bmatrix} M_1(1)&\theta_2(2)\\0&M_1(2)\\ 0 & 0 \end{bmatrix}} \begin{matrix}R^4\\\oplus\\R^7\\\oplus \\R \end{matrix} \xra{\epsilon} D_{R|k}^2 \to 0
\]
where a minimal set of generators for $D_{R|k}^2$ is given by the columns of 
\[
\epsilon=\begin{bmatrix} M_0(1)&\theta_1(2)&0\\0&M_0(2) & 0\\
0 & 0 & 1\end{bmatrix}
\]
and the block matrices are defined in  \cref{M0(2)}, \cref{c:explicitMF(2)}, and \cref{sec:glossary}. 
\end{theorem}

\begin{proof}
By \cref{l:liftS2}, $C$ is a bounded below complex of finite rank free $R$-modules. Moreover, there is the cone exact sequence of complexes
\begin{equation}
\label{e_exact-sequence}
0\to G(1)\to C\to G(2)\to 0\,.
\end{equation}
Since $G(1)$ and $G(2)$ have homology concentrated in degree zero by \cref{resD1,gensS2}, examining the induced long exact sequence in homology forces 
 $C$ to also have homology concentrated only in degree zero. Thus $C$ is a minimal free resolution of $\coker \del_1^C.$ The desired conclusion now follows from the $R$-module isomorphism
\[
D^2_{R|k}\cong \ker \del_1^C\oplus R = \im \del_2^C \oplus R = \im \epsilon \,,
\]
which follows from \cref{c:composition}, observing that $\del_1^C$ equals $P_2$ in \cref{c:presentationS2} and noting that $D_1=J_{1,0}$ and $\theta_0(2) = J_{2,0}$. 
\end{proof}

By truncating the exact sequence of complexes in \cref{e_exact-sequence} in homological degrees one and higher, the resulting long exact sequence in homology yields the following corollary to \cref{resolutionD2}. 

\begin{corollary}
\label{cor:exactsequenceD2}
There is a short exact sequence of $R$-modules 
\[
0\to D^1_{R|k}\to D^2_{R|k}\to \ker J_{2,1}\to 0
\]
where $D^1_{R|k}\to D^2_{R|k}$ is the inclusion.
\end{corollary}

Similar to \cref{c:d1graded}, the differentials in the resolution of $D^2_{R|k}$ are homogeneous with respect to the internal grading of $R$ where 
\[
|\alpha_x|=|\alpha_y|=|\alpha_z|=2d-5\,.
\]
Hence, from \cref{resolutionD2} we can read off the graded Betti numbers of $D^2_{R|k}$.
\begin{corollary} 
\label{cor:d2graded}
The graded Betti numbers of $D^2=D^2_{R|k}$ are given by
\[
\beta_{0,j}^R(D^2)=\begin{cases}
3 & j=0\\
6 & j=d-2\\
3& j=2d-5\\
0 & \text{otherwise}
\end{cases}\,,
\]
and for $n\geqslant 1$, 
\[
\beta_{2n-1,j}^R(D^2)=\begin{cases}
6 & j=nd-1\\
1 & j=nd+d-3\\
3 & j= nd +d -4\\
1 & j = nd+2d-6\\
0 & \text{otherwise}
\end{cases}\quad \text{and}\quad \beta_{2n,j}^R(D^2)=\begin{cases}
2 & j=nd\\
6 & j=nd + d-2\\
3 & j=nd+2d-5\\
0 & \text{otherwise}
\end{cases}\,.
\]
\end{corollary}

\begin{corollary}
\label{cor:levels2}
The following inequalities are both satisfied
\[\level_{\Dsg(R)}^k (\ker J_{2,1})\leqslant 2\quad \text{and} \quad\level_{\Dsg(R)}^k (D_{R|k}^2)\leqslant 3\,.\] 
\end{corollary}
\begin{proof}
The second inequality follows from the first as 
\[
\level_{\Dsg(R)}^k (D_{R|k}^2)\leqslant \level_{\Dsg(R)}^k (\ker J_{2,1})+\level_{\Dsg(R)}^k (k)
\] 
using \cref{cor:exactsequenceD2}, as well as \cref{r:hz}. For the first inequality, from \cref{c:alphamap} and \cref{c:explicitMF(2)} it follows that the matrix factorization modeling the tail of the minimal free resolution of $\ker J_{2,1}$ is a mapping cone on the matrix factorization of $k$ (up to a shift), and as a consequence the desired inequality is satisfied.
\end{proof}

\begin{remark}
As $R$ is an isolated singularity, \[A\coloneqq Q/(f_x,f_y,f_z)=R/(f_x,f_y,f_z)\]
is an artinian complete intersection. Its socle is $\delta$ as defined in \cref{sec:glossary}. It follows from \cite[Exercise~21.23]{Eisenbud:1995}, that the matrix factorization corresponding to the $R$-resolution of $A/(\delta)$ is the same as the matrix factorization constructed in \cref{c:explicitMF(2)}. Therefore there is an isomorphism $\ker J_{2,1}\simeq A/(\delta)$ in $\Dsg(R)$, see \cref{c:equivalence}. Hence, $D^2_{R|k}$ is a mapping cone of a morphism $\shift^{-1}A/(\delta)\to k$ in $\Dsg(R)$. The isomorphism above does not necessarily hold when $Q$ has more than three variables; indeed, even for straightforward examples in four variables their corresponding matrix factorizations have different ranks. 
%However it would be interesting to study the relationship between these objects more generally. 
\end{remark}

%%%%%%%%%%%%%%%%%%%%%%%%%%%%%%%%%%%%
\subsection{A set of generators of $D^2_{R|k}$}\label{subsec:gensD2}
%%%%%%%%%%%%%%%%%%%%%%%%%%%%%%%%%%%%

Now we write the generators of $D^2_{R|k}$ given in \cref{resolutionD2} as differential operators. By abuse of notation, we consider elements of $\ker(J_{2,1})$ as elements of the ambient free module $R^{10}$ of $D^2_{R|k}$ as in \cref{rmk:lifting-gens}. We begin by writing some relevant operators of order two 
in terms of lifts of their images in $\ker(J_{2,1})$ and operators of order one, as follows:
\begin{align}\label{eq:order2opsE}
 \E^2=E^2+\E 
\end{align}
\begin{align}\label{eq:order2opsEH}
 \E\H_{yz}=EH_{yz}+(d-1)\H_{yz},\quad \E\H_{zx}=EH_{zx}+(d-1)\H_{zx},\quad \E\H_{xy}=EH_{xy}+(d-1)\H_{xy}
\end{align}
\begin{equation}\label{eq:order2opsH}
\begin{split}
 \H^2_{yz}&=H^2_{yz}+\frac{1}{d-1}x\left(\Delta_{xx}\partial_x+\Delta_{xy}\partial_y+\Delta_{xz}\partial_z\right)-\frac{1}{d-1}\Delta_{xx}\E \\
 \H^2_{zx}&=H^2_{zx}+\frac{1}{d-1}y\left(\Delta_{xy}\partial_x+\Delta_{yy}\partial_y+\Delta_{yz}\partial_z\right)-\frac{1}{d-1}\Delta_{yy}\E \\
 \H^2_{xy}&=H^2_{xy}+\frac{1}{d-1}z\left(\Delta_{xz}\partial_x+\Delta_{yz}\partial_y+\Delta_{zz}\partial_z\right)-\frac{1}{d-1}\Delta_{zz}\E \,.
\end{split}
\end{equation}
Note that above, we are lifting the representative for $E^2$ in \cref{subsec:gensS2} from $\ker(J_{2,1})\cong D^2_{R|k}/D^1_{R|k}$ to $D^2_{R|k}$, and so 
\[
E^2=x^2\partial_x^2+2xy\partial_x\partial_y+2xz\partial_x\partial_z+y^2 \partial_x+2yz\partial_y\partial_z+z^2\partial_z^2\,;
\]
the same convention is adopted for $EH_{ij}$.

\begin{theorem}\label{genops2}
A minimal set of generators for $D^2_{R|k}$ is given by $G_0\cup G_1\cup G_2$, where $G_0=\{1 \}$, $G_1=\{\E,\H_{yz}, \H_{zx}, \H_{xy}\}$ and 
\begin{align*}
 G_2&=\{\E^2,\,\,\, \E\H_{yz},\,\,\, \E\H_{zx},\,\,\, \E\H_{xy},\,\,\, \A_x,\,\,\, \A_y,\,\,\, \A_z\},
\end{align*}
with 
\begin{align*}
\A_x &= \frac{1}{x} 
\left[
\H_{yz}^2+\frac{1}{(d-1)^2}\Delta_{xx}\E^2+\frac{d-2}{(d-1)^2}\Delta_{xx}\E
\right]
\\
\A_y &= \frac{1}{y} 
\left[
\H_{zx}^2+\frac{1}{(d-1)^2}\Delta_{yy}\E^2+\frac{d-2}{(d-1)^2}\Delta_{yy}\E
\right]
\\
\A_z &= \frac{1}{z} 
\left[
\H_{xy}^2+\frac{1}{(d-1)^2}\Delta_{zz}\E^2+\frac{d-2}{(d-1)^2}\Delta_{zz}\E
\right]\,.
\end{align*}
\end{theorem}

\begin{proof}

First note that a generating set for $D^2_{R|k}$ can be given by the union of a minimal set of generators for $D^1_{R|k}$ and any set of lifts to $D^2_{R|k}$ of a minimal set of generators for the quotient $D^2_{R|k}/D^1_{R|k}$. 
First, from \cref{subsec:d1}, especially \cref{E} and \cref{H}, the union $G_0 \cup G_1$ is a minimal generating set for $D^1_{R|k}$. 
Second, we claim that $G_2$ is the desired set of lifts from the quotient $D^2_{R|k}/D^1_{R|k}$. 
As compositions of lower order operators, the first 7 elements of $G_2$ are clearly in $D^2_{R|k}$. 
Furthermore, $\A_x,\A_y,\A_z$ are also well-defined operators, so elements of $D^2_{R|k}$: Indeed, using \cref{eq:order2opsE} and \cref{eq:order2opsH}, one can check that $\A_x$, $\A_y$, and $\A_z$ agree with the generators corresponding to last three columns of the augmentation map $\epsilon$ in \cref{resolutionD2}. 
Now, since the images of the elements of $G_2$ in the quotient $D^2_{R|k}/D^1_{R|k}$ are the minimal set of generators $\{ E^2, EH_{yz}, EH_{zx}, EH_{xy}, \alpha_x, \alpha_y, \alpha_z\ \}$
from \cref{localgensS2}, we get that $G_0 \cup G_1\cup G_2$ is a set of generators for $D^2_{R|k}$. This is in fact a minimal set of generators of $D^2_{R|k}$ since its cardinality agrees with $\beta_0^R(D_{R|k})$ calculated in \cref{resolutionD2}. 
\end{proof}

\begin{remark}
\label{r:agreement} 
Note that except for $\A_x$, $\A_y$, and $A_z$, the generators in \cref{genops2} do not agree with the lifts of the generators of $\ker J_{2,1}$, that is, the generators of $D^2_{R|k}$ given in the augmentation map $\epsilon$ in \cref{resolutionD2}. However, one can subtract lower order operators so that they do agree. For example, one could replace $\E^2$ in the generating set by $\E^2-\E$ and $\E\H_{yz}$ by $\E\H_{yz}-(d-1)\H_{yz}$ so that they agree with the corresponding generators from \cref{resolutionD2}. One can replace the others similarly using \cref{eq:order2opsEH}. 
\end{remark}

%%%%%%%%%%%%%%%%%%%%%%%%%%%%%%%%%%%%%%%%%%%%%%%%%%%%%%%%%%%%%%
\section{Differential operators of order 3}\label{sec:D3}
%%%%%%%%%%%%%%%%%%%%%%%%%%%%%%%%%%%%%%%%%%%%%%%%%%%%%%%%%%%%%%

We continue with \cref{notation:main}. In this section, we construct the minimal $R$-free resolution of $D^3_{R|k}$; cf.\@ \cref{resolutionD3}. Following the model in the previous section the majority of the work in the present section is contained in \cref{subsec:gensS3,subsec:mfS3,subsec:exactnessS3}.

%%%%%%%%%%%%%%%%%%%%%%%%%%%%%%%%%%%%
\subsection{A set of generators of $\ker J_{3,2}$}\label{subsec:gensS3}
%%%%%%%%%%%%%%%%%%%%%%%%%%%%%%%%%%%%

In this subsection we introduce a set of generators for $\ker J_{3,2}$; cf.\@ \cref{gensS3}.
 We begin by composing $\cE$ with each of the generators of $\ker J_{2,1}$ (see \cref{c:composition})
\[
\{E^2, EH_{yz}, EH_{zx}, EH_{xy}, \alpha_x, \alpha_y, \alpha_z\}
\]
from \cref{gensS2} as well as each Hamiltonian with itself three times and we compute their coefficients in the basis 
\[\{\partial_x^{(3)}, \partial_x^{(2)}\partial_y, \partial_x^{(2)}\partial_z, \partial_x\partial_y^{(2)}, \partial_x\partial_y\partial_z, \partial_x\partial_z^{(2)}, \partial_y^{(3)}, \partial_y^{(2)}\partial_z, \partial_y\partial_z^{(2)}, \partial_z^{(3)}\}\] 
as follows:
\[
E^3=6\begin{bmatrix}
x^3 \\
x^2y \\
x^2z \\
xy^2 \\
xyz \\
xz^2 \\
y^3 \\
y^2z \\
yz^2 \\
z^3
\end{bmatrix},
\hspace{0.25cm}
E^2H_{yz}=\begin{bmatrix}
0 \\
2x^2f_z \\
-2x^2f_y \\
4xyf_z \\
-2xyf_y+2xzf_z \\
-4xzf_y \\
6y^2f_z \\
-2y^2f_y+4yzf_z \\
-4yzf_y+2z^2f_z \\
-6z^2f_y
\end{bmatrix},
\hspace{0.25cm}
E^2H_{zx}=\begin{bmatrix}
-6x^2f_z \\
-4xyf_z \\
2x^2f_x-4xzf_z \\
-2y^2f_z \\
2xyf_x-2yzf_z \\
4xzf_x-2z^2f_z \\
0 \\
2y^2f_x \\
4yzf_x \\
6z^2f_x
\end{bmatrix},
\]

\[
E^2H_{xy}=\begin{bmatrix}
6x^2f_y \\
-2x^2f_x+4xyf_y \\
4xzf_y \\
-4xyf_x+2y^2f_y \\
-2xzf_x+2yzf_y \\
2z^2f_y \\
-6y^2f_x \\
-4yzf_x \\
-2z^2f_x \\
0
\end{bmatrix},
\hspace{0.25cm}
E\alpha_x=\frac{2}{(d-1)^2}\begin{bmatrix}
3x^2\Delta_{xx} \\
3xy\Delta_{xx} \\
3xz\Delta_{xx} \\
2xy\Delta_{xy}-x^2\Delta_{yy}+2y^2\Delta_{xx} \\
xy\Delta_{xz}+xz\Delta_{xy}-x^2\Delta_{yz}+2yz\Delta_{xx} \\
2xz\Delta_{xz}-x^2\Delta_{zz}+2z^2\Delta_{xx}\\
6y^2\Delta_{xy}-3xy\Delta_{yy} \\
2y^2\Delta_{xz}+4yz\Delta_{xy}-2xy\Delta_{yz}-xz\Delta_{yy} \\
4yz\Delta_{xz}-xy\Delta_{zz}+2z^2\Delta_{xy}-2xz\Delta_{yz} \\
6z^2\Delta_{xz}-3xz\Delta_{zz}
\end{bmatrix},
\]

\[
\hspace*{-0.5cm}
E\alpha_y=\frac{2}{(d-1)^2}\begin{bmatrix}
6x^2\Delta_{xy}-3xy\Delta_{xx} \\
2xy\Delta_{xy}-y^2\Delta_{xx}+2x^2\Delta_{yy} \\
4xz\Delta_{xy}-yz\Delta_{xx}+2x^2\Delta_{yz}-2xy\Delta_{xz} \\
3xy\Delta_{yy} \\
yz\Delta_{xy}+xy\Delta_{yz}-y^2\Delta_{xy}+2xz\Delta_{yy} \\
2z^2\Delta_{xy}+4xz\Delta_{yz}-2yz\Delta_{xz}-xy\Delta_{zz}\\
3y^2\Delta_{yy} \\
3yz\Delta_{yy} \\
2yz\Delta_{yz}-y^2\Delta_{zz}+2z^2\Delta_{yy} \\
6z^2\Delta_{yz}-3yz\Delta_{zz}
\end{bmatrix},
\hspace{0.25cm}
E\alpha_z=\frac{2}{(d-1)^2}\begin{bmatrix}
6x^2\Delta_{xz}-3xz\Delta_{xx} \\
2x^2\Delta_{yz}+4xy\Delta_{xz}-2xz\Delta_{xy}-yz\Delta_{xx} \\
2xz\Delta_{xz}-z^2\Delta_{xx}+2x^2\Delta_{zz} \\
4xy\Delta_{yz}-xz\Delta_{yy}+2y^2\Delta_{xz}-2yz\Delta_{xy} \\
xz\Delta_{yz}+yz\Delta_{xz}-z^2\Delta_{yz}+2xy\Delta_{zz} \\
3xz\Delta_{zz}\\
6y^2\Delta_{yz}-3yz\Delta_{yy} \\
2yz\Delta_{yz}-z^2\Delta_{yy}+2y^2\Delta_{zz} \\
3yz\Delta_{zz} \\
3z^2\Delta_{zz}
\end{bmatrix},
\]

\[
H_{yz}^3=6\begin{bmatrix}
 0 \\
 0 \\
 0 \\
 0 \\
 0 \\
 0 \\
 f_z^3\\
 -f_yf_z^2\\
 f_y^2f_z\\
 -f_y^3
\end{bmatrix},
\hspace{0.25cm}
H_{zx}^3=6\begin{bmatrix}
 -f_z^3 \\
 0 \\
 f_xf_z^2 \\
 0 \\
 0 \\
 -f_x^2f_z \\
 0 \\
 0 \\
 0 \\
 f_x^3
\end{bmatrix},
\hspace{0.25cm}
H_{xy}^3=6\begin{bmatrix}
 f_y^3 \\
 -f_xf_y^2 \\
 0 \\
 f_x^2f_y \\
 0 \\
 0 \\
 -f_x^3\\
 0 \\
 0 \\
 0
\end{bmatrix}\,.
\]

The rest of the subsection is spent constructing the remaining generators of a minimal generating set of $\ker J_{3,2}$; see Proposition~\ref{localgensS3}.

\begin{lemma}
\label{zetas}
The elements
\begin{align*}
\zeta_x &= \frac{1}{x^2} 
\left[
H_{yz}^3+\frac{3}{(d-1)^2}\Delta_{xx}E^2H_{yz}-\frac{1}{(d-1)^2(d-2)}H_{yz}(\Delta_{xx})E^3
\right]
\\
\zeta_y &= \frac{1}{y^2} 
\left[
H_{zx}^3+\frac{3}{(d-1)^2}\Delta_{yy}E^2H_{zx}-\frac{1}{(d-1)^2(d-2)}H_{zx}(\Delta_{yy})E^3
\right]
\\
\zeta_z &= \frac{1}{z^2} 
\left[
H_{xy}^3+\frac{3}{(d-1)^2}\Delta_{zz}E^2H_{xy}-\frac{1}{(d-1)^2(d-2)}H_{xy}(\Delta_{zz})E^3
\right]
\end{align*}
are well-defined in $R^{10}$.
\end{lemma}

\begin{proof}
We verify that $\zeta_x$ is well-defined in $R^{10}$ and the other two calculations follow by symmetry. 
That is, we claim
\begin{align*}
 H_{yz}^3+\frac{3}{(d-1)^2}\Delta_{xx}E^2H_{yz}-\frac{1}{(d-1)^2(d-2)}H_{yz}(\Delta_{xx})E^3,
\end{align*}
which is given by the matrix 
\begin{align}\label{zetaxmatrix}
\kbordermatrix{
& \\
\partial_x^{(3)} & -\frac{6}{(d-1)^2(d-2)}H_{yz}(\Delta_{xx})x^3 \\
\partial_x^{(2)}\partial_y & -\frac{6}{(d-1)^2(d-2)}H_{yz}(\Delta_{xx})x^2y+\frac{6}{(d-1)^2}x^2f_z\Delta_{xx} \\
\partial_x^{(2)}\partial_z & -\frac{6}{(d-1)^2(d-2)}H_{yz}(\Delta_{xx})x^2z-\frac{6}{(d-1)^2}x^2f_y\Delta_{xx} \\
\partial_x\partial_y^{(2)} & -\frac{6}{(d-1)^2(d-2)}H_{yz}(\Delta_{xx})xy^2+\frac{12}{(d-1)^2}xyf_z\Delta_{xx} \\
\partial_x\partial_y\partial_z & -\frac{6}{(d-1)^2(d-2)}H_{yz}(\Delta_{xx})xyz-\frac{6}{(d-1)^2}(xyf_y-xzf_z)\Delta_{xx} \\
\partial_x\partial_z^{(2)} & -\frac{6}{(d-1)^2(d-2)}H_{yz}(\Delta_{xx})xz^2-\frac{12}{(d-1)^2}xzf_y\Delta_{xx} \\
\partial_y^{(3)} & -\frac{6}{(d-1)^2(d-2)}H_{yz}(\Delta_{xx})y^3+\frac{18}{(d-1)^2}y^2f_z\Delta_{xx}+6f_z^3\\
\partial_y^{(2)}\partial_z & -\frac{6}{(d-1)^2(d-2)}H_{yz}(\Delta_{xx})y^2z-\frac{6}{(d-1)^2}(y^2f_y-2yzf_z)\Delta_{xx}-6f_yf_z^2 \\
\partial_y\partial_z^{(2)} & -\frac{6}{(d-1)^2(d-2)}H_{yz}(\Delta_{xx})yz^2-\frac{6}{(d-1)^2}(2yzf_y-z^2f_z)\Delta_{xx}+6f_y^2f_z \\
\partial_z^{(3)} & -\frac{6}{(d-1)^2(d-2)}H_{yz}(\Delta_{xx})z^3-\frac{18}{(d-1)^2}z^2f_y\Delta_{xx}-6f_y^3
}
\end{align}
is divisible by $x^2$. We show that the $\partial_x\partial_z^{(2)}$ and $\partial_z^{(3)}$ entries are divisible by $x^2$; the calculations for the other entries are similar. 

For the $\partial_x\partial_z^{(2)}$ entry, we have
\begin{align*}
 -&\frac{6}{(d-1)^2(d-2)}H_{yz}(\Delta_{xx})xz^2-\frac{12}{(d-1)^2}xzf_y\Delta_{xx} \\
 &=-\frac{6}{(d-1)^2(d-2)}\left[(f_z\Delta_{xx,y}-f_y\Delta_{xx,z})xz^2+xzf_y(x\Delta_{xx,x}+y\Delta_{xx,y}+z\Delta_{xx,z})\right] \\
 &=-\frac{6}{(d-1)^2(d-2)}\left[xz^2f_z\Delta_{xx,y}+x^2zf_y\Delta_{xx,x}+xyzf_y\Delta_{xx,y}\right] \\
 &=-\frac{6}{(d-1)^2(d-2)}\left[x^2zf_y\Delta_{xx,x}-x^2zf_x\Delta_{xx,y}\right] \\
 &=-\frac{6x^2z}{(d-1)^2(d-2)}H_{xy}(\Delta_{xx})\,, 
\end{align*}
where the third equality follows from the Euler identity \cref{euler}.

For the $\partial_z^{(3)}$ entry, we have
\begin{align*}
-&\frac{6}{(d-1)^2(d-2)}H_{yz}(\Delta_{xx})z^3-\frac{18}{(d-1)^2}z^2f_y\Delta_{xx}-6f_y^3 \\
&=-\frac{6}{(d-1)^2(d-2)}\left[H_{yz}(\Delta_{xx})z^3+3(d-2)z^2f_y\Delta_{xx}+(d-2)f_y^3\right]\\
&=-\frac{6}{(d-1)^2(d-2)}\left[H_{yz}(\Delta_{xx})z^3+(d-2)f_y\left(3z^2\Delta_{xx}+2xz\Delta_{xz}-x^2\Delta_{zz}-z^2\Delta_{xx}\right)\right] \\
&=-\frac{6}{(d-1)^2(d-2)}\left[H_{yz}(\Delta_{xx})z^3+2(d-2)f_y\left(z^2\Delta_{xx}+xz\Delta_{xz}\right)-(d-2)x^2f_y\Delta_{zz}\right] \\
&=-\frac{6}{(d-1)^2(d-2)}\left[H_{yz}(\Delta_{xx})z^3+f_y\left(z^2E(\Delta_{xx})+xzE(\Delta_{xz})\right)-(d-2)x^2f_y\Delta_{zz}\right] 
\end{align*}
where the second equality follows from the identity \cref{PD-prodaa} in \cref{a:eulerappl}.

Applying the identity \cref{Rel-EHabHac} from \cref{a:subdetderiv} to rewrite $f_yE(\Delta_{xx})$ and $f_yE(\Delta_{xz})$, we have
\begin{align*}
f_y\left(z^2E(\Delta_{xx})+xzE(\Delta_{xz})\right)&=z^2\left(xH_{xy}(\Delta_{xx})-zH_{yz}(\Delta_{xx})\right)+xz\left(xH_{xy}(\Delta_{xz})-zH_{yz}(\Delta_{xz})\right) 
\end{align*}
so that the $\partial_z^{(3)}$ entry becomes
\begin{align*}
-&\frac{6}{(d-1)^2(d-2)}H_{yz}(\Delta_{xx})z^3-\frac{18}{(d-1)^2}z^2f_y\Delta_{xx}-6f_y^3 \\
&=-\frac{6}{(d-1)^2(d-2)}\left[xz^2\left(H_{xy}(\Delta_{xx})-H_{yz}(\Delta_{xz})\right)+x^2(zH_{xy}(\Delta_{xz})-(d-2)f_y\Delta_{zz})\right] \\
&=-\frac{6}{(d-1)^2(d-2)}\left[xz^2\left (-\frac{1}{d-1}x\delta_y\right )+x^2(zH_{xy}(\Delta_{xz})-(d-2)f_y\Delta_{zz})\right]\\
&=-\frac{6x^2}{(d-1)^2(d-2)}\left[zH_{xy}(\Delta_{xz})-(d-2)f_y\Delta_{zz}-\frac{1}{d-1}z^2\delta_y\right],
\end{align*}
where the second equality follows from the identity \cref{HamRel2} in \cref{a:Hrelns}. 

By similar calculations for the remaining entries, we find that the matrix \cref{zetaxmatrix} is given by 
\begin{align*}
-\frac{6x^2}{(d-1)^2(d-2)}\kbordermatrix{
& \\
\partial_x^{(3)} & xH_{yz}(\Delta_{xx}) \\
\partial_x^{(2)}\partial_y & yH_{yz}(\Delta_{xx})-(d-2)f_z\Delta_{xx} \\
\partial_x^{(2)}\partial_z & zH_{yz}(\Delta_{xx})+(d-2)f_y\Delta_{xx} \\
\partial_x\partial_y^{(2)} & yH_{zx}(\Delta_{xx}) \\
\partial_x\partial_y\partial_z & \frac{1}{2}\left(yH_{xy}(\Delta_{xx})+zH_{zx}(\Delta_{xx})\right) \\
\partial_x\partial_z^{(2)} & zH_{xy}(\Delta_{xx}) \\
\partial_y^{(3)} & yH_{zx}(\Delta_{xy})+\frac{1}{d-1}y^2\delta_z+(d-2)f_z\Delta_{yy} \\
\partial_y^{(2)}\partial_z & yH_{xy}(\Delta_{xy})-\frac{1}{d-1}yz\delta_z-(d-2)f_y\Delta_{yy} \\
\partial_y\partial_z^{(2)} & zH_{zx}(\Delta_{xz})+\frac{1}{d-1}yz\delta_y+(d-2)f_z\Delta_{zz} \\
\partial_z^{(3)} & zH_{xy}(\Delta_{xz})-\frac{1}{d-1}z^2\delta_y-(d-2)f_y\Delta_{zz}
}
\end{align*}
which is divisible by $x^2$, as desired. The corresponding statements about divisibility for $\zeta_y$ and $\zeta_z$ follow from $x,y,z$-symmetry.
\end{proof}

We will show that $\{E^3, E^2 H_{yz}, E^2 H_{zx}, E^2 H_{xy}, E\alpha_x, E\alpha_y, E\alpha_z, \zeta_x, \zeta_y, \zeta_z\}$ is a minimal generating set for $\ker J_{3,2}$ in \cref{subsec:exactnessS3}; see \cref{localgensS3}.

\begin{chunk}\label{M0(3)}
From the computations above, writing these generators in terms of the basis $\{\partial_x^{(2)} , \partial_x\partial_y , \partial_x\partial_z , \partial_y^{(2)} , \partial_y\partial_z , \partial_z^{(2)}\}$ gives the columns of the following matrix 
\[
M_0(3)\coloneqq
\begin{bmatrix}
E^3 & E^2H_{yz} & E^2H_{zx} & E^2H_{xy} & E\alpha_x & E\alpha_y & E\alpha_z & -\frac{6}{(d-1)^2(d-2)}Z 
\end{bmatrix},
\]
where
\begin{align*}
Z\coloneqq\scriptsize{\kbordermatrix{ 
& \zeta_x & \zeta_y & \zeta_z \\
& xH_{yz}(\Delta_{xx}) & xH_{yz}(\Delta_{xy})-\frac{1}{d-1}x^2\delta_z-(d-2)f_z\Delta_{xx} & xH_{yz}(\Delta_{xz})+\frac{1}{d-1}x^2\delta_y+(d-2)f_y\Delta_{xx} \\
& yH_{yz}(\Delta_{xx})-(d-2)f_z\Delta_{xx} & xH_{yz}(\Delta_{yy}) & xH_{zx}(\Delta_{xz})-\frac{1}{d-1}xy\delta_y-(d-2)f_x\Delta_{xx} \\
& zH_{yz}(\Delta_{xx})+(d-2)f_y\Delta_{xx} & xH_{xy}(\Delta_{xy})+\frac{1}{d-1}xz\delta_z+(d-2)f_x\Delta_{xx} & xH_{yz}(\Delta_{zz}) \\
& yH_{zx}(\Delta_{xx}) & xH_{zx}(\Delta_{yy})+(d-2)f_z\Delta_{yy} & yH_{yz}(\Delta_{yz})+\frac{1}{d-1}xy\delta_x+(d-2)f_y\Delta_{yy} \\
& \frac{1}{2}\left(yH_{xy}(\Delta_{xx})+zH_{zx}(\Delta_{xx})\right) & \frac{1}{2}\left(zH_{yz}(\Delta_{yy})+xH_{xy}(\Delta_{yy})\right) & \frac{1}{2}\left(xH_{zx}(\Delta_{zz})+yH_{yz}(\Delta_{zz})\right) \\
& zH_{xy}(\Delta_{xx}) & zH_{yz}(\Delta_{yz})-\frac{1}{d-1}xz\delta_x-(d-2)f_z\Delta_{zz} & xH_{xy}(\Delta_{zz})-(d-2)f_y\Delta_{zz} \\
& yH_{zx}(\Delta_{xy})+\frac{1}{d-1}y^2\delta_z+(d-2)f_z\Delta_{yy} & yH_{zx}(\Delta_{yy}) & yH_{zx}(\Delta_{yz})-\frac{1}{d-1}y^2\delta_x-(d-2)f_x\Delta_{yy} \\
& yH_{xy}(\Delta_{xy})-\frac{1}{d-1}yz\delta_z-(d-2)f_y\Delta_{yy} & zH_{zx}(\Delta_{yy})-(d-2)f_x\Delta_{yy} & yH_{zx}(\Delta_{zz}) \\
& zH_{zx}(\Delta_{xz})+\frac{1}{d-1}yz\delta_y+(d-2)f_z\Delta_{zz} & zH_{xy}(\Delta_{yy}) & yH_{xy}(\Delta_{zz})+(d-2)f_x\Delta_{zz} \\
& zH_{xy}(\Delta_{xz})-\frac{1}{d-1}z^2\delta_y-(d-2)f_y\Delta_{zz} & zH_{xy}(\Delta_{yz})+\frac{1}{d-1}z^2\delta_x+(d-2)f_x\Delta_{zz} & zH_{xy}(\Delta_{zz})} 
}.
\end{align*}
\end{chunk}

%%%%%%%%%%%%%%%%%%%%%%%%%%%%%%%%%%%%
\subsection{Matrix factorization}\label{subsec:mfS3}
%%%%%%%%%%%%%%%%%%%%%%%%%%%%%%%%%%%%

Throughout this subsection we recall the dg $E$-module $L$ constructed in \cref{c:alphamap}. In particular, its differentials $\partial_i^L$ are defined in \cref{e:Lcomplex}. 

\begin{chunk}
\label{c:betamap}
Consider the map $\beta\colon \Kos^Q(x,y,z)\to L$ given by 
\begin{equation*}
\begin{aligned} 
 \xymatrix@R+.5pc@C+.7pc{
%%%%%%%%%%%%%%%%%%%%%%%%%%%%%%%%%%%%%%%%%%%%%%%%%%% row 1
 0 \ar[r] & Q \ar[r]^{\partial_3}\ar@{=}[d]^{\beta_3} &Q^3 \ar[r]^{\partial_2} \ar[d]^{\beta_2}&Q^3 \ar[r]^{\partial_1}\ar[d]^{\beta_1}& Q \ar[r]\ar[d]^{\beta_0}&0\\
 %%%%%%%%%%%%%%%%%%%%%%%%%%%%%%%row 2
 0 \ar[r] & Q \ar[r]^{\partial_3^L} &Q^4 \ar[r]^{\partial_2^L}&Q^6 \ar[r]^{\partial_1^L}& Q^3 \ar[r]&0
}
\end{aligned} 
 \end{equation*}
where 
\[
\beta_2=\begin{bmatrix}-\alpha_1\\ \frac{1}{3}\sigma_1 \end{bmatrix}\,, \quad
\beta_1=\begin{bmatrix}
\alpha_2\\ \frac{1}{3}\sigma_2
\end{bmatrix}\,, \quad
\beta_0=\frac{1}{3}\sigma_3\,,
\]
and the matrices $\sigma_i$ are defined in \cref{sec:glossary2}.
First note that this is a chain map: Using \cref{a:matrixidentities}, the rightmost square commutes by the matrix identity \cref{c:identity-RS-4-2-1}, the middle square commutes by matrix identities \cref{c:d-D-alpha-q} and \cref{c:basic-Eul}, and the leftmost square commutes by Euler \cref{euler} applied to $f$ and $\delta$.

Next note that $\beta$ is a morphism of dg $E$-modules. Indeed, the $E$-action commutes with $\beta_0,\beta_1$ using \cref{c:identity-RS-4-2-1}. The $E$-action commutes with $\beta_1,\beta_2$ using \cref{c:basic-Eul} and \cref{c:D-to-sigmas}. The $E$-action commutes with $\beta_2,\beta_3$ just using the Euler identity.
 Therefore $\beta$ is a dg $E$-module map and so $\cone(\beta)$ naturally inherits a dg $E$-module structure. Consider the projection of complexes 
\begin{equation*}
\begin{aligned} 
 \xymatrix@R+.5pc@C+.7pc{
%%%%%%%%%%%%%%%%%%%%%%%%%%%%%%%%%%%%%%%%%%%%%%%%%%% row 1
 \cone(\beta):& 0 \ar[r] & Q \ar[r] \ar[d]&Q^4 \ar[r] \ar[d]^{\pi}&Q^7 \ar[r]\ar@{=}[d]& Q^7 \ar[r]\ar@{=}[d]&Q^3 \ar[r] \ar@{=}[d]&0\\
 %%%%%%%%%%%%%%%%%%%%%%%%%%%%%%%row 2
 & 0 \ar[r] & 0 \ar[r] &Q^3 \ar[r] &Q^7 \ar[r]& Q^7 \ar[r]&Q^3 \ar[r] &0\\
}
\end{aligned} 
 \end{equation*}
with $\pi=\begin{bmatrix}1& \del_3\end{bmatrix}$; this is a quasi-isomorphism that defines a dg $E$-module on the target. Namely, if we let $G$ denote the source of the projection, then $G$ is the dg $E$-module whose underlying complex is 
\[
G\colon\quad 0 \to Q^3
\xra{\begin{bmatrix}-\del_2 \\ -\alpha_1 \\ \frac{1}{3}\sigma_1 \end{bmatrix}}Q^7 
\xra{\begin{bmatrix} -\del_1 & 0 &0 \\
\alpha_2 & -D_2 & 0\\ \frac{1}{3}\sigma_2 & \alpha_2 &\del_3
\end{bmatrix}} Q^7 
\xra{\begin{bmatrix}\frac{1}{3}\sigma_3 & \alpha_1 & \del_2 \end{bmatrix}}Q^3 \to 0
\]
with $e\cdot$ given by 
\[
0 \leftarrow Q^3
\xla{\begin{bmatrix} D_2 & -q & 0 \end{bmatrix}}Q^7 
\xla{\begin{bmatrix} 
-D_3 & 0 & 0 \\
0 & \del_2 & 0 \\
0 & 0 &D_1
\end{bmatrix}} Q^7 
\xla{\begin{bmatrix}0 \\ q \\ -D_2 \end{bmatrix}}Q^3 \leftarrow 0\,.
\]
\end{chunk}

\begin{chunk}\label{c:explicitMF(3)}
Let $G$ be as constructed in \cref{c:betamap}. Applying \cref{c:fold} to it produces the matrix factorization $\fold(G)$ of $d\cdot f$ over $Q$: 
\[
Q^{10} \xra{\begin{bmatrix}
0 & -\del_1 & 0 & 0\\
q & \alpha_2 & -D_2 & 0 \\
-D_2 & \frac{1}{3}\sigma_2 & \alpha_2 & \del_3\\
0 & -D_2 & -q & 0 
\end{bmatrix}} Q^{10}\xra{\begin{bmatrix}
\frac{1}{3}\sigma_3 & \alpha_1 & \del_2 & 0 \\
-D_3 & 0 & 0 & -\del_2 \\
0 & \del_2 & 0 & -\alpha_1 \\
0 & 0 & D_1 & \frac{1}{3}\sigma_1
\end{bmatrix}}Q^{10}\,.
\] 
Changing bases yields the equivalent matrix factorization:
\[
Q^{10} \xra{\begin{bmatrix}
\del_3& D_2 & 2\alpha_2 & -2\sigma_2\\
0 & \frac{1}{2}q & D_2 & 3\alpha_2 \\
0 &0& \frac{1}{3}q & D_2\\
0 & 0 &0 & \del_1
\end{bmatrix}} Q^{10} \xra{\begin{bmatrix}
D_1 & 0 & -2\sigma_1 & 0\\
-\del_2 & 2\alpha_1 & 0&-2\sigma_3\\
0 & -\del_2 &3\alpha_1 & 0\\
0 & 0 &-\del_2 & D_3
\end{bmatrix}}Q^{10}\,.
\] Therefore, from \cref{c:hypersurface}, the complex
\[
\ldots \to R^{10}\xra{M_2(3)}R^{10}\xra{M_1(3)} R^{10}\xra{M_2(3)}R^{10}\xra{M_1(3)}R^{10}\to\ldots 
\]
is exact, where 
\[
M_2(3)\coloneqq \begin{bmatrix}
\del_3& D_2 & 2\alpha_2 & -2\sigma_2\\
0 & \frac{1}{2}q & D_2 & 3\alpha_2 \\
0 &0& \frac{1}{3}q & D_2\\
0 & 0 &0 & \del_1
\end{bmatrix}\quad \text{and}\quad M_1(3)\coloneqq
\begin{bmatrix}
D_1 & 0 & -2\sigma_1 & 0\\
-\del_2 & 2\alpha_1 & 0&-2\sigma_3\\
0 & -\del_2 &3\alpha_1 & 0\\
0 & 0 &-\del_2 & D_3
\end{bmatrix}\,.
\] 
\end{chunk}

%%%%%%%%%%%%%%%%%%%%%%%%%%%%%%%%%%%%
\subsection{Resolution of $\ker J_{3,2}$}\label{subsec:exactnessS3}
%%%%%%%%%%%%%%%%%%%%%%%%%%%%%%%%%%%%

In this subsection we establish exactness of the following sequence of $R$-modules 
\[
\cdots
\xrightarrow{M_1(3)}
R^{10}
\xrightarrow{M_2(3)}
R^{10}
\xrightarrow{M_1(3)}
R^{10}
\xrightarrow{M_0(3)} 
R^{10} 
\xrightarrow{J_{3,2}}
R^6 
\]
where the definitions of the matrices $M_i(3)$ can be found in \ref{M0(3)} and \ref{c:explicitMF(3)}. 

The proof of the following lemma is similar to \cref{localgensS2}.
\begin{lemma}
\label{localgensS3}
The $R_x$-module 
$(\ker_{R}(J_{3,2}))_x=\ker_{R_x}(J_{3,2})$ is freely generated by the operators $E^3$, $E^2H_{yz}$, $E\alpha_x$, and $H_{yz}^3$. 
\end{lemma}

\begin{lemma}\label{l:complexM0M1(3)}
We have $M_0(3) M_1(3)=0_{10\times10}$.
\end{lemma}

\begin{proof}
First note that the relations on the generators of $\ker J_{2,1}$ give 7 relations on the first 7 columns of $M_0(3)$. These are simply obtained by composing the generators of $\ker J_{2,1}$ by $E$ (and viewing the images in the filtration factor $D_3/D_2$); see \cref{SomeRelGen-3} for the details. To see that the remaining columns of $M_1(3)$ give relations, see \cref{c:OBLC-M1(3)} and \cref{lastrelnM1(3)}. 
\end{proof}

\begin{lemma}\label{l:localexactM0M1(3)}
We have $(\ker M_0(3))_x=(\im M_1(3))_x$.
\end{lemma}

\begin{proof}
The proof follows along the same lines as \cref{l:localexactM0M1(2)}. Namely, one can express the columns of $M_0(3)$ in terms of the free basis $E^3$, $E^2H_{yz}$, $E\alpha_x$, $H_{yz}^3$ above. Considering the $H_{yz}^3$ coefficients, the coefficients of $\zeta_x,\zeta_y,\zeta_z$ in the relation must be in the span of $\partial_2$. Subtracting off a suitable linear combination of columns 7--9 of $M_1(3)$, one can replace the given relation with one in which the $\zeta_x,\zeta_y,\zeta_z$ coefficients are zero. Repeating like so with the $\alpha_x,\alpha_y,\alpha_z$ coefficients, and then the $E^2 H_{yz}, E^2 H_{zx}, E^2 H_{xy}$ coefficients, one obtains a relation of the form $a E^3=0$, which must be the zero relation.
\end{proof}

We are now ready to prove the main result of this subsection. 

\begin{proposition}
\label{gensS3}
The sequence of $R$-modules 
\[
\cdots
\xrightarrow{M_1(3)}
R^{10}
\xrightarrow{M_2(3)}
R^{10}
\xrightarrow{M_1(3)}
R^{10}
\xrightarrow{M_0(3)} 
R^{10} 
\xrightarrow{J_{3,2}}
R^6 
\]
is exact. 
In particular, the $R$-module 
$\ker_{R}J_{3,2}$ is generated by the elements 
\[
\{E^3, E^2H_{yz}, E^2H_{zx}, E^2H_{xy}, E\alpha_x, E\alpha_y, E\alpha_z, \zeta_x, \zeta_y, \zeta_z\}\,.
\]
\end{proposition}

\begin{proof}
In view of \ref{c:explicitMF(3)} and Lemmas~\ref{localgensS3}, \ref{l:complexM0M1(3)}, and \ref{l:localexactM0M1(3)}, a proof parallel to that of Proposition~\ref{gensS2} can be used to prove the statement.
\end{proof}
%%%%%%%%%%%%%%%%%%%%%%%%%%%%%%%%%%%%
\subsection{Resolution of $D^3_{R|k}$}\label{subsec:liftS3}
%%%%%%%%%%%%%%%%%%%%%%%%%%%%%%%%%%%%

In this subsection we construct the resolution of $D^3_{R|k}$ by defining a ``lift" of $\theta_0(3)$; here it is interpreted as a map from the degree one part of the resolution of $\coker J_{3,2}$, constructed in \cref{subsec:exactnessS3}, to the complex $C=\cone(\theta(2))$ defined in \cref{l:liftS2}. 

\begin{chunk}
\label{lifts3diagram}
Consider the diagram 
\begin{equation*}
\begin{aligned} 
 %\xymatrixrowsep{1pc} \xymatrixcolsep{1pc}
 \xymatrix@R+.5pc@C+.7pc{
 % \xymatrix{
%%%%%%%%%%%%%%%%%%%%%%%%%%%%%%%%%%%%%%%%%%%%%%%%%%% row 1
\cdots\ar[r]&R^{10}\ar[d]^{\theta_3(3)} \ar[r]^{-M_2(3)}&R^{10}\ar[r]^{-M_1(3)}\ar[d]^{\theta_2(2)}& R^{10}\ar[r]^{-M_0(3)} \ar[d]^{\theta_1(3)}&R^{10} \ar[r]^{-J_{3,2}}\ar[d]^{\theta_0(3)}&R^6 \ar[r] &0\\
%%%%%%%%%%%%%%%%%%%%%%%%%%%%%%%%%%%%%%%%%%%%%%%%%% row 2
\cdots \ar[r]&R^{11} \ar[r]^{\partial_3^C} &R^{11} \ar[r]^{\partial_2^C}&R^9\ar[r]^{\partial_1^C}&R^4\ar[r]&0&
}
\end{aligned} 
 \end{equation*}
where $M_i(3)$ are defined in \cref{M0(3)} and \cref{c:explicitMF(3)}, and the $\theta_i(3)$ are declared in \cref{sec:glossary}.
\end{chunk}

The proof of the next lemma is similar to that of \cref{l:liftS2}. In light of this, and the lengthy computations involved, its proof is \cref{a:theta3chain}.

\begin{lemma}
\label{l:liftd3}
The diagram in \cref{lifts3diagram} commutes. That is, $\theta(3)$ is a morphism of complexes.
\end{lemma}
%%%%%%%%%%%%%%%%%%

\begin{theorem}
\label{resolutionD3}
Assume $R=k[x,y,z]/(f)$ is an isolated singularity hypersurface where $k$ is a field of characteristic zero. The augmented minimal $R$-free resolution of $D_{R|k}^3\subseteq R^9\oplus R^6\oplus R$ has the form 
\begin{align*}
 \cdots \to \begin{matrix}R^{11}\\\oplus\\R^{10} \end{matrix} \xra{\small \begin{bmatrix} \partial_4^C &\theta_3(3)\\0&M_2(3) \end{bmatrix}} 
 \begin{matrix}R^{11}\\\oplus\\R^{10} \end{matrix} \xra{\small \begin{bmatrix} \partial_3^C &\theta_2(3)\\0&M_1(3) \end{bmatrix}}\begin{matrix}R^{11}\\\oplus\\R^{10} \end{matrix} \xra{\small \begin{bmatrix} \partial_4^C &\theta_3(3)\\0&M_2(3) \end{bmatrix}} \begin{matrix}R^{11}\\\oplus\\R^{10} \end{matrix} \xra{\small \begin{bmatrix} \partial_3^C &\theta_2(3)\\0&M_1(3) \\ 0 & 0 \end{bmatrix}} \begin{matrix}R^{11}\\\oplus\\R^{10}\\\oplus \\R \end{matrix} \xra{\epsilon} D_{R|k}^3 \to 0
 \end{align*}
where a minimal set of generators for $D_{R|k}^3$ is given by the columns of 
 \[
 \epsilon=\begin{bmatrix} \partial_2^C &\theta_1(3)&0\\0&M_0(3)&0\\0&0&1 \end{bmatrix}
 \]
 and the block matrices making up the differentials are defined in \cref{l:liftS2}, \cref{M0(3)}, and \cref{c:explicitMF(3)}; see also \cref{sec:glossary}.
\end{theorem}

\begin{proof}
The proof is essentially the same as that of \cref{resolutionD2}; one uses \cref{gensS3,l:liftd3} in lieu of \cref{gensS2,l:liftS2}, respectively. 
\end{proof}

The theorem yields the following corollary which is analogous to \cref{cor:exactsequenceD2}.

\begin{corollary}
\label{cor:exactsequenced3}
There is a short exact sequence of $R$-modules 
\[
0\to D^2_{R|k}\to D^3_{R|k}\to \ker J_{3,2}\to 0
\]
where $D^2_{R|k}\to D^3_{R|k}$ is the inclusion.
\end{corollary}

\begin{question}
Is there a short exact sequence of $R$-modules 
\[
0\to D^{i-1}_{R|k}\to D^{i}_{R|k}\to \ker J_{i,i-1}\to 0
\]
where $D^{i-1}_{R|k}\to D^i_{R|k}$ is the inclusion for all $i>0$?
\end{question}

In contrast, the reader can find short exact sequences relating the cokernel of the inclusion $D^{i-1}_{R|k} \to D^{i}_{R|k}$ to global sections of certain sheaves in \cite{Vigue:1975}. 

Similarly to \cref{c:d1graded,cor:d2graded}, the differentials in the resolution of $D^3_{R|k}$ are homogeneous with respect to the internal grading of $R$ where 
\[
|\zeta_x|=|\zeta_y|=|\zeta_z|=3d-8\,.
\]
Hence, from \cref{resolutionD3} we can read off the graded Betti numbers of $D^3_{R|k}$.

\begin{corollary}
\label{cor:d3graded}
The graded Betti numbers of $D^3=D^3_{R|k}$ are given by
\[
\beta_{0,j}^R(D^3)=\begin{cases}
4 & j=0\\
9 & j= d-2\\
6 & j=2d-5\\
3 & j=3d -8\\
0 & \text{otherwise}
\end{cases}
\]
and for $n\geqslant 1$
\[
\beta_{2n-1,j}^R(D^3)=\begin{cases}
9 & j=nd -1\\
1 & j=nd+d-3\\
6 & j= nd +d -4\\
1 & j = nd+2d-6\\
3 & j= nd + 2d-7\\
1 & j =nd + 3d -9\\
0 & \text{otherwise}
\end{cases}
\quad \text{and}\quad \beta_{2n,j}^R(D^3)=\begin{cases}
3 & j=nd\\
9 & j=nd + d-2\\
6 & j=nd+2d-5\\
3 & j=nd + 3d -8\\
0 & \text{otherwise}
\end{cases}
\,.
\]
\end{corollary}

\begin{corollary}
\label{cor:levels3}
The following inequalities are satisfied
\[\level_{\Dsg(R)}^k (\ker J_{3,2})\leqslant 3\quad\text{and}\quad \level_{\Dsg(R)}^k (D_{R|k}^3)\leqslant 6\,.\] 
\end{corollary}
\begin{proof}
Using \cref{cor:exactsequenced3} and \cref{cor:levels2} we have the following 
\[
\level_{\Dsg(R)}^k (D_{R|k}^3)\leqslant\level_{\Dsg(R)}^k (\ker J_{3,2})+\level_{\Dsg(R)}^k (D_{R|k}^2)\leqslant \level_{\Dsg(R)}^k (\ker J_{3,2})+3\,,
\]
and so the first inequality implies the second. 
The first inequality is from \cref{c:betamap} and \cref{gensS3}.
\end{proof}
Based on the (partial) evidence in \cref{cor:levels2,cor:levels3}, as well as \cref{r:hz}, we ask the following. 
\begin{question}
\label{q:levels}
Let $R$ be an isolated singularity hypersurface ring with residue field $k$ of characteristic zero. For every positive integer $i$, do the following inequalities hold 
\[
\level_{\Dsg(R)}^k (\ker J_{i,i-1})\leqslant i\quad \text{and}\quad\level_{\Dsg(R)}^k (D_{R|k}^i)\leqslant i(i+1)/2\,?
\]
\end{question}

\begin{remark}
By \cite[Proposition~4.11]{Ballard/Favero/Katzarkov:2012}, for any $N$ in $\Dsg(R)$ the following inequality is satisfied
\[
\level_{\Dsg(R)}^k(N)\leqslant 2\ell\ell(R/(f_x,f_y,f_z))\,;
\]
here $\ell\ell(-)$ denotes the Loewy length of an $R$-module. 
So \cref{q:levels} has a positive answer whenever $n$ is at least $2\ell\ell(R/(f_x,f_y,f_z)).$ However it would still be enlightening if $\ker J_{i,i-1} $ can be built from $k$ using $i$ explicitly described mapping cones as was the case for $i=2,3$ in \cref{cor:levels2,cor:levels3}, respectively.
 \end{remark}
 
%%%%%%%%%%%%%%%%%%%%%%%%%%%%%%%%%%%%
\subsection{A set of generators of $D^3_{R|k}$}\label{subsec:gensD3}
%%%%%%%%%%%%%%%%%%%%%%%%%%%%%%%%%%%%

Now we write the generators of $D^3_{R|k}$ given in \cref{resolutionD3} as differential operators. In what follows, as in \cref{subsec:gensD2}, we write $E^2$, $E^3$, $EH_{ij}$, $E^2H_{ij}$, $H^3_{ij}$ for the representatives in $D^3_{R|k}$ specified in \cref{subsec:gensS2,subsec:gensS3}. For example, 
\[
E^3=x^3\del_x^3+
3x^2y\del_x^2\del_y+
3x^2z \del_x^2\del_z+
3xy^2\del_x\del_y^2 +
6xyz\del_x\del_y\del_z +
3xz^2\del_x\del_z^2 +
y^3\del_z^3 +
3y^2z\del_y^2\del_z +
3yz^2 \del_y\del_z^2+
z^3\del_z^3\,.
\]
Finally, we also adopt the notation from \cref{subsec:gensD2}.

We begin by writing some relevant operators of order at most three in terms of their images in the $\ker(J_{3,2})\cong D^3_{R|k}/D^2_{R|k}$ and lower order operators. Composing $\E$ with \cref{eq:order2opsE} and \cref{eq:order2opsEH} we have
\begin{equation}\label{eq:order3opsE}
 \E^3=E^3+3E^2+E
 =E^3+3\E^2-2\E
\end{equation}
\begin{equation}\label{eq:order3opsEH}
\begin{split}
 \E^2\H&=E^2H+dEH+(d-1)\E\H \\
 &=E^2H+(2d-1)EH+(d-1)^2\H \\
 &=E^2H+(2d-1)\E\H-d(d-1)\H,
\end{split}
\end{equation}
where here $\H$ denotes any one of the Hamiltonians, $\H_{yz}$, $\H_{zx}$, or $\H_{xy}$. Composing $\H$ with \cref{eq:order2opsH} we have
\begin{align}\label{eq:H3A}
 \H^3_{yz}=H^3_{yz}&+\H_{yz}(f_z^2)\partial_y^2-2\H_{yz}(f_yf_z)\partial_y\partial_z+\H_{yz}(f_y^2)\partial_z^2 \\ \label{eq:H3B}
 &-\frac{1}{d-1}\H_{yz}(\Delta_{xx})\E-\frac{1}{d-1}\Delta_{xx}\H_{yz}\E \\ \label{eq:H3C}
 &+\frac{1}{d-1}\left(\H_{yz}(x\Delta_{xx})\partial_x+\H_{yz}(x\Delta_{xy})\partial_y+\H_{yz}(x\Delta_{xz})\partial_z\right) \\ \label{eq:H3D}
 &+\frac{1}{d-1}\left(x\Delta_{xx}\H_{yz}\partial_x+x\Delta_{xy}\H_{yz}\partial_y+x\Delta_{xz}\H_{yz}\partial_z\right).
\end{align}
Noting that, by \cref{eq:order2opsEH}, we have $
 \H_{yz}\E=EH_{yz}+H_{yz}=\E\H_{yz}-(d-2)\H_{yz},$
simplifying \cref{eq:H3C} and \cref{eq:H3D}, and using the identities \cref{eq:MinDiff} to rewrite \cref{eq:H3A}, we find that 
\begin{align*}
 \H^3_{yz}=H^3_{yz}&+\frac{2}{d-1}\left((d-1)\Delta_{xx}\H_{yz}-\Delta_{xx}\E\H_{yz}+x\D_x\H_{yz}\right) \\ 
 &-\frac{1}{d-1}\H_{yz}(\Delta_{xx})\E-\frac{1}{d-1}\Delta_{xx}\left(\E\H_{yz}-(d-2)\H_{yz}\right) \\ 
 &+\frac{1}{d-1}x\left(\H_{yz}(\Delta_{xx})\partial_x+\H_{yz}(\Delta_{xy})\partial_y+\H_{yz}(\Delta_{xz})\partial_z\right) +\frac{1}{d-1}x\D_x\H_{yz}, 
\end{align*}
where $\D_x=\Delta_{xx}\partial_x+\Delta_{xy}\partial_y+\Delta_{xz}\partial_z$. Simplifying once more we obtain the following expression for $H^3_{yz}$:
\begin{equation}\label{eq:H3final}
\begin{split}
 \H^3_{yz}=H^3_{yz}&-\frac{1}{d-1}\big[3\Delta_{xx}\E\H_{yz}+\H_{yz}(\Delta_{xx})\E-(3d-4)\Delta_{xx}\H_{yz}-3x\D_x\H_{yz} \\ 
 &-x\left(\H_{yz}(\Delta_{xx})\partial_x+\H_{yz}(\Delta_{xy})\partial_y+\H_{yz}(\Delta_{xz})\partial_z\right)\big].
 \end{split}
\end{equation}
We obtain similar expressions for the cubes of the other Hamiltonians below.
\begin{align}\label{eq:H3finaly}
 \H^3_{zx}&=H^3_{zx}-\frac{1}{d-1}\big[3\Delta_{yy}\E\H_{zx}+\H_{zx}(\Delta_{yy})\E-(3d-4)\Delta_{yy}\H_{zx}-3y\D_yH_{zx} \\ 
 &\quad -y\left(\H_{zx}(\Delta_{xy})\partial_x+\H_{zx}(\Delta_{yy})\partial_y+\H_{zx}(\Delta_{yz})\partial_z\right)\big] \notag
\end{align}
\begin{align}\label{eq:H3finalz}
 \H^3_{xy}&=H^3_{xy}-\frac{1}{d-1}\big[3\Delta_{zz}\E\H_{xy}+\H_{xy}(\Delta_{zz})\E-(3d-4)\Delta_{zz}\H_{xy}-3z\D_z\H_{xy} \\ 
 &\quad -z\left(\H_{xy}(\Delta_{xz})\partial_x+\H_{xy}(\Delta_{yz})\partial_y+\H_{xy}(\Delta_{zz})\partial_z\right)\big] \notag
\end{align}

\begin{theorem}
\label{genops3}
A minimal set of generators for $D^3_{R|k}$ is given by $G_0\cup G_1\cup G_2\cup G_3$ where $G_0$, $G_1$, and $G_2$ are defined in \cref{genops2} and 
\begin{equation*}
G_3=\{\E^3,\quad\E^2\H_{yz},\quad\E^2\H_{zx},\quad\E^2\H_{xy},\quad \E\A_x,\quad \E\A_y\quad \E\A_z,\quad\Z_x,\quad\Z_y,\quad\Z_z\},
\end{equation*}
with 
\begin{align*}
\Z_x = \frac{1}{x^2} 
\Bigg[
\H_{yz}^3&+\frac{1}{2(d-1)^2(d-2)}\Bigg(6(d-2)\Delta_{xx}\E^2\H_{yz}-2\H_{yz}(\Delta_{xx})\E^3-6(d-2)\Delta_{xx}\E\H_{yz} \\
&-3(d-3)\H_{yz}(\Delta_{xx})\E^2-2(d-1)(d-2)(d-3)\Delta_{xx}\H_{yz}+(3d-7)\H_{yz}(\Delta_{xx})\E
\Bigg)\Bigg]
\\
\Z_y =\frac{1}{y^2} 
\Bigg[
\H_{zx}^3&+\frac{1}{2(d-1)^2(d-2)}\Bigg(6(d-2)\Delta_{yy}\E^2\H_{zx}-2\H_{zx}(\Delta_{yy})\E^3-6(d-2)\Delta_{yy}\E\H_{zx} \\
&-3(d-3)\H_{zx}(\Delta_{yy})\E^2-2(d-1)(d-2)(d-3)\Delta_{yy}\H_{zx}+(3d-7)\H_{zx}(\Delta_{yy})\E
\Bigg)\Bigg]
\\
\Z_z =\frac{1}{z^2} 
\Bigg[
\H_{xy}^3&+\frac{1}{2(d-1)^2(d-2)}\Bigg(6(d-2)\Delta_{zz}\E^2\H_{xy}-2\H_{xy}(\Delta_{zz})\E^3-6(d-2)\Delta_{zz}\E\H_{xy} \\
&-3(d-3)\H_{xy}(\Delta_{zz})\E^2-2(d-1)(d-2)(d-3)\Delta_{zz}\H_{xy}+(3d-7)\H_{xy}(\Delta_{zz})\E
\Bigg)\Bigg]
\end{align*}
\end{theorem}

\begin{proof}
The proof is similar to \cref{genops2}, so it suffices to show that $\Z_x$, $\Z_y$, and $\Z_z$ agree with the generators corresponding to the last three columns of $\epsilon$ in \cref{resolutionD3}, and thus they are well defined operators in $D^3_{R|k}$. 

Analyzing the column of the augmentation map $\epsilon$ in \cref{resolutionD3} that lifts $\zeta_x$, we see that the corresponding operator is given by 
\begin{align}\label{e:ZxopA}
 \Z_x&= \frac{1}{x^2}\Bigg[H_{yz}^3+\frac{3}{(d-1)^2}\Delta_{xx}E^2H_{yz}-\frac{1}{(d-1)^2(d-2)}H_{yz}(\Delta_{xx})E^3\Bigg] \\ \label{e:ZxopB}
 &+\frac{1}{(d-1)^2}(\delta_y\partial_z-\delta_z\partial_y) \\ \label{e:ZxopC}
 &-\frac{3}{(d-1)(d-2)}\Bigg[H_{yz}(\Delta_{xx})\partial_x^{(2)}+H_{zx}(\Delta_{xx})\partial_x\partial_y+H_{xy}(\Delta_{xx})\partial_x\partial_z+\left(H_{zx}(\Delta_{xy})+\frac{y\delta_z}{d-1}\right)\partial_y^{(2)} \\ \label{e:ZxopD}
 &+\frac{1}{2}\left( H_{zx}(\Delta_{xz})+H_{xy}(\Delta_{xy})+\frac{z\delta_z-y\delta_y}{d-1}\right)\partial_y\partial_z+\left(H_{xy}(\Delta_{xz})-\frac{z\delta_y}{d-1}\right)\partial_z^{(2)}\Bigg],
\end{align}
where \cref{e:ZxopA} comes from the formula for $\zeta_x$ in \cref{zetas}, and \cref{e:ZxopB}, \cref{e:ZxopC}, and \cref{e:ZxopD} come from the lift $\theta_1(3)$ defined in \cref{sec:glossary}. 

Next we use \cref{eq:order3opsE}, \cref{eq:order3opsEH}, and \cref{eq:H3final} to rewrite \cref{e:ZxopA}, as follows
\begin{align}\label{eq:Zxoporder1A}
 \Z_x= \frac{1}{x^2}\Bigg[\H_{yz}^3&+\frac{1}{d-1}\big[3\Delta_{xx}\E\H_{yz}-(3d-4)\Delta_{xx}\H_{yz} \\ \label{eq:Zxoporder1A2}
 &-3x\D_x\H_{yz} \\ \label{eq:Zxoporder1B}
 &+\H_{yz}(\Delta_{xx})\E-x\left(\H_{yz}(\Delta_{xx})\partial_x+\H_{yz}(\Delta_{xy})\partial_y+\H_{yz}(\Delta_{xz})\partial_z\right)\big] \\ \label{eq:Zxoporder1C}
 &+\frac{3}{(d-1)^2}\Delta_{xx}\left(\E^2\H_{yz}-(2d-1)\E\H_{yz}+d(d-1)\H_{yz}\right) \\ \label{eq:Zxoporder1D}
 &-\frac{1}{(d-1)^2(d-2)}\H_{yz}(\Delta_{xx})\left(\E^3-3\E^2+2\E\right)\Bigg] \\ \label{eq:Zxoporder1E}
 &+\frac{1}{(d-1)^2}(\delta_y\partial_z-\delta_z\partial_y) \\ \label{eq:Zxoporder1F}
 &-\frac{3}{(d-1)(d-2)}\Bigg[\H_{yz}(\Delta_{xx})\partial_x^{(2)}+\H_{zx}(\Delta_{xx})\partial_x\partial_y+\H_{xy}(\Delta_{xx})\partial_x\partial_z+\left(H_{zx}(\Delta_{xy})+\frac{y\delta_z}{d-1}\right)\partial_y^{(2)} \\ \label{eq:Zxoporder1G}
 &+\frac{1}{2}\left( \H_{zx}(\Delta_{xz})+\H_{xy}(\Delta_{xy})+\frac{z\delta_z-y\delta_y}{d-1}\right)\partial_y\partial_z+\left(\H_{xy}(\Delta_{xz})-\frac{z\delta_y}{d-1}\right)\partial_z^{(2)}\Bigg].
\end{align}
One can check that using the identities \cref{HamRel1} and \cref{HamRel2} to rewrite \cref{eq:Zxoporder1B}, combining it with \cref{eq:Zxoporder1E}, and simplifying using the Euler identity \cref{euler} yields
\begin{equation}\label{eq:Zxorder1terms}
\begin{split}
\frac{1}{x^2}&\Bigg[\frac{x^2}{(d-1)^2}(\delta_y\partial_z-\delta_z\partial_y)+\frac{1}{d-1}\Big(\H_{yz}(\Delta_{xx})\E-x\big(\H_{yz}(\Delta_{xx})\partial_x+\H_{yz}(\Delta_{xy})\partial_y+\H_{yz}(\Delta_{xz})\partial_z\big)\Big)\Bigg] \\
 &=\frac{2(d-2)}{x^2(d-1)}\Delta_{xx}\H_{yz}
\end{split}
\end{equation}
One can also check that combining \cref{eq:Zxoporder1A2} with \cref{eq:Zxoporder1F} and \cref{eq:Zxoporder1G} and using the Euler identity \cref{euler} to simplify yields
\begin{equation}\label{eq:Zxorder2terms}
\begin{split}
 -\frac{3}{x^2(d-1)}&\Bigg[x\D_x\H_{yz}+\frac{x^2}{d-2}\Big[\H_{yz}(\Delta_{xx})\partial_x^{(2)}+\H_{zx}(\Delta_{xx})\partial_x\partial_y+\H_{xy}(\Delta_{xx})\partial_x\partial_z+\left(H_{zx}(\Delta_{xy})+\frac{y\delta_z}{d-1}\right)\partial_y^{(2)} \\ 
 &+\frac{1}{2}\left( \H_{zx}(\Delta_{xz})+\H_{xy}(\Delta_{xy})+\frac{z\delta_z-y\delta_y}{d-1}\right)\partial_y\partial_z+\left(\H_{xy}(\Delta_{xz})-\frac{z\delta_y}{d-1}\right)\partial_z^{(2)}\Big]\Bigg] \\
 &=\frac{3}{2x^2(d-1)(d-2)}\Big(2(d-2)\Delta_{xx}EH_{yz}-\H_{yz}(\Delta_{xx})E^2\Big) \\
 &=\frac{3}{2x^2(d-1)(d-2)}\Big(2(d-2)\Delta_{xx}\big(\E\H_{yz}-(d-1)\H_{yz}\big)-\H_{yz}(\Delta_{xx})(\E^2-\E)\Big),
 \end{split}
\end{equation}
where the last equality follows from \cref{eq:order2opsE} and \cref{eq:order2opsEH}. Finally, using \cref{eq:Zxorder1terms} and \cref{eq:Zxorder2terms} to rewrite the formula for $\Z_x$ in \cref{eq:Zxoporder1A} through \cref{eq:Zxoporder1G} gives the desired formula for $\Z_x$. The proofs for $\Z_y$ and $\Z_z$ are similar. 
\end{proof}

\begin{remark}
\label{r:agreement3} 
Similar to \cref{r:agreement}, except for $\Z_x$, $\Z_y$, and $\Z_z$, the generators in \cref{genops3} do not agree with the generators of $D^3_{R|k}$ given in the augmentation map $\epsilon$ in \cref{resolutionD3}. However, one can subtract lower order operators to obtain generators that do agree. For the first four generators, such formulas follow directly from \cref{eq:order3opsE} and \cref{eq:order3opsEH}. One can check that the operator given by subtracting $\frac{5d-7}{2}\A_x$ from $\E\A_x$ agrees with the corresponding generator in \cref{resolutionD3}, and similarly for $\E\A_y$ and $\E\A_z$.
\end{remark}

%%%%%%%%%%%%%%%%%%%%%%%%%%%%%%%%%%%%%%%%%%%%%%%%%%%%%%%%%%%%%%
\appendix
\section{Identities}\label{appendix}
%%%%%%%%%%%%%%%%%%%%%%%%%%%%%%%%%%%%%%%%%%%%%%%%%%%%%%%%%%%%%%

In this appendix we collect identities that are used extensively throughout the paper. Let $f\in k[x,y,z]$ be homogeneous of degree
$d\geqslant 3$, where $k$ is a field of characteristic zero. 
Recall that
\begin{align}\label{euler}
 E(f)=xf_x+yf_y+zf_z=d\cdot f
\end{align}
in $k[x,y,z]$ (and $0$ in $R$), where $E=x\partial_x+y\partial_y+z\partial_z$ is the \emph{Euler operator}. Differentiating \cref{euler} with respect to $x,y,z$ produces the following system where $\{x,y,z\}=\{a,b,c\}$: 
\begin{equation}
 \label{2-1deriv}
 af_{aa}+bf_{ab}+cf_{ac}=(d-1) f_a\,.
\end{equation}

%%%%%%%%%%%%%%%%%%%
\subsection{Identities involving second order derivatives} Multiplying the first equation in the system above by $x$, the second one by $y$ and the third by $z$ and then adding them all together we obtain a second order Eulerian identity: 

\begin{equation}
\label{2nd order Euler id}
x^2 f_{xx}+y^2 f_{yy}+z^2 f_{zz}+2xy f_{xy}+2xzf_{xz}+2yz f_{yz}=0.\end{equation}

Importing $\Delta$ from \cref{sec:glossary}, we have
\[
\text{adj}(\Delta)=\begin{bmatrix} \Delta_{xx}& \Delta_{xy} &\Delta_{xz}\\ 
\Delta_{xy} &\Delta_{yy}& \Delta_{yz}\\
\Delta_{xz}&\Delta_{yz}& \Delta_{zz}
\end{bmatrix}\,, 
\]
and from the matrix identity $ \Delta \;\text{adj}(\Delta)=\delta I_{3\times3}$, where $I_{3\times 3}$ is the identity $3\times3$ matrix, we obtain the following identities with $\{x,y,z\}=\{a,b,c\}$:
\begin{subequations}
\label{det-exp}
\begin{align}
\delta&=f_{aa} \Delta_{aa}+f_{ab}\Delta_{ab}+f_{ac}\Delta_{ac}\label{det-exp1}\\
0&=f_{aa}\Delta_{ab}+f_{ab}\Delta_{bb}+f_{ac}\Delta_{bc}\label{det-exp4}.
\end{align}
\end{subequations}
An application of Cramer's rule to the system \cref{2-1deriv} yields for $\{x,y,z\}=\{a,b,c\}$: 
\begin{equation}
 \label{eqCram}
 a \delta=(d-1) (f_a \Delta_{aa}+f_b\Delta_{ab}+f_c \Delta_{ac})\,.
\end{equation}

Using equations \cref{2-1deriv} one can easily obtain the following useful identities: 
\begin{gather}
\begin{aligned}\label{eq:MinDiff}
x\Delta_{xy}-y\Delta_{xx}&= (d-1) ( f_z f_{yz}-f_y f_{zz}) & y\Delta_{xy}-x\Delta_{yy}&=(d-1) (f_z f_{xz}-f_x f_{zz}) \\
 x\Delta_{yz}-y\Delta_{xz}&= (d-1) ( f_y f_{xz}-f_x f_{yz}) &z\Delta_{yy}-y\Delta_{zy}&=(d-1) (f_z f_{xx}-f_x f_{xz})\\
 y\Delta_{zz}-z\Delta_{zy}&=(d-1) (f_y f_{xx}-f_x f_{xy}) & z\Delta_{xx}-x\Delta_{xz}&=(d-1) (f_z f_{yy}-f_y f_{zy})\\
 z\Delta_{xy}-x\Delta_{yz}&=(d-1) (f_x f_{zy}-f_z f_{xy})& z\Delta_{xz}-x\Delta_{zz}&=(d-1)(f_y f_{xy}-f_x f_{yy})\\
 z\Delta_{xy}-y\Delta_{xz}&=(d-1) (f_y f_{xz}-f_z f_{xy})\,.
\end{aligned}
\end{gather}
To prove these, one starts by the right hand side of the equations above and replaces $(d-1)f_x$, $(d-1) f_y$ and $(d-1)f_z$ 
by \cref{2-1deriv} respectively and then simplifies the resulting equations.

%%%%%%%%%%%%%%%%%%%
\subsection{Identities involving third order derivatives} 
Differentiating \cref{2-1deriv} in various manners yields the following system of equations for $\{a,b,c\}=\{x,y,z\}$:
\begin{align}
 af_{aaa}+bf_{aab}+cf_{aac}&=(d-2) f_{aa}\label{eq3-2aa}\\
af_{aab}+bf_{abb}+cf_{abc}&=(d-2) f_{ab}\label{eq3-2ab}\,.
\end{align}

From \cref{eq3-2aa} and \cref{eq3-2ab}, one obtains the following identities with $\{a,b,c\}=\{x,y,z\}$: 
\begin{equation}
 \label{eq2-3}
 a^2 f_{aaa}+b^2 f_{abb}+c^2 f_{acc}+2ab f_{aab}+2ac f_{acc}+2bc f_{abc}=(d-2)(d-1)f_a\,.
\end{equation}

From \eqref{eq2-3} 
we obtain the following \emph{third order} Eulerian identity:
\begin{equation}
\label{3rd-Euler}
x^3 f_{xxx}+y^3 f_{yyy}+z^3 f_{zzz}+3( x^2 y f_{xxy}+xy^2 f_{xyy}+ y^2 zf_{yyz}%+\\
+ yz^2 f_{yzz}+x^2 zf_{xxz}+xz^2 f_{xzz})+6xyz f_{xyz}=0\,.
\end{equation}

%%%%%%%%%%%%%%%%%%%

\subsection{Applications of the Eulerian identities}\label{a:eulerappl} We can express $f_x f_y, \;f_x f_z, \;f_y f_z$ as linear combinations of the cofactors of $\Delta$ with coefficients certain monomials in $x,y,z$. More precisely we have for $\{a,b,c\}=\{x,y,z\}$: 

\begin{align}
 f_a f_b&=\frac{1}{(d-1)^2} \left\{ c^2 \Delta_{ab}+ab \Delta_{cc}-ac \Delta_{bc}-bc\Delta_{ac}\right\}+\frac{d}{(d-1)} f_{ab} f\label{PD-prodab}\\ 
 f_a^2&=\frac{1}{(d-1)^2}\left\{-b^2\Delta_{cc}-c^2\Delta_{bb}+2bc\Delta_{bc}\right\}+\frac{d}{(d-1)} f_{aa}f\label{PD-prodaa}.
\end{align}

We shall prove \eqref{PD-prodab} with $a=y$, $b=z$ and $c=x$. The others follow similarly. 

\begin{proof}[Proof of \eqref{PD-prodab}] 
From \cref{2-1deriv} we obtain
\begin{align*}
f_y f_z&=\frac{1}{(d-1)^2} [ (xf_{xy}+yf_{yy}+zf_{yz}) (xf_{xz}+yf_{yz}+zf_{zz} )] \\
&=\frac{1}{(d-1)^2} [ x^2 f_{xz}f_{xy}+xy f_{yz}f_{xy}+xz f_{zz} f_{xy}+xy f_{xz} f_{yy}\\
&\quad+y^2 f_{yz}f_{yy}+yz f_{zz}f_{yy}+xz f_{yz} f_{xz}+yz f^2_{yz}+z^2 f_{zz}f_{yz} ] \\
&=\frac{1}{(d-1)^2} [ f_{yz} (y^2 f_{yy}+z^2 f_{zz}+yz f_{yz}+xz f_{xz}+ xy f_{xy})\\
&\quad+f_{xy}(x^2 f_{xz}+xz f_{zz})+ xy f_{xz}f_{yy}+yz f_{zz}f_{yy} ]\\
&=\frac{1}{(d-1)^2} [ f_{yz} (-x^2 f_{xx}-yz f_{yz}-xz f_{xz}-xyf_{xy}+d(d-1) f)\\
&\quad+f_{xy}(x^2 f_{xz}+xz f_{zz})+ xy f_{xz}f_{yy}+yz f_{zz}f_{yy}]\,. 
\end{align*}

Using \eqref{2nd order Euler id} to replace $(y^2 f_{yy}+z^2 f_{zz}+yz f_{yz}+xz f_{xz}+ xy f_{xy})$ with $-x^2 f_{xx}-yz f_{yz}-xz f_{xz}-xyf_{xy}+d(d-1)f$, we obtain:
\[
f_y f_z=\frac{1}{(d-1)^2} [x^2 \Delta_{yz}-xy \Delta_{xz}-xz\Delta_{xy}+yz\Delta_{xx}+d(d-1) f_{yz} f ]\,.
\]
\end{proof}

%%%%%%%%%%%%%%%%%%%
\subsection{Some facts on partial derivatives}
Now, if we differentiate \eqref{eqCram} with respect to $y$ and $z$, \eqref{eqCram} with respect to $x$ and $z$, \eqref{eqCram} with respect to $x$ and $y$ and use the identities (d-i) in \cref{det-exp}, we obtain the following with $\{x,y,z\}=\{a,b,c\}$:
\begin{equation}
 a \delta_b=(d-1) (f_a \Delta_{aa,b}+f_b \Delta_{ab,b}+f_c \Delta_{ac,b})\,.
 \label{eqDiffCram}
\end{equation}

Here $\Delta_{ab,c}$ denotes the derivative of $\Delta_{ab}$ with respect to the variable $c$. Moreover, if we differentiate \eqref{eqCram} with respect to $x,$ \eqref{eqCram} with respect to $y,$ \eqref{eqCram} with respect to $z,$ use the identities (a-c) in \cref{det-exp} and then simplify, we obtain for $\{x,y,z\}=\{a,b,c\}$:
\begin{equation}
 a \delta_a=(d-2) \delta +(d-1) (f_a \Delta_{aa,a}+f_b \Delta_{ab,a}+f_c \Delta_{ac,a})\label{eqDiff-Cram-same}\,.
\end{equation}

Let \[F\coloneqq \begin{bmatrix} f_{xxx}&f_{xyx}& f_{xzx}\\
f_{xyx}&f_{yyx}&f_{yzx}\\
f_{xzx}& f_{yzx}& f_{zzx}\\
\end{bmatrix}\,, \phantom{sfh} G\coloneqq \begin{bmatrix} f_{xxy}&f_{xyy}& f_{xzy}\\
f_{xyy}&f_{yyy}&f_{yzy}\\
f_{xzy}& f_{yzy}& f_{zzy}\\
\end{bmatrix}\,,\phantom{sjd}H\coloneqq \begin{bmatrix} f_{xxz}&f_{xyz}& f_{xzz}\\
f_{xyz}&f_{yyz}&f_{yzz}\\
f_{xzz}& f_{yzz}& f_{zzz}\\
\end{bmatrix}
\]
denote the matrices obtained by differentiating the columns of the matrix $\Delta$ with respect to $x, y, z$ respectively. 
If we differentiate (a-c) in \cref{det-exp} with respect to the variables $x,y,z$ respectively and simplify we can see that: 
\begin{equation}
\label{eq:Der-delt}
\delta_x=\text{tr}(\text{adj}(\Delta) F)\quad 
\delta_y=\text{tr}(\text{adj}(\Delta) G) \quad 
\delta_z=\text{tr}(\text{adj}(\Delta) H)\,,
\end{equation}
where $\text{tr}A$ denotes the trace of a matrix $A$. We will prove the first equality in \eqref{eq:Der-delt}. The others follow similarly. 

To see this, recall $\delta=f_{xx} \Delta_{xx}+f_{xy} \Delta_{xy}+f_{xz} \Delta_{xz}$. If we differentiate with respect to $x$ we obtain
\begin{align}
\delta_x&=(f_{xxx}\Delta_{xx}+f_{xyx}\Delta_{xy}+f_{xzx}\Delta_{xz})+f_{xx}\Delta_{xx,x}+f_{xy}\Delta_{xy,x}+f_{xz}\Delta_{xz,x}\label{eqComp}.
\end{align}
Differentiating the following subdeterminants
\[
\Delta_{xx}=f_{yy} f_{zz}-f^2_{yz}\,, \phantom{sds}
\Delta_{xy}=f_{xz}f_{yz}-f_{xy}f_{zz}\,, \phantom{sdsd}\Delta_{xz}=f_{xy} f_{yz}-f_{xz}f_{yy}
\] 
with respect to $x$  we obtain the following equations:
\begin{subequations}
\label{eq:MinDeri}
\begin{align}
\Delta_{xx,x}&=f_{yyx}f_{zz}+f_{yy}f_{zzx}-2f_{yz}f_{yzx}\label{eqMinDerxx,x}\\
\Delta_{xy,x}&=f_{xzx}f_{yz}+f_{zx}f_{yzx}-f_{xyx}f_{zz}-f_{xy}f_{zzx}\label{eqMinDerxy,x}\\
\Delta_{xz,x}&=f_{xyx} f_{yz}+f_{xy}f_{yzx}-f_{xzx}f_{yy}-f_{xz}f_{yyx}\,.\label{eqMinDerxz,x}
\end{align}
\end{subequations}
Now we see that 
\begin{align*}
f_{xx}\Delta_{xx,x}+f_{xy}\Delta_{xy,x}+f_{xx}\Delta_{xz,x}&=f_{xx}(f_{yyx}f_{zz}+f_{yy}f_{zzx}-2f_{yz}f_{yzx})\\
&\qquad+f_{xy} (f_{xzx}f_{yz}+f_{zx}f_{yzx}-f_{xyx}f_{zz}-f_{xy}f_{zzx})\\
&\qquad+f_{xz} (f_{xyx} f_{yz}+f_{xy}f_{yzx}-f_{xzx}f_{yy}-f_{xz}f_{yyx})\\
&=f_{xyx}(f_{xz}f_{yz}-f_{xy}f_{zz})+f_{yyx}(f_{xx}f_{zz}-f^2_{xy})+f_{xyz}(f_{zx}f_{xy}-f_{yz}f_{xx})\\
&\qquad+ f_{xyz}(f_{zx}f_{xy}-f_{yz}f_{xx})+f_{xzx}(f_{xy}f_{yz}-f_{yy}f_{xz})+f_{xzz}(f_{xx}f_{yy}-f^2_{xy})\\
&=(\Delta_{xy}f_{xyx}+\Delta_{yy}f_{yyx}+\Delta_{yz}f_{yzx})+(\Delta_{xz}f_{xzx}+\Delta_{yz}f_{yzx}+\Delta_{zz}f_{zzx})\,.
\end{align*}
So we have: 
\[
f_{xx}\Delta_{xx,x}+f_{xy}\Delta_{xy,x}+f_{xx}\Delta_{xz,x}=(\Delta_{xy}f_{xyx}+\Delta_{yy}f_{yyx}+\Delta_{yz}f_{yzx})+(\Delta_{xz}f_{xzx}+\Delta_{yz}f_{yzx}+\Delta_{zz}f_{zzx})\,.
\]
Putting this together with \eqref{eqComp} we see that $\delta_x=\text{tr}(\mathrm{adj}( \Delta) F)$, as claimed.

%%%%%%%%%%%%%%%%%%%
\subsection{Identities involving the derivatives of the subdeterminants $\Delta_{ab}$}\label{a:subdetderiv} We collect some useful identities involving the derivatives of the $(i,j)$
cofactors of $\Delta$:
\begin{equation}
\label{RelDerMin}
\Delta_{aa,a}+\Delta_{ab,b}+\Delta_{ac,c}=0\quad\text{with}\quad \{x,y,z\}=\{a,b,c\}\,.
\end{equation}

\begin{proof}[Proof of \eqref{RelDerMin}] 
We show this for one choice of $a,b,c,$ but the others follow from symmetry.  Differentiating the equation 
\[\Delta_{xy}=f_{xz}f_{yz}-f_{xy}f_{zz}\] 
with respect to $x$, we obtain the following equation
\[
\Delta_{xy,x}=f_{xzx}f_{yz}+f_{xz}f_{yzx}-f_{xyx}f_{zz}-f_{xy}f_{zzx}\,.
\]
Now, we have that $\Delta_{yy,y}=f_{xxy}f_{zz}+f_{xx}f_{zzy}-2f_{xz}f_{xzy}$ and $\Delta_{yz,z}=f_{xyz}f_{xz}+f_{xy}f_{xzz}-f_{xxz}f_{yz}-f_{xx}f_{yzz}$. 
Adding the last two equalities together we obtain 
\[
\Delta_{yy,y}+\Delta_{yz,z}=f_{xxy}f_{zz}-f_{xyz}f_{xz}-f_{xxz}f_{yz}+f_{xzz}f_{xy}=-\Delta_{xy,x}\,.\qedhere
\]
\end{proof}

%%%%%%%%%%%%%%%%%%%%%%%%%%%%%%%%%%%%%%%%%%%%%%%%%%%%%%%%%%%%%%
\section{Hamiltonian identities}\label{a:hamiltonian}
%%%%%%%%%%%%%%%%%%%%%%%%%%%%%%%%%%%%%%%%%%%%%%%%%%%%%%%%%%%%%%

In this appendix, we collect some identities involving the Hamiltonians acting on various cofactors of 
the matrix $\Delta$. These identities are used in the calculation of the entries of the matrix product $\theta_1(3) M_1(3)$ and thus in construction of the lift $\theta_2(3)$ in \cref{subsec:liftS3}. Recall
\begin{equation*}
H_{yz}=f_z\partial_y-f_y\partial_z, 
 \phantom{sdfs}
H_{zx}=f_x\partial_z-f_z\partial_x,
 \phantom{sdfs}
H_{xy}=f_y\partial_x-f_x\partial_y
\end{equation*}
and we define $H_{zy}=-H_{yz}$, $H_{xz}=-H_{zx}$, $H_{yx}=-H_{xy}$. 
There are some fundamental relations among all the Hamiltonians and the Euler derivations:
\begin{subequations}
\begin{align}
&f_x H_{yz}+f_y H_{zx}+f_z H_{xy}=0\label{eq:FDR-Ham}\\
&f_a E-c H_{ca}+ b H_{ab}=0\quad\text{with}\quad\{a,b,c\}=\{x,y,z\}\,.\label{Rel-EHabHac}
\end{align}
\end{subequations} 

%%%%%%%%%%%%%%%%%%%
\subsection{Relations among $H_{ij}(\Delta_{kl})$.}\label{a:Hrelns} We have the following identities with $\{a,b,c\}=\{x,y,z\}$: 

\begin{align}
 H_{ca}(\Delta_{aa}) &-H_{bc}(\Delta_{ab})=\frac{a \delta_c}{(d-1)}\label{HamRel1}\\ 
 H_{bc}(\Delta_{bc})&-H_{ab}(\Delta_{ab})=\frac{b\delta_b-(d-2)\delta}{(d-1)}\label{HamRel2}\,.
\end{align}
Using \cref{HamRel1} we obtain
\begin{align}
\label{HamRel4+5Comb}
2\,H_{yz}(\Delta_{yz})=H_{zx}(\Delta_{xz})+H_{xy}(\Delta_{xy})+\frac{y\delta_y-z\delta_z}{(d-1)}\\
\label{HamRel4-5Comb}
0=H_{xy}(\Delta_{xy})-H_{zx}(\Delta_{xz})+\frac{y\delta_y+z\delta_z}{(d-1)}-\frac{2(d-2)\delta}{(d-1)}\,.
\end{align}
Additionally, if we replace $y\delta_y+z\delta_z$ by its equal $3(d-2)\delta-x\delta_x,$ then 
\eqref{HamRel4-5Comb} becomes
\begin{equation}
\label{HamRel4-5Comb2}
0=H_{xy}(\Delta_{xy})-H_{zx}(\Delta_{xz})+\frac{(d-2)\delta-x\delta_x}{(d-1)}\,.
\end{equation}
We present the proofs \eqref{HamRel1} with $a=x$, $b=y$ and $c=z$ and the remaining follow along the same lines. 

\begin{proof}
We use the identities in \cref{eqDiffCram} from \cref{appendix} to verify the identity:

\begin{align*}
 H_{yz}(\Delta_{xy})&=f_z\Delta_{xy,y}-f_y\Delta_{xy,z}\\
 &=-f_z\Delta_{xx,x}-f_z\Delta_{xz,z}-f_y\Delta_{xy,z}\\
&=H_{zx}(\Delta_{xx})-\frac{x\delta_z}{d-1}\,;
\end{align*}
the second equality uses \cref{RelDerMin} while the third uses \cref{eqDiffCram}.
\end{proof}

%%%%%%%%%%%%%%%%%%%
\subsection{Identities on how $H_{ij}$ act on some elements in $k[x,y,z]$} Here we collect some useful identities on how $H_{ij}$ act on \cref{2-1deriv},
\eqref{eqCram}. 
 Note that $H_{yz}(x)=0$, $H_{yz}(y)=f_z$, and $H_{yz}(z)=-f_y$. Moreover, using the equalities in \eqref{eq:MinDiff} we have:
\begin{subequations}
\label{H-on-partial f}
\begin{align}
H_{yz}(f_x)&=f_z f_{yx}-f_y f_{zx}=\frac{1}{(d-1)} (y\Delta_{xz}-z\Delta_{xy})\label{H-on-fx}\\
H_{yz}(f_y)&=f_z f_{yy}-f_y f_{zy}=\frac{1}{(d-1)} (z\Delta_{xx}-x\Delta_{xz})\label{H-on-fy}\\
H_{yz}(f_z)&=f_z f_{yz}-f_y f_{zz}=\frac{1}{(d-1)} (x\Delta_{xy}-y\Delta_{xx})\,.\label{H-on-fz}
\end{align}
\end{subequations}
If we apply the various Hamiltonians to \eqref{eqCram} we obtain the following with $\{a,b,c\}=\{x,y,z\}$:
\begin{subequations}
\label{eq:H-on-Cram}
\begin{align}
f_a H_{bc}(\Delta_{aa})+f_b H_{bc}(\Delta_{ab})+f_c H_{bc}(\Delta_{bc})&=\frac{a(f_c\delta_b-f_b\delta_c)}{d-1}\label{H-on-Cram1}
\\
f_a H_{bc}(\Delta_{ab})+f_b H_{bc}(\Delta_{bb})+f_c H_{bc}(\Delta_{bc})&=\frac{b(f_c\delta_b-f_b\delta_c)}{d-1}-\frac{(d-2)\delta}{d-1} f_c\label{H-on-Cram2}
\\
f_a H_{bc}(\Delta_{ac})+f_b H_{bc}(\Delta_{bc})+f_c H_{bc}(\Delta_{cc})&=\frac{c(f_c\delta_b-f_b\delta_c)}{d-1}+\frac{(d-2)\delta}{d-1} f_b\label{H-on-Cram3}
\\
H_{bc}(f_a)\Delta_{aa}+H_{bc}(f_b)\Delta_{ab}+H_{bc}(f_c)\Delta_{ac}&=0\,.\label{H-on-Cram4}
\end{align}
\end{subequations}

\begin{remark} 
Some of the equations above can be rewritten in different ways. For example if we look at the equation \eqref{H-on-Cram1}
using the basic identity $x\delta_x+y\delta_y+z\delta_z=3(d-2)\delta$ and rewriting $xf_x=-yf_y-zf_z$ (over $R$) or $xf_x=d\cdot f-yf_y-zf_z$ on $k[x,y,z]$ 
we can write

\begin{align*}
\frac{x(f_x\delta_z-f_z\delta_x)}{d-1}+\frac{(d-2)\delta}{d-1}f_z
&=\frac{\delta_z (d\cdot f-yf_y-zf_z)-x f_z \delta_x+(d-2)\delta f_z}{d-1}
\\
&=\frac{-f_z(x\delta_x+z\delta_z)-yf_y\delta_z+(d-2)\delta f_z}{d-1}+d\frac{d \delta_z}{d-1}f
\\
&=\frac{-f_z(3(d-2)\delta-y\delta_y)-yf_y\delta_z+(d-2)\delta f_z}{d-1}+d\frac{d \delta_z}{d-1}f
\\
&=\frac{y(f_z\delta_y-f_y\delta_z)}{d-1}-\frac{2(d-2)\delta}{d-1} f_z+d\frac{d \delta_z}{d-1}f\,.
\end{align*}
\end{remark} 

We prove the identities in \eqref{eq:H-on-Cram} for $ a=x$, $b=y$, and $c=z$. The remaining identities follow similarly. 

\begin{proof}[Proof of \eqref{eq:H-on-Cram}] 
We first prove $H_{yz}(f_x)\Delta_{xx}+H_{yz}(f_y)\Delta_{xy}+H_{yz}(f_z)\Delta_{xz}=0.$ We have:
\begin{align*}
H_{yz}(f_x)\Delta_{xx}+H_{yz}(f_y)\Delta_{xy}+H_{yz}(f_z)\Delta_{xz}
&=(f_z f_{xy}-f_y f_{zx})\Delta_{xx}+(f_z f_{yy}-f_y f_{yz})\Delta_{xy}+(f_z f_{yz}-f_y f_{zz})\Delta_{xz}
\\
&=f_z(f_{xy}\Delta_{xx}+f_{yy}\Delta_{xy}+f_{yz}\Delta_{xz})-f_y(f_{xz}\Delta_{xx}+f_{yz}\Delta_{xy}+f_{zz}\Delta_{zz})
\\
&=0\,,
\end{align*}
where the last equality uses \cref{det-exp}.
Applying $H_{yz}$ to \eqref{eqCram} yields 
\begin{align*}
xH_{yz}(\delta)
&=(d-1) [ f_x H_{yz}(\Delta_{xx})+f_y H_{yz}(\Delta_{xy})+f_z H_{yz}(\Delta_{xz})
\\
&\qquad +H_{yz}(f_x)\Delta_{xx}+H_{yz}(f_y)\Delta_{xy}+H_{yz}(f_z)\Delta_{xz}]
\\
&=
(d-1) [ f_x H_{yz}(\Delta_{xx})+f_y H_{yz}(\Delta_{xy})+f_z H_{yz}(\Delta_{xz})]\,,
\end{align*}
where the second equality is from \cref{H-on-Cram4}.

Hence \[f_x H_{yz}(\Delta_{xx})+f_y H_{yz}(\Delta_{xy})+f_z H_{yz}(\Delta_{xz})=\frac{x H_{yz}(\delta)}{d-1}\,,\] and thus \eqref{H-on-Cram1}
is proved. 

For \eqref{H-on-Cram2} we use \eqref{HamRel1} and \eqref{HamRel2}, rewrite $H_{yz}(\Delta_{yy})$ in terms of $H_{zx}(\Delta_{xy})$, rewrite
$H_{yz}(\Delta_{yz})$ in terms of $H_{xy}(\Delta_{xy})$, and then use the fundamental relation among the Hamiltonians \eqref{eq:FDR-Ham} applied to $\Delta_{xy}$. 
So we have:

\begin{align*}
f_x H_{yz}(\Delta_{xy})&+f_y H_{yz}(\Delta_{yy})+f_z H_{yz}(\Delta_{yz})\\
&=f_x H_{yz}(\Delta_{xy})+f_y\left(H_{zx}(\Delta_{xy})-\frac{y\delta_z}{d-1}\right)+f_z\left(H_{xy}(\Delta_{xy})+\frac{y\delta_y-(d-2)\delta}{d-1}\right)\\
&=f_x H_{yz}(\Delta_{xy})+f_y H_{zx}(\Delta_{xy})+f_z H_{xy}(\Delta_{xy})+\frac{y(f_z\delta_y-f_y\delta_z)}{d-1}-\frac{(d-2)\delta}{d-1} f_z\\
&=\frac{y(f_z\delta_y-f_y\delta_z)}{d-1}-\frac{(d-2)\delta}{d-1} f_z\,,
\end{align*}
where the last equality uses \cref{eq:FDR-Ham}.
Finally for \eqref{H-on-Cram3} we will use \eqref{HamRel1} and \eqref{HamRel2}, rewrite $H_{yz}(\Delta_{yz})$ in terms of $H_{zx}(\Delta_{xz})$, rewrite
$H_{yz}(\Delta_{zz})$ in terms of $H_{xy}(\Delta_{xz})$, and then use the fundamental relation among the Hamiltonians \eqref{eq:FDR-Ham} applied to $\Delta_{xz}$. 

So we have:
\begin{align*}
f_x H_{yz}(\Delta_{xz})&+f_y H_{yz}(\Delta_{yz})+f_z H_{yz}(\Delta_{zz})\\
&=f_x H_{yz}(\Delta_{xz})+f_y\left(H_{zx}(\Delta_{xz})-\frac{z\delta_z-(d-2)\delta}{d-1}\right)+f_z\left(H_{xy}(\Delta_{xz})+\frac{z\delta_y}{d-1}\right)\\
&=f_x H_{yz}(\Delta_{xz})+f_y H_{zx}(\Delta_{xz})+f_z H_{xy}(\Delta_{xz})+\frac{z(f_z\delta_y-f_y\delta_z)}{d-1}+\frac{(d-2)\delta}{d-1} f_y\\
&\overset{\eqref{eq:FDR-Ham}}=\frac{z(f_z\delta_y-f_y\delta_z)}{d-1}+\frac{(d-2)\delta}{d-1} f_z\,.\qedhere
\end{align*}
\end{proof}

The remaining identities are used in \cref{a:theta3chain}.

By direct calculation one can show: 
\begin{equation}
\label{eq:Hij-on-2der}
\begin{split}
&H_{yz}(f_{xx})+H_{zx}(f_{xy})+H_{xy}(f_{xz})=0\\
&H_{yz}(f_{xy})+H_{zx}(f_{yy})+H_{xy}(f_{yz})=0\\
&H_{yz}(f_{xz})+H_{zx}(f_{yz})+H_{xy}(f_{zz})=0\,.
\end{split}
\end{equation}
We also see how the Hamiltonians $H_{ij}$ act on \cref{det-exp1} and \cref{det-exp4}.
Applying $H_{yz}$ to \eqref{det-exp1} yields
\[
H_{yz}(\delta)=H_{yz}(f_{xx}) \Delta_{xx}+H_{yz}(f_{xy})\Delta_{xy}+H_{yz}(f_{xz})\Delta_{xz}
+[f_{xx} H_{yz}(\Delta_{xx})+f_{xy}H_{yz}(\Delta_{xy})+f_{xz}H_{yz}(\Delta_{xz})]
\]
and hence 
\begin{equation}
\label{H-on-detexp1}
\begin{split}
f_{xx} H_{yz}(\Delta_{xx})+f_{xy}H_{yz}(\Delta_{xy})+f_{xz}H_{yz}(\Delta_{xz})&=\hspace{1in}\\
&H_{yz}(\delta)-[H_{yz}(f_{xx}) \Delta_{xx}+H_{yz}(f_{xy})\Delta_{xy}+H_{yz}(f_{xz})\Delta_{xz}]\,.
\end{split}
\end{equation}
Similarly if we apply $H_{zx}$ to 
\cref{det-exp4} we have:
\begin{align*}
0=&H_{zx}(f_{xy}) \Delta_{xx}+H_{zx}(f_{yy})\Delta_{xy}+H_{zx}(f_{yz})\Delta_{xz}+ [f_{xy} H_{zx}(\Delta_{xx})+f_{yy}H_{zx}(\Delta_{xy})+f_{yz}H_{zx}(\Delta_{xz})]
\end{align*}
and hence 
\begin{equation}
\label{Hzx-on-detexp6}
\begin{split}
f_{xy} H_{zx}(\Delta_{xx})+f_{yy}H_{zx}(\Delta_{xy})+f_{yz}H_{zx}(\Delta_{xz})=
&-[H_{zx}(f_{xy}) \Delta_{xx}+H_{zx}(f_{yy})\Delta_{xy}+H_{zx}(f_{yz})\Delta_{xz}]\,.
\end{split}
\end{equation}
And finally if we apply $H_{xy}$ to 
\cref{det-exp4} we have:
\begin{align*}
0=&H_{xy}(f_{xz}) \Delta_{xx}+H_{xy}(f_{yz})\Delta_{xy}+H_{xy}(f_{zz})\Delta_{xz}+[f_{xz} H_{xy}(\Delta_{xx})+f_{yy}H_{xy}(\Delta_{xy})+f_{zz}H_{xy}(\Delta_{xz})]
\end{align*}
and hence
\begin{equation}
\label{Hxy-on-detexp8}
\begin{split}
f_{xz} H_{xy}(\Delta_{xx})+f_{yz}H_{xy}(\Delta_{xy})+f_{zz}H_{xy}(\Delta_{xz})=
&-[H_{xy}(f_{xz}) \Delta_{xx}+H_{xy}(f_{yz})\Delta_{xy}+H_{xy}(f_{zz})\Delta_{xz}]\,.
\end{split}
\end{equation}

\begin{remark}
 Analogous equations hold if we apply the other Hamiltonians to the equations in \eqref{det-exp}. 
\end{remark}

%%%%%%%%%%%%%%%%%%%%%%%%%%%%%%%%%%%%%%%%%%%%%%%%%%%%%%%%%%%%%%
\section{Computations for {$M_0(3) M_1(3)=0$}}\label{a:genrelforS3}
%%%%%%%%%%%%%%%%%%%%%%%%%%%%%%%%%%%%%%%%%%%%%%%%%%%%%%%%%%%%%%

%%%%%%%%%%%%%%%%%%%
\subsection{From relations among generators of $\ker(J_{2,1})$ to relations among generators of $\ker(J_{3,2})$.} 
Recall from \cref{l:complexM0M1(2)} that the following relations hold for the generators of $\ker(J_{2,1})$: 
\begin{align}
&f_a E^2 -c EH_{ca}+b EH_{ab}=0,~ \text{for}~ \{a,b,c\}=\{x,y,z\}\label{E2-EHca-EHab}\\
&\frac{2}{d-1} \left( f_{aa} EH_{bc}+f_{ab} EH_{ca}+f_{ac}EH_{ab}\right)-c\alpha_b+b\alpha_c=0~ \text{for}~ \{a,b,c\}=\{x,y,z\}\label{EHam-alphab-c}\\
&-\frac{2\delta}{(d-1)^3} E^2+f_x \alpha_x+f_y \alpha_y+f_z\alpha_z=0\label{E2-alphas}\,.
\end{align}

Precomposing with $E$ we obtain seven relations amongst the generators of $\ker(J_{3,2})$, corresponding to the first seven columns of $M_1(3)$. \label{SomeRelGen-3}

We shall now describe what will turn out to be the remaining relations among the generators. 
They correspond to columns 8--10 of $M_1(3)$. 

\begin{proposition}\label{c:OBLC-M1(3)} We have over $R$ that 
\begin{subequations}
\begin{align}
-\frac{2\delta_x}{d'^3 (d-2)} E^3+\frac{3}{d-1} \left( f_{xx} E\alpha_x+f_{xy}E\alpha_y+f_{xz}E\alpha_z\right)-z\, \zeta_y+y\, \zeta_z&=0\label{c:eq1OBLC-M1(3)}\\
-\frac{2\delta_y}{d'^3 (d-2)} E^3+\frac{3}{d-1} \left( f_{xy} E\alpha_x+f_{yy}E\alpha_y+f_{yz}E\alpha_z\right)+z \,\zeta_x-x\, \zeta_z&=0\label{c:eq2OBLC-M1(3)}\\
-\frac{2\delta_z}{d'^3 (d-2)} E^3+\frac{3}{d-1} \left( f_{xz} E\alpha_x+f_{yz}E\alpha_y+f_{zz}E\alpha_z\right)-y\, \zeta_x+x\, \zeta_y&=0\label{c:eq3OBLC-M1(3)}\,.
\end{align}
\end{subequations}
\end{proposition}

\begin{proof} We shall prove \eqref{c:eq1OBLC-M1(3)}. The rest follow similarly. We will use $d'\coloneqq d-1$.

Recall that $H^2_{xy}+\frac{\Delta_{zz}}{d'^2} E^2=z\alpha_z$. If we apply $E$ to it and look at the ``cubic order terms'' we obtain 
\label{EH2xy-equiv}
\begin{equation}
EH^2_{xy}=zE\alpha_z-\frac{\Delta_{zz}}{d'^2} E^3\,.
\end{equation}
We have:
\begin{align*}\label{Rel-2-zetas}
-z\zeta_y+y\zeta_z&=-\frac{z}{y^2 }\left(\frac{1}{z^3} (z H_{zx})^3 +\frac{3\,\Delta_{yy}}{d'^2} E^2H_{zx}-\frac{H_{zx}(\Delta_{yy})}{d'^2 (d-2)}E^3\right)\\
 & \qquad +\frac{y}{z^2}\left( H^3_{xy}
+\frac{3\,\Delta_{zz}}{d'^2} E^2 H_{xy}-\frac{H_{xy}(\Delta_{zz})}{d'^2 (d-2)}E^3\right)\\
 & =
-\frac{z}{y^2} \left [ \frac{1}{z^3} \left( f_x^3 E^3+3y f_x^2 E^2H_{xy}+3y^2 f_x EH^2_{xy}+y^3 H^3_{xy} \right)+\frac{3\,\Delta_{yy}}{d'^2} E^2H_{zx} \right .\\
 &\quad \quad \left . -\frac{H_{zx}(\Delta_{yy})}{d'^2 (d-2)}E^3 \right ]
+\frac{y}{z^2}\left( H^3_{xy}+\frac{3\,\Delta_{zz}}{d'^2} E^2 H_{xy}-\frac{H_{xy}(\Delta_{zz})}{d'^2 (d-2)}E^3\right)\\
&=-\frac{z}{y^2} \left [ \frac{1}{z^3} \left( f_x^3 E^3+3y f_x^2 E^2H_{xy}+3y^2 f_x (zE\alpha_z-\frac{\Delta_{zz}}{d'^2}E^3)+ {y^3 H^3_{xy}} \right) \right .\\
 &\quad \quad \left . +\frac{3\,\Delta_{yy}}{d'^2} E^2H_{zx}-\frac{H_{zx}(\Delta_{yy})}{d'^2 (d-2)}E^3 \right ]
+\frac{y}{z^2}\left( { H^3_{xy}}+\frac{3\,\Delta_{zz}}{d'^2} E^2 H_{xy}-\frac{H_{xy}(\Delta_{zz})}{d'^2 (d-2)}E^3\right)\\
 &=\left [-\frac{f_x^3}{y^2 z^2}+\frac{3 f_x\Delta_{zz}}{d'^2 z^2}+\frac{z H_{zx}(\Delta_{yy})}{d'^2 (d-2) y^2}-\frac{y H_{xy}(\Delta_{zz})}{d'^2 (d-2) z^2}\right ] E^3 \\
 &\quad\quad +
 \left( -\frac{3 f_x^2}{y z^2}+\frac{3 y \Delta_{zz}}{d'^2 z^2}\right)E^2H_{xy}-\frac{3\Delta_{yy}}{d'^2 y^2} { zE^2 H_{zx}}-\frac{3f_x}{z}E\alpha_z\\
 & =\left [-\frac{f_x^3}{y^2 z^2}+\frac{3f_x\Delta_{zz}}{d'^2 z^2}-\frac{3 f_x\Delta_{yy}}{d'^2 y^2}+\frac{1}{d'^2 (d-2)}\left( \frac{z H_{zx}(\Delta_{yy})}{y^2}-\frac{y H_{xy}(\Delta_{zz})}{z^2}\right) \right] E^3\\
 & \quad \quad +\left(-\frac{3 f_x^2}{yz^2}+\frac{3 y \Delta_{zz}}{d'^2 z^2}-\frac{3\Delta_{yy}}{d'^2 y}\right) E^2 H_{xy}-\frac{3 f_x}{z} E\alpha_z\,,
\end{align*}
where the second equality uses $zH_{zx}=f_x E+yH_{xy}$ and in order to get the last equality we use \eqref{E2-EHca-EHab} and rewrite $z E^2H_{zx}$ in terms of $E^3, E^2H_{xy}$ via the formula $zE^2 H_{zx}=f_x E^3+y E^2H_{xy}.$
To conclude the proof of \eqref{c:eq1OBLC-M1(3)} we need the following lemma. 

\begin{lemma}\label{Rel-Ealphas} The following equality holds:
\begin{equation}
\frac{3}{d'} \left (f_{xx} E\alpha_x+f_{xy}E\alpha_y+f_{xz} E\alpha_z\right)=\frac{3\,f_x}{z} E\alpha_z+\frac{6}{d'^2\,z^2} \left( z\Delta_{yz} -y\Delta_{zz}\right) E^2H_{xy}-\frac{6\,f_x}{d'^2 z^2} \Delta_{zz} E^3.
\end{equation}
\end{lemma}

\begin{proof} 
We have
\begin{align*}
\frac{3}{d'} \left (f_{xx} E\alpha_x+f_{xy}E\alpha_y+f_{xz} E\alpha_z\right)&=
\frac{3}{d'\,z} [f_{xx} (zE\alpha_x)+f_{xy} (zE\alpha_y)+f_{xz} (zE\alpha_z) ]\\
&=
\frac{3}{d'\,z} \left[ (xf_{xx}+yf_{xy}+zf_{xz})E\alpha_z+\frac{2}{d'}\left ( \Delta_{yz} E^2 H_{xy}-\Delta_{zz} E^2 H_{zx}\right )\right]\\
 &=
 \frac{3\,f_x}{z} E\alpha_z+(\frac{6}{d'^2 z}\Delta_{yz}-\frac{6\,y}{d'^2 z^2}\Delta_{zz}) E^2H_{xy}-\frac{6\,f_x}{d'^2 z^2} \Delta_{zz} E^3\\
&=\frac{3\,f_x}{z} E\alpha_z+\frac{6}{d'^2\,z^2} \left( z\Delta_{yz} -y\Delta_{zz}\right) E^2H_{xy}-\frac{6\,f_x}{d'^2 z^2} \Delta_{zz} E^3\,,
\end{align*}
where the second equality follows from \eqref{EHam-alphab-c}, the third one follows from \eqref{2-1deriv}, and the last one holds by \eqref{E2-EHca-EHab}.
\end{proof}

To return to the proof of \cref{c:eq1OBLC-M1(3)} we combine the equalities in the proof of \cref{c:OBLC-M1(3)} and \cref{Rel-Ealphas} to obtain
\begin{align*}
\frac{3}{d'} (
f_{xx} E\alpha_x &+f_{xy}E\alpha_y+f_{xz} E\alpha_z)
-z\zeta_y+y\zeta_z
\\
&=\left [-\frac{f_x^3}{y^2 z^2}+\frac{3f_x\Delta_{zz}}{d'^2 z^2}-\frac{3 f_x\Delta_{yy}}{d'^2 y^2}+\frac{1}{d'^2 (d-2)}\left( \frac{z H_{zx}(\Delta_{yy})}{y^2}-\frac{y H_{xy}(\Delta_{zz})}{z^2}\right) \right ] E^3
\\
&\qquad +\left(-\frac{3 f_x^2}{yz^2}+\frac{3 y \Delta_{zz}}{d'^2 z^2}-\frac{3\Delta_{yy}}{d'^2 y}\right) E^2 H_{xy}- {\frac{3 f_x}{z} E\alpha_z}
\\
&\qquad + {\frac{3\,f_x}{z} E\alpha_z}+\frac{6}{d'^2\,z^2} \left( z\Delta_{yz} -y\Delta_{zz}\right) E^2H_{xy}-\frac{6\,f_x}{d'^2 z^2} \Delta_{zz} E^3
\\
&=\left [\frac{f_x}{d'^2 y^2 z^2} \left(y^2 \Delta_{zz}+z^2\Delta_{yy}-2yz\Delta_{yz}+3 y^2 \Delta_{zz}-3z^2\Delta_{yy}-6y^2 \Delta_{zz}\right) \right .
\\
&\qquad \left . +\frac{1}{d'^2 (d-2) y^2 z^2}(z^3 H_{zx}(\Delta_{yy})-y^3 H_{xy}(\Delta_{zz})\right ]E^3
\\
&\qquad +\left [\frac{3}{d'^2 yz^2} \left(y^2\Delta_{zz}+z^2\Delta_{yy}-2yz\Delta_{yz}\right)+\frac{3y}{d'^2 z^2} \Delta_{zz}-\frac{3\Delta_{yy}}{d'^2 y}+\frac{6 \Delta_{yz}}{d'^2 z}-\frac{6y\Delta_{zz}}{d'^2 z^2} \right ]E^2 H_{xy}
\\
&=\frac{1}{d'^2(d-2) y^2 z^2} \left [-y^2 {f_x E (\Delta_{zz})}-z^2 {f_x E (\Delta_{yy})}-yz {f_x E(\Delta_{yz})} \right .
\\
&\qquad \left .+z^3 H_{zx}(\Delta_{yy})-y^3 H_{xy}(\Delta_{zz}) \right ] E^3+{0\,E^2 H_{xy}}
\\
&= \frac{1}{d'^2(d-2) y^2 z^2} \left [-y^2 \left(z H_{zx}(\Delta_{zz})-yH_{xy}(\Delta_{zz})\right)-z^2 \left(z H_{zx}(\Delta_{yy})-yH_{xy}(\Delta_{yy})\right) \right .
\\
&\qquad \left . -yz \left(zH_{zx}(\Delta_{yz}-yH_{xy}(\Delta_{yz}))\right)+z^3 H_{zx}(\Delta_yy)-y^3 H_{xy}(\Delta_{zz})\right ] E^3
\\
&=\frac{1}{d'^2(d-2) y^2 z^2} \times \left [y^2z \left(H_{xy}(\Delta_{yz})-H_{zx}(\Delta_{zz})\right)+yz^2 \left(H_{xy}(\Delta_{yy})-H_{zx}(\Delta_{yz})\right) \right ] E^3 
\\
&=
\frac{2}{d'^3 (d-2)}\delta_x E^3,
\end{align*}
where the second equality follows from \eqref{PD-prodaa}; to obtain the fifth equality we repeatedly apply the fundamental identity $f_x E-z H_{zx}+y H_{xy}=0$ to each of $\Delta_{zz}$, $\Delta_{yy}$ and $\Delta_{yz}$ and use the fact that $E (\Delta_{ij})=2(d-2) \Delta_{ij}$; and the last equality follows from \eqref{HamRel1} and \eqref{HamRel2}.
We have thus proven \eqref{c:eq1OBLC-M1(3)}.
\end{proof}

%%%%%%%%%%%%%%%%%%%
\subsection{Relation among $\zeta_x, \zeta_y,\zeta_z$ and the various $E^2 H_{ij}$} In this subsection we establish a relation that corresponds to the last column of $M_1(3)$. We use the following notation: 
 $d'\coloneqq (d-1), \;d''\coloneqq (d-1)^3 (d-2)=d'^3 (d-2)$. 

\begin{lemma}\label{lastrelnM1(3)} We have
\begin{equation}
\label{LCM1(3)}
f_x\zeta_x+f_y\zeta_y+f_z \zeta_z-\frac{2\,\delta_x}{d''} E^2H_{yz}-\frac{2\,\delta_y}{d''} E^2H_{zx}-\frac{2\,\delta_z}{d''} E^2H_{xy}=0\,.
\end{equation}
\end{lemma}

\begin{proof} 
Observe that
\begin{align*} 
f_x\zeta_x&= \frac{1}{x} (-yf_y-zf_z) \zeta_x
\\
&=\frac{f_y}{x} (-y\zeta_x)+\frac{f_z}{x} (-z\zeta_x)
\\
&=\frac{f_y}{x} \left[ \frac{2\delta_z}{d''} E^3-\frac{3}{d'}\left(f_{xz} E\alpha_x+f_{yz}E\alpha_y+f_{zz} E\alpha_z\right)-x\zeta_y\right]
\\
&\qquad+\frac{f_z}{x} \left[ -\frac{2\delta_y}{d''} E^3+\frac{3}{d'}\left(f_{xy} E\alpha_x+f_{yy}E\alpha_y+f_{yz} E\alpha_z\right)-x\zeta_z\right]
\\
&=-f_y\zeta_y-f_z\zeta_z+\frac{2\,(f_y\delta_z-f_z\delta_y)}{d''\,x} E^3+\frac{3}{d'\,x} [ (f_zf_{xy}-f_yf_{xz}) E\alpha_x
\\
&\qquad+(f_zf_{yy}-f_yf_{yz}) E\alpha_y+
(f_zf_{yz}-f_yf_{zz}) E\alpha_z]
\\
&=-f_y\zeta_y-f_z\zeta_z+\frac{2\,(f_y\delta_z-f_z\delta_y)}{d''\,x} E^3
\\
&\qquad+\frac{3}{d'^2\,x} \left[ (y\Delta_{xz}-z\Delta_{xy}) E\alpha_x+(z\Delta_{xx}-x\Delta_{xz})E\alpha_y+(x\Delta_{xy}-y\Delta_{xx})E\alpha_z)\right]
\\
&=-f_y\zeta_y-f_z\zeta_z+\frac{2}{d''\,x} \left(\delta_z {f_y E^3}-\delta_y {f_z E^3}\right)
\\
&\qquad+\frac{3}{d'^2\,x} \left[ \Delta_{xx} (z E\alpha_y-y E\alpha_x)+\Delta_{xy} (xE\alpha_z-zE\alpha_x)+\Delta_{xz} (yE\alpha_x-xE\alpha_y)\right]
\\
&=-f_y\zeta_y-f_z\zeta_z+\frac{2}{d''\,x} [ \delta_z \left(-zE^2H_{yz}+x E^2H_{xy}\right)-\delta_y\left( yE^2H_{yz}-xE^2H_{zx}\right) ]
\\
&\qquad +\frac{3}{d'^2\,x} \left[ \Delta_{xx} {(z E\alpha_y-y E\alpha_x)}+\Delta_{xy} {(xE\alpha_z-zE\alpha_x)}+\Delta_{xz} {(yE\alpha_x-xE\alpha_y)}\right]
\\
&=-f_y\zeta_y-f_z\zeta_z+\frac{2}{d''\,x} [(-3(d-2)\delta+x\delta_x) E^2 H_{yz}+x\delta_y E^2H_{zx}+x\delta_z E^2H_{xy}]
\\
&\qquad +\frac{6}{d'^3\,x} [\Delta_{xx} (f_{xx}E^2H_{yz}+f_{xy}E^2H_{zx}+f_{xz}E^2H_{xy})+ \Delta_{xy} (f_{xy}E^2H_{yz}+f_{yy}E^2H_{zx}+f_{yz}E^2H_{xy})
\\
&\qquad + \Delta_{xz} (f_{xz}E^2H_{yz}+f_{yz}E^2H_{zx}+f_{zz}E^2H_{xy})]
\\
&= -f_y\zeta_y-f_z\zeta_z-\frac{6\delta}{d'^3\,x}E^2H_{yz} +\frac{2}{d''} (\delta_x E^2H_{yz}+\delta_y E^2H_{zx}+\delta_z E^2H_{xy})
\\
&\qquad+ \frac{6}{d'^3\,x}[ (f_{xx}\Delta_{xx}+f_{xy}\Delta_{xy}
+f_{xz}\Delta_{xz}) E^2H_{yz}
\\
&\qquad + (f_{xy}\Delta_{xx}+f_{yy}\Delta_{xy}+f_{yz}\Delta_{xz})E^2H_{zx}+(f_{xz}\Delta_{xx}+f_{yz}\Delta_{xy}+f_{zz}\Delta_{xz})E^2H_{xy} ]
\\
&=
%\eqref{det-exp1},\eqref{det-exp6},\eqref{det-exp8}} 
-f_y\zeta_y-f_z\zeta_z+\frac{2}{d''} (\delta_x E^2H_{yz}+\delta_y E^2H_{zx}+\delta_z E^2H_{xy})- {\frac{6\delta}{d'^3\,x}E^2H_{yz}}+ {\frac{6\delta}{d'^3\,x}E^2H_{yz}}\,,
\end{align*}
where the first equality follows from \cref{euler}, the third from \cref{c:eq3OBLC-M1(3)} and \cref{c:eq2OBLC-M1(3)}, the fifth from \cref{eq:MinDiff}, the seventh from \cref{E2-EHca-EHab}, the eighth from \cref{EHam-alphab-c}, and the ninth from \cref{det-exp}. 
Hence we have shown 
\[
f_x\zeta_x+f_y\zeta_y+f_z\zeta_z-\frac{2}{d''} \delta_x E^2H_{yz}-\frac{2}{d''} \delta_y E^2H_{zx}-\frac{2}{d''} \delta_z E^2H_{xy}=0\,.\qedhere
\]
\end{proof}

%%%%%%%%%%%%%%%%%%%%%%%%%%%%%%%%%%%%%%%%%%%%%%%%%%%%%%%%%%%%%%
\section{Matrix identities}\label{a:matrixidentities} 
%%%%%%%%%%%%%%%%%%%%%%%%%%%%%%%%%%%%%%%%%%%%%%%%%%%%%%%%%%%%%%

 In this appendix we collect matrix identities involving the matrices, with entries in $R$, from \cref{sec:glossary}.

\begin{equation}
\label{c:q-d-D-sigma-interp}
\del_1=\del_3^T\quad\quad \partial^T _2=-\partial_2\quad\quad D_1=D_3^T \quad\quad D^T _2=-D_2\quad\quad \sigma_1=\sigma_3^T \quad \quad\sigma^T _2=-\sigma_2 
\end{equation}
\begin{equation}
\label{c:basic-Eul}
\partial_1 D_3=0_{3\times 3}\quad \quad \partial_3 D_1=D_2\partial_2 \quad\quad\del_i\alpha_i=\alpha_{i-1}D_i
\end{equation}
\begin{equation}
\label{c:identity-Del-adjDel}
\alpha_1\alpha_2=\alpha_2\alpha_1=\frac{\delta}{(d-1)^3} I_{3\times 3}
\end{equation}
\begin{equation}
\label{c:identity-RS-4-2-1}
\frac{1}{3}\sigma_3 \partial_1
=\alpha_1\alpha_2+\frac{1}{3}\partial_2\sigma_2 \quad\quad q\sigma_3=3 \alpha_2 D_3 
\end{equation}
\begin{equation}\label{c:d-D-alpha-q}
q=\partial_3 \partial_1 \quad\quad \alpha_1 q=\partial_2 D_2 \quad\quad\alpha_1 q=
D_3\partial_1 \quad\quad q\,\partial_2=0
\end{equation}

\begin{equation}
\label{c:d-sigmas-1}
\sigma_3\partial_1-\partial_2 \sigma_2=\frac{3\delta}{(d-1)^3} I_{3\times 3} 
\end{equation}
\begin{equation}
\label{c:D-to-sigmas} 
D_2\sigma_3+\sigma_2 D_3=0 \quad\quad D_1\sigma_2=-\sigma_1D_2
\end{equation}

%%%%%%%%%%%%%%%%%%%%%%%%%%%%%%%%%%%%%%%%%%%%%%%%%%%%%%%%%%%%%%
\section{Proof of \cref{l:liftd3}}\label{a:theta3chain}
%%%%%%%%%%%%%%%%%%%%%%%%%%%%%%%%%%%%%%%%%%%%%%%%%%%%%%%%%%%%%%

%%%%%%%%%%%%%%%%%%%
\subsection{First square commutes} \label{first lifting for 3rd ord}
 
Write
{\footnotesize
\setlength{\arraycolsep}{3pt}
\medmuskip=1mu
\[
\theta_0(3)=
\begin{bmatrix*}[r]
\frac{1}{6} f_{xxx}&\frac{1}{2} f_{xxy}&\frac{1}{2} f_{xxz}&\frac{1}{2} f_{xyy}&f_{xyz}&\frac{1}{2} f_{xzz}&\frac{1}{6} f_{yyy}&\frac{1}{2} f_{yyz}&\frac{1}{2} f_{yzz}&\frac{1}{6} f_{zzz}\\
\frac{1}{2} f_{xx}& f_{xy}\;&f_{xz}\;&\frac{1}{2} f_{yy}&f_{yz}\;&\frac{1}{2} f_{zz}&0\;\;&0\;\;&0\;\;&0\;\\
0\;\;&\frac{1}{2}f_{xx}&0\;\;&f_{xy}\;&f_{xz}\;&0\;\;& \frac{1}{2} f_{yy}\;&f_{yz}\;&\frac{1}{2}f_{zz}&0\;\\ 
0\;\;&0\;\;&\frac{1}{2}f_{xx}&0\;\;&f_{xy}\;&f_{xz}\;&0\;\;&\frac{1}{2}f_{yy}&f_{yz}&\frac{1}{2}f_{zz}\\
\end{bmatrix*}=\begin{bmatrix*}<-r_1->\\<-r_2->\\<-r_3->\\<-r_4->\\ \end{bmatrix*}
\]}
{\normalsize where $r_i$ denote the rows of the matrix $\theta_0(3).$}

\begin{proposition}\label{first lift for 3rd ord} 
The following equality is satisfied: $-\theta_0(3)M_0(3)=\del_1^C\theta_1(3)$.
\end{proposition}

The proof is given at the end of the subsection after a series of preparatory lemmas. 
We make use of the description of $M_0(3)$ as 
\[
M_0(3)=
\begin{bmatrix*} E^3 & E^2 H_{yz}&E^2 H_{zx}&E^2 H_{xy}&E\alpha_x&E\alpha_y&E\alpha_z&\zeta_x&\zeta_y&\zeta_z\\
\end{bmatrix*}\phantom{sjdhfjsdfsdfsfa}
\] discussed in \cref{M0(3)}.
\begin{lemma}
\label{E2Hvanish}
The products $\theta_0(3) E^3$, $\theta_0(3) E^2 H_{yz}$, $\theta_0(3) E^2 H_{zx}$, and $\theta_0(3) E^2 H_{xy}$ all vanish.
\end{lemma}

\begin{proof}
Using \eqref{3rd-Euler}, one shows that $r_1 E^3=0$. Also, 
\begin{align*}
r_2 E^3&=3x(x^2 f_{xx}+y^2 f_{yy}+z^2 f_{zz}+2xy f_{xy}+2xzf_{xz}+2yz f_{yz})\\
r_3 E^3&=3y(x^2 f_{xx}+y^2 f_{yy}+z^2 f_{zz}+2xy f_{xy}+2xzf_{xz}+2yz f_{yz})\\
r_4 E^3&=3z(x^2 f_{xx}+y^2 f_{yy}+z^2 f_{zz}+2xy f_{xy}+2xzf_{xz}+2yz f_{yz})\,,
\end{align*}
and note the right hand side of the equations above are zero in $R$ by \eqref{2nd order Euler id}, and hence $\theta_0(3) E^3=0_{4\times 1}$. 
Next we show that $\theta_0(3) E^2 H_{yz}=0_{4\times 1}$. Similarly one can show that $\theta_0(3) E^2 H_{zx}=0_{4\times 1},\;\theta_0(3) E^2 H_{xy}=0_{4\times 1}$.
We have
\begin{align*}
r_1 E^2 H_{yz}&=f_z \left(x^2 f_{xxy}+2xyf_{xyy}+2xz f_{xyz}+y^2 f_{yyy}+2yz f_{yyz}+z^2 f_{yzz}\right)\\
&\qquad -f_y\left(x^2 f_{xxz}+2xyf_{xyz}+2xz f_{xzz}+y^2 f_{yyz}+2yz f_{yzz}+z^2 f_{zzz}\right)\\ &=
(d-2)(d-1) (f_z f_y-f_y f_z)\\
&=0\,,
\end{align*}
where the second equality uses \eqref{eq2-3}. The calculations that $r_i E^2 H_{yz}={0}$, for $i=2,3,4$, follow similarly using \eqref{2-1deriv} and \eqref{2nd order Euler id}.
Hence $\theta_0(3) E^2 H_{yz}=0_{4\times 1}$, as claimed. 
\end{proof}

\begin{lemma}\label{theta0(3) Ealphas}
The following equalities are satisfied
\[
\theta_0(3) E\alpha_x=\frac{-x}{(d-1)^2} \begin{bmatrix} (d-2)\delta\\ x\delta\\ y\delta\\ z\delta\\ \end{bmatrix},\; \theta_0(3) E\alpha_y=\frac{-y}{(d-1)^2}\begin{bmatrix} (d-2)\delta\\x\delta\\y\delta\\z\delta \end{bmatrix}, \phantom{s}\theta_0(3) E\alpha_z=\frac{-z}{(d-1)^2}\begin{bmatrix} (d-2)\delta\\x\delta\\y\delta\\z\delta \end{bmatrix}\,.
\]
\end{lemma}

\begin{proof} 
We verify that the first equality is satisfied in the first and third entry; the rest are similar and left to the reader.

Observe that 
\begin{align*}
\frac{(d-1)^2}{2} r_1 E\alpha_x
&=
\Delta_{xx}\left((d-2)(d-1)f_x-\frac{x}{2} (xf_{xxx}+yf_{xxy}+zf_{xxz})\right)\\
&\qquad+\Delta_{xy}(xyf_{xyy}+xz f_{xyz}+y^2 f_{yyy}+2yz f_{yyz}+z^2 f_{yzz})\\
&\qquad+\Delta_{xz}(xy f_{xyz}+xz f_{xzz}+y^2 f_{yyz}+2yz f_{yzz}+z^2 f_{zzz})+\Delta_{yy}(-\frac{x^2}{2} f_{xyy}-\frac{xy}{2} f_{yyy}-\frac{xz}{2} f_{yyz})\\
&\qquad+\Delta_{yz}(-x^2 f_{xyz}-xy f_{yyz}-xz f_{yzz})+\Delta_{zz}\left(-\frac{x^2}{2} f_{xzz}-\frac{xy}{2} f_{yzz}-\frac{xz}{2} f_{zzz}\right)\\ &=
(d-2)(d-1) f_x \Delta_{xx}-\frac{(d-2) x f_{xx}\Delta_{xx}}{2}+\Delta_{xy} \left((d-2)(d-1)f_y-x (x f_{xxy}+y f_{xyy}+z f_{xyz})\right)\\
&\qquad+\Delta_{xz} ((d-2)(d-1) f_z-xy f_{xyz}-xz f_{xzz})-\frac{x}{2} \Delta_{yy} (x f_{xyy}+yf_{yyy}+zf_{yyz})\\
&\qquad-x\Delta_{yz}(xf_{xyz}+y f_{yyz}+z f_{yzz})-\frac{x}{2} \Delta_{zz}(xf_{xzz}+yf_{yzz}+zf_{zzz}) \\
&=
(d-2)(d-1) (f_x\Delta_{xx}+f_y\Delta_{xy}+f_z\Delta_{xz})-\frac{(d-2)x}{2}(f_{xx}\Delta_{xx}+f_{xy}\Delta_{xy}+f_{xz}\Delta_{xz})\\
&\qquad-\frac{(d-2)x}{2}(f_{xy}\Delta_{xy}+f_{yy}\Delta_{yy}+f_{yz}\Delta_{yz})-\frac{(d-2)x}{2}(f_{xz}\Delta_{xz}+f_{yz}\Delta_{yz}+f_{zz}\Delta_{zz})\\
&=-\frac{d-2}{2} x\delta\,.
\end{align*}
The first equality uses \cref{eq2-3}, the second follows from \eqref{eq3-2aa} and \eqref{eq2-3}, the third equality is from \eqref{eq3-2aa} and \eqref{eq3-2ab}, and the fourth equality is from \cref{det-exp} and \eqref{eqCram}.
Hence $r_1 E\alpha_x=-d \frac{(d-2)}{(d-1)^2} x \delta$, as needed.

Next consider
\begin{align*}
\frac{(d-1)^2}{2} r_3 E\alpha_x
&=2y \Delta_{xx} (xf_{xx}+yf_{xy}+zf_{xz})-\frac{xy}{2} f_{xx}\Delta_{xx}+\Delta_{xy}\left(2xyf_{xy}+xz f_{xz}+3y^2 f_{yy}+4yzf_{yz}+z^2 f_{zz}\right)\\
&\qquad+2y\Delta_{xz}( xf_{xz}+yf_{yz}+zf_{zz})
-xyf_{xz}\Delta_{xz}-x\Delta_{yy}(xf_{xy}+yf_{yy}+zf_{yz})\\
& \qquad -\frac{xy}{2}f_{yy}\Delta_{yy}-x\Delta_{yz} (xf_{xz}+yf_{xz}+zf_{zz})-\frac{xy}{2} f_{zz}\Delta_{zz}\\
&= 
2 (d-1)y f_x \Delta_{xx}-\frac{xy}{2} f_{xx}\Delta_{xx}+\Delta_{xy} (-x^2 f_{xx}-xzf_{xz}+2y^2 f_{yy}+2yz f_{yz})+2(d-1) yf_z\Delta_{xz}\\
&\qquad-xyf_{xz}\Delta_{xz}-(d-1)x f_y\Delta_{yy}-\frac{xy}{2}\Delta_{yy}-(d-1)xf_z\Delta_{yz}-\frac{xy}{2}f_{zz}\Delta_{zz}\\
&=2 (d-1)y f_x \Delta_{xx}-\frac{xy}{2} f_{xx}\Delta_{xx}+2(d-1)y f_y\Delta_{xy}-(d-1)x f_x\Delta_{xy}-xyf_{xy}\Delta_{xy}\\
&\qquad +2(d-1) yf_z\Delta_{xz}-xyf_{xz}\Delta_{xz}-(d-1)x f_y\Delta_{yy}-\frac{xy}{2}\Delta_{yy}-(d-1)xf_z\Delta_{yz}-\frac{xy}{2}f_{zz}\Delta_{zz}\\
&=2(d-1)y (f_x \Delta_{xx}+f_y\Delta_{xy}+f_z\Delta_{xz})-(d-1) x (f_x\Delta_{xy}+f_{y}\Delta_{yy}+f_z\Delta_{yz})\\
&\qquad-\frac{xy}{2}(f_{xx}\Delta_{xx}+f_{xy}\Delta_{xy}+f_{xz}\Delta_{xz})-\frac{xy}{2}(f_{xy}\Delta_{xy}+f_{yy}\Delta_{yy}+f_{yz}\Delta_{yz})\\
&\qquad-\frac{xy}{2} (f_{xz}\Delta_{xz}+f_{yz}\Delta_{yz}+f_{zz}\Delta_{zz})\\
&=
2yx\delta-xy\delta-\frac{3xy}{2}\delta\,,
\end{align*}
where the second equality uses \eqref{2-1deriv} and \eqref{2nd order Euler id}, and the last equality uses \cref{det-exp} and \eqref{eqCram}.
Hence $r_3 E\alpha_x=-\frac{x\delta}{(d-1)^2} y$, as claimed.
\end{proof} 

\begin{lemma}\label{theta0(3)-all-zetas} 
There is an equality 
\[
\theta_0(3)\begin{bmatrix} 
\zeta_z & \zeta_y & \zeta_z
\end{bmatrix}
=\frac{1}{(d-1)^2}\begin{bmatrix}
\delta_z f_y-\delta_y f_z & \delta_x f_z-\delta_z f_x& \delta_y f_x-\delta_x f_y\\
0 & 3\delta f_z & -3\delta f_y\\
-3\delta f_z& 0 &3\delta f_x\\
3\delta f_y& -3\delta f_x & 0
\end{bmatrix}\,.
\]
\end{lemma}

\begin{proof}
We compute first $\theta_0(3)\zeta_x$. Recall that 
\[
\zeta_x=\frac{1}{x^2}\left[ H^3_{yz}+ \frac{3}{(d-1)^2}\Delta_{xx}E^2 H_{yz}-\frac{1}{(d-2)(d-1)^2} H_{yz}(\Delta_{xx}) E^3\right]\,,
\]
and since by \cref{E2Hvanish} \[\theta_0(3) E^3=0_{4\times 1}=\theta_0(3) E^2 H_{yz},\]
it suffices to compute $\theta_0(3) H^3_{yz}$. 
We claim that 
\begin{equation}
\label{theta0-zetax} 
\theta_0(3) H^3_{yz}=\frac{x^2}{(d-1)^2}\begin{bmatrix} \delta_z f_y-\delta_y f_z\\ 0\\ -3\delta f_z\\ 3\delta f_y\\ \end{bmatrix}\,.
\end{equation}
Indeed, as the first six entries of $H^3_{yz}$ are zero we see that $r_2 H^3_{yz}=0$. Next, observe
\begin{align*}
r_3 H^3_{yz}&=3f_{yy}f^3_z-6f_{yz}f^2_zf_y+3f_{zz} f_z f^2_y\\
&=3f^2_z (f_{yy}f_z-f_{yz}f_y)-3f_z f_y (f_{yz}f_z-f_{zz}f_y)\\
&=
\frac{3\,f_z}{(d-1)} \left[f_z (z\Delta_{xx}-x\Delta_{xz})-f_y(x\Delta_{xy}-y\Delta_{xx})\right]\\
&=\frac{3\,f_z}{(d-1)} \left[ \Delta_{xx}(zf_z+yf_y)-x(f_y\Delta_{xy}+f_z\Delta_{xz})\right]\\
&=\frac{3\,f_z}{(d-1)}\left[-x(f_x\Delta_{xx}+f_y\Delta_{xy}+f_z\Delta_{xz})\right]\\
&= x^2 \frac{-3\delta\,f_z}{(d-1)^2}\,,
\end{align*}
where the third equality uses \cref{eq:MinDiff} and the last equality is from \cref{eqCram}.
Similarly one can verify that $r_4 H^3_{yz}=x^2\frac{3\delta\,f_y}{(d-1)^2}$. 

To compute $r_1 H^3_{yz}$ notice that
\[
H_{yz}(f_{yy})=f_z f_{yyy}-f_yf_{yyz},\phantom{part}H_{yz}(f_{yz})=f_z f_{yyz}-f_y f_{yzz},\phantom{par}H_{yz}(f_{zz})=f_z f_{yzz}-f_y f_{zzz}\,.
\]
As a consequence, with $H=H_{yz}$, the following equalities hold
\label{prel-r1H3}
\begin{align*} 
r_1H^3&=f_z^3 f_{yyy}-3f^2_z f_y f_{yyz}+3f_z f^2_y f_{yzz}-f^3_y f_{zzz}\\
&=f^2_z (f_z f_{yyy}-f_y f_{yyz})-2f_zf_y(f_z f_{yyz}-f_y f_{yzz})+f^2_y(f_z f_{yzz}-f_y f_{zzz})\\
&=f^2_z H(f_{yy})-2f_z f_y H (f_{yz})+f^2_y H(f_{zz})\\
&= f_z \left[ f_z H(f_{yy})-f_y H(f_{yz})\right]-f_y \left[ f_z H(f_{yz})-f_y H(f_{zz})\right]\,.
\end{align*}

From \eqref{eq:MinDiff} we have: 
\[
H(f_y)=f_z f_{yy}-f_y f_{yz}=\frac{1}{(d-1)}( z\Delta_{xx}-x\Delta_{xz})\quad \text{and}\quad H(f_z)=f_z f_{yz}-f_y f_{zz}=\frac{1}{d-1} (x\Delta_{xy}-y\Delta_{xx})\,.
\]
Applying $H$ to the first and second equation, respectively, yield 
\begin{align}
\label{prelH-almHyy}
f_z H(f_{yy})-f_y H(f_{yz})&=\frac{1}{d-1}\left[ -f_y \Delta_{xx}+zH(\Delta_{xx})-xH(\Delta_{xz})\right]-f_{yy}H(f_z)+f_{yz}H(f_y)\\
\label{prelH-almHyz}
f_z H(f_{yz})-f_y H(f_{zz})&=\frac{1}{d-1}\left[ x H(\Delta_{xy})-f_z \Delta_{xx}-yH(\Delta_{xx})\right]-f_{yz}H(f_z)+f_{zz}H(f_y)\,.
\end{align}
Multiplying \eqref{prelH-almHyy} with $f_z$, \eqref{prelH-almHyz} with $f_y$ and subtracting establishes the following 

\begin{align*}
r_1H^3&=f_z \left[ f_z H(f_{yy})-f_y H(f_{yz})\right]-f_y \left[ f_z H(f_{yz})-f_y H(f_{zz})\right]\\
&=\frac{1}{d-1} \left[ H(\Delta_{xx}) (zf_z+yf_y)-x (f_z H(\Delta_{xz})+f_yH(\Delta_{xy})\right]\\
&\qquad+H(f_z) (f_y f_{yz}-f_{yy}f_z) +H(f_y)(f_z f_{yz}-f_y f_{zz})\\
&=\frac{-x}{d-1} (f_x H(\Delta_{xx})+f_y H(\Delta_{xy})+f_z H(\Delta_{xz}))\\
&=
x^2\;\;\frac{\delta_z f_y-\delta_y f_z}{(d-1)^2}\,,
\end{align*}
where the first equality was established above and the last equality is from \eqref{H-on-Cram1}.

Similarly one can check that
\[
\theta_0(3) H^3_{zx}=\frac{y^2}{(d-1)^2}\begin{bmatrix} \delta_x f_z-\delta_z f_x\\ 3\delta f_z\\ 0\\ -3\delta f_x\\ \end{bmatrix}\quad \text{and }\quad
\theta_0(3) H^3_{xy}=\frac{z^2}{(d-1)^2}\begin{bmatrix} \delta_y f_x-\delta_x f_y\\ -3\delta f_y\\ 3\delta f_x\\0 \end{bmatrix}
\]
from which 
\begin{align*}
\theta_0(3)\zeta_y=\frac{1}{(d-1)^2}\begin{bmatrix} \delta_x f_z-\delta_z f_x\\ 3\delta f_z\\ 0\\ -3\delta f_x \end{bmatrix}\quad\text{and }\quad
\theta_0(3)\zeta_z=\frac{1}{(d-1)^2}\begin{bmatrix} \delta_y f_x-\delta_x f_y\\ -3\delta f_y\\ 3\delta f_x\\0 \end{bmatrix}\,
\end{align*}
follow, respectively.
\end{proof} 

\begin{chunk}
Next consider the vectors
\[ 
V_x=\left[ \begin{array} {c} 
0\\
\dfrac{\delta_z}{(d-1)^2} \\
-\dfrac{\delta_y}{(d-1)^2} \\ \hline
A_x
\end{array}\right]\quad\text{
where}\quad
A_x\coloneqq 
\frac{3}{(d-1)(d-2)} 
\begin{bmatrix} H_{yz}(\Delta_{xx})\\ 
H_{zx}(\Delta_{xx})\\
H_{xy}(\Delta_{xx})\\
H_{zx}(\Delta_{xy})+\dfrac{y\delta_z}{d-1}\\
\dfrac{1}{2}\left( H_{zx}(\Delta_{xz})+H_{xy}(\Delta_{xy})+\dfrac{z\delta_z-y\delta_y}{d-1}\right)\\
H_{xy}(\Delta_{xz})-\dfrac{z\delta_y}{d-1}\\
\end{bmatrix}\,,
\]
\[ 
V_y=\left[ \begin{array} {c} 
-\dfrac{\delta_z}{(d-1)^2} \\
0\\
\dfrac{\delta_x}{(d-1)^2}\\ \hline
A_y
\end{array}\right]\quad\text{where}\quad
A_y\coloneqq 
\frac{3}{(d-1)(d-2)} 
\begin{bmatrix} H_{yz}(\Delta_{xy})-\dfrac{x\delta_z}{d-1}\\ 
H_{yz}(\Delta_{yy})\\
\dfrac{1}{2}\left( H_{yz}(\Delta_{yz})+H_{xy}(\Delta_{xy})+\dfrac{x\delta_x-z\delta_z}{d-1}\right)\\
H_{zx}(\Delta_{yy})\\
H_{xy}(\Delta_{yy})\\
H_{xy}(\Delta_{yz})+\dfrac{z\delta_x}{d-1}
\end{bmatrix}\,,
\]
\[ 
V_z=\left[ \begin{array} {c} 
\dfrac{ \delta_y}{(d-1)^2}\\
-\dfrac{\delta_x}{(d-1)^2} \\ 
0\\ \hline
A_z
\end{array}\right]\quad\text{where}\quad
A_z\coloneqq 
\dfrac{3}{(d-1)(d-2)} 
\begin{bmatrix} H_{yz}(\Delta_{xz})+\dfrac{x\delta_y}{d-1}\\ 
\dfrac{1}{2}\left( H_{yz}(\Delta_{yz})+H_{zx}(\Delta_{xz})+\frac{y\delta_y-x\delta_x}{d-1}\right)\\
H_{yz}(\Delta_{zz})\\
H_{zx}(\Delta_{yz})-\dfrac{y\delta_x}{d-1}\\
H_{zx}(\Delta_{zz})\\
H_{xy}(\Delta_{zz})
\end{bmatrix}
\,.\]
\end{chunk}

\begin{lemma}
\label{d1CV}
The following equality is satisfied 
\[
\del_1^C\begin{bmatrix} V_x & V_y & V_z\end{bmatrix}=
 \frac{1}{(d-1)^2} \begin{bmatrix} \delta_z f_y-\delta_y f_z & \delta_x f_z-\delta_z f_x & \delta_y f_x-\delta_x f_y \\ 
0 & 3\, \delta f_z & -3\,\delta f_y \\
-3\, \delta f_z & 0& 3\, \delta f_x\\
 3\, \delta f_y & -3\, \delta f_x&0
\end{bmatrix}\,.
\]
\end{lemma}
\begin{proof}
Using \eqref{HamRel1}, \eqref{HamRel2}, \eqref{H-on-detexp1}, \eqref{Hzx-on-detexp6}, \eqref{Hxy-on-detexp8}, and \eqref{eq:Hij-on-2der} it follows that $A_x\in \ker\theta_0(2)$. Now applying \eqref{eq:FDR-Ham}, \eqref{H-on-Cram2}, and \eqref{HamRel4-5Comb}, one can verify that $\del_1^CV_x$ is the first column of the matrix on the right hand side. Similar computations show the other columns agree. 
\end{proof}

\begin{proof}[Proof of \cref{first lift for 3rd ord}]
Using the identities \cref{det-exp} and
\eqref{eqCram}, and \cref{d1CV} it is straightforward to compute the entries of $\del_1^C \theta_1(3)$. Comparing these with the entries computed in \cref{E2Hvanish,theta0(3) Ealphas,theta0(3)-all-zetas} one concludes
\[
\del_1^C \theta_1(3)=-\theta_0(3) M_0(3)\,.\qedhere
\]
\end{proof}

%%%%%%%%%%%%%%%%%%%
\subsection{The second square commutes} \label{2nd lift for 3rd ord}

Next, we show the following. 

\begin{proposition} 
The equality $-\theta_1(3)M_1(3)=\del_2^C\theta_2(3)$ holds.
\end{proposition} 

\begin{proof}
Let $d'=d-1$. Using the identities in \cref{appendix,a:hamiltonian}, one can show that $-\theta_1(3) M_1(3)$ is of the form: 
\[
\left[\begin{array}{@{}c c c | c @{} }
 0_{3\times3}&\phantom{pa} \dfrac{d+1}{2}K&\phantom{par}\Lambda \phantom{par}& \\
 {}&{}&{}&W\\
0_{6\times3}&\phantom{pa}0_{6\times3}&\Pi&{}
\end{array}\right]
\]
where 
\begin{align*}
K&=\begin{bmatrix} f_z f_{xy}-f_y f_{xz}& f_z f_{yy}-f_y f_{yz}& f_z f_{yz}-f_y f_{zz}\\
f_x f_{xz}-f_z f_{xx}& f_x f_{yz}-f_z f_{xy}& f_x f_{zz}-f_z f_{xz}\\
f_y f_{xx}-f_x f_{xy}& f_y f_{xy}-f_x f_{yy}& f_y f_{xz}-f_x f_{yz}
\end{bmatrix}
\end{align*}

\begin{align*}
&\phantom{askjfsf}\Lambda=\dfrac{9\,\delta}{2d'} I_{3\times 3}-\dfrac{1}{d'^2}\begin{bmatrix} x\delta_x& x\delta_y & x\delta_z\\
y\delta_x& y\delta_y& y\delta_z\\
z\delta_x&z\delta_y&z\delta_z 
\end{bmatrix} \phantom{sajshfajshfajshfafhj}
\end{align*}

\begin{align*}
W=\dfrac{1}{d'^2} \begin{bmatrix} 
H_{yz}(\delta)\\ 
 H_{zx}(\delta)\\ 
 H_{xy}(\delta)\\
\dfrac{6x}{d-2} \,H_{yz} (\delta)\\
\dfrac{3}{d-2} \left(x\,H_{zx}(\delta)+y H(\delta)\right)\\
\dfrac{3}{d-2}\left(x\,H _{xy}(\delta)+z H(\delta)\right)\\
\dfrac{6y}{d-2}\,H_{zx}(\delta)\\
\dfrac{3}{d-2}\left(y\,H _{xy}(\delta)+z H_{zx}(\delta)\right)\\
 \dfrac{6z}{d-2}\,H _{xy}(\delta)
\end{bmatrix}
\end{align*}
and where $\Pi$ is the matrix described as 
\begin{align*}
\label{def-Pi-Mat}
{}\\
\Pi=-\dfrac{9}{d'} \begin{bmatrix} \Delta_{xx} &\\
\Delta_{xy} \\
 \Delta_{xz} \\
 \Delta_{yy} \\
\Delta_{yz} \\
 \Delta_{zz} \\
\end{bmatrix} \begin{bmatrix} f_x&f_y&f_z \end{bmatrix} +\dfrac{9}{d'^2} \begin{bmatrix} 2x&0&0\\
y&x&0\\
z&0&x\\
0&2y&0\\
0&z&y\\
0&0&2z\\
\end{bmatrix}-\dfrac{3}{d'^2 (d-2)}\begin{bmatrix} 2x^2\\
2xy\\
2xz\\
2y^2\\
2yz\\
2z^2\\
\end{bmatrix} \begin{bmatrix} \delta_x&\delta_y&\delta_z\\ \end{bmatrix}.
\end{align*}
Note that the first three columns of $-\theta_1(3) M_1(3)$ are zero. As the first three columns of $\theta_2(3)$ are also zero, one has
\[
\mathrm{Col}_i(-\theta_1(3) M_1(3))=\mathrm{Col}_i(\del^C_2\theta_2(3)) \quad\text{for}\quad i=1,2,3\,.
\]
Next the fourth column of $-\theta_1(3)M_1(3)$ is of the form 
\begin{align*}
\begingroup
\renewcommand{\arraystretch}{1.5}
 \begin{bmatrix} 
&\dfrac{d+1}{2} (f_z f_{xy}-f_y f_{xz})\\ 
&\dfrac{d+1}{2}(f_x f_{xz}-f_z f_{xx}) \\
&\dfrac{d+1}{2}(f_y f_{xx}-f_x f_{xy})\\ \hline
&0_{6\times 1}
\end{bmatrix}
\endgroup
\end{align*}
and we have the following equality

\begin{align*}
\frac{d+1}{2} \begin{bmatrix} f_z f_{xy}-f_y f_{xz}\\ 
f_x f_{xz}-f_z f_{xx} \\
f_y f_{xx}-f_x f_{xy}
\end{bmatrix}
&=-\frac{d+1}{2} \left( f_{xx}\begin{bmatrix} 0\\f_z\\-f_y\\ \end{bmatrix}+f_{xy} \begin{bmatrix} -f_z\\0\\ f_x\\ \end{bmatrix}+f_{xz} \begin{bmatrix}
f_y\\ -f_x\\ 0\\ \end{bmatrix}\right)\,.
\end{align*}

Therefore, 
\[
\begingroup
\renewcommand{\arraystretch}{1.5}
 \begin{bmatrix} 
&\dfrac{d+1}{2} (f_z f_{xy}-f_y f_{xz})\\ 
&\dfrac{d+1}{2}(f_x f_{xz}-f_z f_{xx}) \\
&\dfrac{d+1}{2}(f_y f_{xx}-f_x f_{xy})\\ \hline
&0_{6\times 1}
\end{bmatrix}=\mathrm{Col}_4(\del_2^C\theta_2(3))\,.
\endgroup
\]
Similar calculations show that  the fifth and sixth columns of $-\theta_1(3)M_1(3)$ are the same as those of $\del_2^C\theta_2(3)$, respectively. The last column of $-\theta_1(3)M_1(3)$ is given by 
\[
W=\begin{bmatrix} \dfrac{1}{d'^2} (f_z\delta_y-f_y \delta_z)\\ 
\dfrac{1}{d'^2} (f_x\delta_z-f_z\delta_x)\\ 
\dfrac{1}{d'^2} (f_y\delta_x-f_x \delta_y)\\
\dfrac{3}{d'^2(d-2)} \phantom{p}2x\, (f_z\delta_y-f_y\delta_z)\\
\dfrac{3}{d'^2(d-2)} \phantom{p}\left(x\,(f_x\delta_z-f_z\delta_x)+y (f_z\delta_y-f_y\delta_z)\right)\\
\dfrac{3}{d'^2(d-2)} \phantom{p}\left(x\,(f_y\delta_x-f_x\delta_y)+z (f_z\delta_y-f_y\delta_z)\right)\\
\dfrac{3}{d'^2(d-2)} \phantom{p} 2y\,(f_x\delta_z-f_z\delta_x)\\
\dfrac{3}{d'^2(d-2)} \phantom{p}\left(y\,(f_y\delta_x-f_x\delta_y)+z (f_x\delta_z-f_z\delta_x)\right)\\
\dfrac{3}{d'^2(d-2)} \phantom{p} 2z\,(f_y\delta_x-f_x\delta_y)
\end{bmatrix}
=\dfrac{\delta_x}{d'^2} \begin{bmatrix} 0\\-f_z\\f_y\\ \hline 
0_{6\times1}\\ \end{bmatrix}+\dfrac{\delta_y}{d'^2} \begin{bmatrix} f_z\\0\\-f_x\\ \hline 
0_{6\times1}\\ \end{bmatrix}+\dfrac{\delta_z}{d'^2} \begin{bmatrix} -f_y\\f_x\\0\\ \hline 
0_{6\times1}\\ \end{bmatrix}+\\
\]

\[
+\dfrac{3\delta_x}{d'^2 (d-2)} \begin{bmatrix} 0_{3\times 1} \\ \hline
0\\ -xf_z\\ xf_y\\ -2yf_z\\ (yf_y-zf_z)\\ 2zf_y\\
\end{bmatrix}+\dfrac{3\delta_y}{d'^2 (d-2)} \begin{bmatrix} 0_{3\times 1} \\ \hline
2xf_z\\ yf_z\\ zf_z-xf_x\\ 0\\ -yf_x\\ -2zf_x\\ \end{bmatrix}+\dfrac{3\delta_z}{d'^2 (d-2)} \begin{bmatrix} 0_{3\times 1} \\ \hline
-2xf_y\\ xf_x-yf_y\\ -zf_y\\ 2yf_x\\ zf_x\\ 0\\
\end{bmatrix}.
\]

Hence 
\[W=\frac{-1}{d'^2}\left(\delta_x \text{Col}_2(\del_2^C )+\delta_y\text{Col}_3(\del_2^C )+\delta_z \text{Col}_4(\del_2^C )+\frac{3}{d-2}[\delta_x \text{Col}_6(\del_2^C )+\delta_y \text{Col}_7(\del_2^C )+\delta_z\text{Col}_8(\del_2^C )]\right)\,.\]
It then follows that 
\[\mathrm{Col}_{10}(-\theta_1(3) M_1(3))=\mathrm{Col}_{10}(\del^C_2\theta_2(3))\,.\]

In order to verify equalities in columns 7-9, we need the following observations: 

\medskip
\noindent {\it{Observation 1.}}
The following equalities are satisfied:
\[
\frac 92 f_{xx}\alpha_x+ \frac 92 f_{xy}\alpha_y+\frac 92 f_{xz}\alpha_z=\frac{9}{d'^2} \begin{bmatrix}
2x\delta-(d-1) f_x\Delta_{xx}\\ y\delta-(d-1) f_x\Delta_{xy}\\
z\delta-(d-1) f_x \Delta_{xz}\\-(d-1) f_x\Delta_{yy}\\ -(d-1) f_x \Delta_{yz}\\-(d-1) f_x\Delta_{zz}
\end{bmatrix}=\frac{9\delta}{d'^2} \begin{bmatrix} 2x\\y\\z\\0\\0\\0\end{bmatrix}-\frac{9\,f_x}{d'} 
\begin{bmatrix} \Delta_{xx} \\
\Delta_{xy} \\
 \Delta_{xz} \\
 \Delta_{yy} \\
\Delta_{yz} \\
 \Delta_{zz} 
\end{bmatrix}\,.
\]
Similarly, one can show 
\[
\frac{9}{2} f_{xy}\alpha_x+\frac{9}{2} f_{yy}\alpha_y+\frac{9}{2} f_{yz}\alpha_z=\frac{9\delta}{d'^2 } \begin{bmatrix} 0\\x\\0\\2y\\z\\0\end{bmatrix}-\frac{9\,f_y}{d'} 
\begin{bmatrix} \Delta_{xx} \\
\Delta_{xy} \\
 \Delta_{xz} \\
 \Delta_{yy} \\
\Delta_{yz} \\
 \Delta_{zz} 
\end{bmatrix} 
\]
\[
\frac{9}{2} f_{xz}\alpha_x+\frac{9}{2} f_{yz}\alpha_y+\frac{9}{2} f_{zz}\alpha_z=\frac{9\delta}{d'^2 } \begin{bmatrix} 0\\0\\x\\0\\y\\2z\end{bmatrix}-\frac{9\,f_z}{d'} 
\begin{bmatrix} \Delta_{xx} \\
\Delta_{xy} \\
 \Delta_{xz} \\
 \Delta_{yy} \\
\Delta_{yz} \\
 \Delta_{zz} 
\end{bmatrix}\,.
\]

\medskip
\noindent {\it{Observation 2.}} 
One has an equality 
\[ 
\frac{9}{2} f_{xx} \text{Col}_9 (\del_2^C )+\frac{9}{2} f_{xy} \text{Col}_{10} (\del_2^C )+\frac{9}{2} f_{xz} \text{Col}_{11} (\del_2^C )=\frac{9\delta}{d'^2 }
\begin{bmatrix}0\\ 
0\\ 
0\\ 
 2x\\y\\z\\0\\0\\0
\end{bmatrix}-
\dfrac{9}{d'}
\begin{bmatrix} \frac{\delta}{2}\\ 0 \\0\\ f_x\Delta_{xx} \\
f_x\Delta_{xy} \\
 f_x\Delta_{xz} \\
 f_x\Delta_{yy} \\
f_x\Delta_{yz} \\
 f_x\Delta_{zz}
\end{bmatrix}\,.
\]

Using the observations above, we conclude that the seventh column of $-\theta_1(3)M_1(3)$ can be written as 
\[
-\frac{\delta_x}{d'^2}\text{Col}_1(\del_2^C )-\frac{3\delta_x}{d'^2(d-2)} \text{Col}_5(\del_2^C )
+\frac{9}{2}\left( f_{xx} \text{Col}_9 (\del_2^C )
+ f_{xy} \text{Col}_{10} (\del_2^C )+ f_{xz} \text{Col}_{11} (\del_2^C )\right)\,,
\]
which is the same as the seventh column of $\del^C_2\theta_2(3)$.
 Similar arguments take care of the eighth and ninth columns. 
\end{proof}

%%%%%%%%%%%%%%%%%%%
\subsection{The third square commutes} \label{3rd lift for 3rd ord}

\begin{proposition} 
The following equality is satisfied: $-\theta_2(3)M_2(3)=\del_3^C\theta_3(3).$
\end{proposition} 

\begin{proof} 
Let $K$ and $P$ denote the $3\times 3$ matrices defined by 
\[
K\coloneqq\Delta D_2 \quad\text{and}\quad P\coloneqq D_1\del_1\,.
\]
Let $\tau$ be the $3\times 3$ matrix defined by 
\begin{equation*}
\tau\coloneqq \begin{bmatrix} \delta_x\\ \delta_y\\ \delta_z \end{bmatrix} \begin{bmatrix} x& y&z\\ \end{bmatrix}.
\end{equation*}
The matrix $-\theta_2(3)M_2(3)$ is of the form:
{\footnotesize
\setlength{\arraycolsep}{2.5pt}
\renewcommand{\arraystretch}{1.5}
\medmuskip=1mu
\[
\left[ \begin{array}{r| rrr| rrr| rrr}
0&0\phantom{oar}&0\phantom{oaraa}&0\phantom{o}&\frac{(d-2)\,}{d'^2} x\delta&\frac{(d-2)\,}{d'^2} y\delta\phantom{pr}&\frac{(d-2)\,}{d'^2} z\delta& \frac{H(\delta)}{d'^2}\phantom{par}&\frac{H_{zx}(\delta)}{d'^2}\phantom{parta}& \frac{H_{xy}(\delta)}{d'^2}\\ \hline
0&{}\phantom{oar}&{}\phantom{oar}&{}\phantom{o}&{}&{}&{}&{} &{}&{}\\ 
0&{}&\frac{d^2-1}{4} \;P&{}&{}&\frac{(d+1)}{2} \;K &{}&{}&\frac{3(d+1)\delta}{2d'^2} I_{3\times 3}+\frac{1}{d'^2}\;\tau&{}\\
0&{}&{}&{}&{}&{}&{}&{}&{}&{}\\ \hline
0&0\phantom{par}&0\phantom{partt}&0\phantom{p}&\frac{3 x\delta}{d'^2}&\frac{3 y\delta}{d'^2}&\frac{3 z\delta}{d'^2}&\frac{3 H(\delta)}{d'^2 (d-2)}&\frac{3 H_{zx}(\delta)}{d'^2 (d-2)}\phantom{par}&\frac{3 H_{xy}(\delta)}{d'^2 (d-2)}\\ \hline
0&0\phantom{oar}&0\phantom{oartt}&0\phantom{p}&0\phantom{par}&0\phantom{par}&0\phantom{oar}&{}&{}&{}\\ 
0&0\phantom{oar}&0\phantom{oartt}&0\phantom{p}&0\phantom{par}&0\phantom{par}&0\phantom{par}&{}&\frac{3}{d'^2(d-2) }\,\tau\phantom{parta}&{}\\ 
0&0\phantom{oar}&0\phantom{oartt}&0\phantom{p}&0\phantom{par}&0\phantom{par}&0\phantom{par}&{}&{}&{}\\ \hline
0&0\phantom{oar}&0\phantom{oattr}&0\phantom{p}&{}&{}&{}&{}&{}& {}\\ 
0&0\phantom{oar}&0\phantom{oartt}&0\phantom{p}&{}&\frac{-3d'}{2}\,P &{}&{}& \frac{-9}{2}\,K\phantom{ppart}& {}\\
0&0\phantom{oar}&0\phantom{oartt}&0\phantom{p}&{}&{}&{}&{}&{}& {}\\ 
\end{array}\right].
\]}\\
Using the identities in \cref{appendix}, 
it is a direct check, similar to those in \cref{first lift for 3rd ord,2nd lift for 3rd ord}, that $\del_3^C\theta_3(3)$ agrees with the matrix above. Alternatively, expressing each matrix as a block matrix one can use the identities from \cref{a:matrixidentities}. 
 \end{proof}

\begin{ack}
    We are grateful to the referee for the careful reading of the paper and the excellent comments and suggestions. 
\end{ack}
%%%%%%%%%%%%%%%%%%%%%%%%%%%%%%%%%%
\bibliographystyle{amsplain}
\bibliography{diffops}
%%%%%%%%%%%%%%%%%%%%%%%%%%%%%%%%%%
%%%%%%%%%%%%%%%%%%%%%%%%%%%%%%%%%%
%%%%%%%%%%%%%%%%%%%%%%%%%%%%%%%%%%
\end{document}